\documentclass[a4paper,11pt,reqno]{amsart}

\usepackage{lmodern}

\usepackage[utf8]{inputenc} 
\usepackage[T1]{fontenc}

\usepackage[english]{babel}

\usepackage{amsfonts}
\usepackage{amssymb}
\usepackage{amscd}
\usepackage{amsthm}
\usepackage{a4wide}
\usepackage{mathrsfs}
\usepackage{amsmath}
\usepackage{amssymb}
\usepackage{amsthm}
\usepackage{textcomp}
\usepackage{graphicx}
\usepackage{enumerate}
\usepackage{mathrsfs}
\usepackage{frcursive}
\usepackage{tikz}
\usepackage[cyr]{aeguill}
\usepackage{xspace}
\usepackage{hyperref}
\usepackage{appendix}
\usepackage{esint}
\usepackage{cancel}
\usepackage{calc}
\usepackage[colorinlistoftodos,prependcaption]{todonotes}

\newtheorem{defn}{Definition}[section]
\newtheorem{lemma}[defn]{Lemma}
\newtheorem{prop}[defn]{Proposition}
\newtheorem{theo}[defn]{Theorem}
\newtheorem{coro}[defn]{Corollary}

\newtheorem{claim}{Claim}
\newtheorem{rk}[defn]{Remark}

\newtheorem{note}[defn]{Note}
\newtheorem{exmp}[defn]{Example}

\def\RR{\mathbb{R}}
\def\Ric{\mathop{\rm Ric}\nolimits}

\def\Rm{\mathop{\rm Rm}\nolimits}
\def\tr{\mathop{\rm tr}\nolimits}

\def\vol{\mathop{\rm vol}\nolimits}

\def\vol{\mathop{\rm Vol}\nolimits}

\def\div{\mathop{\rm div}\nolimits}
\def\ima{\mathop{\rm Im}\nolimits}

\def\AVR{\mathop{\rm AVR}\nolimits}

\def\Li{\mathop{\rm \mathscr{L}}\nolimits}

\def\Ric{\mathop{\rm Ric}\nolimits}

\def\Rm{\mathop{\rm Rm}\nolimits}
\def\tr{\mathop{\rm tr}\nolimits}

\def\vol{\mathop{\rm vol}\nolimits}

\def\vol{\mathop{\rm Vol}\nolimits}

\def\div{\mathop{\rm div}\nolimits}
\def\ima{\mathop{\rm Im}\nolimits}

\def\AVR{\mathop{\rm AVR}\nolimits}

\def\Li{\mathop{\rm \mathscr{L}}\nolimits}

\def\R{\mathop{\rm R}\nolimits}

\def\ima{\mathop{\rm im}\nolimits}

\newsavebox\CBox
\newcommand\hcancel[2][0.5pt]{%
	\ifmmode\sbox\CBox{$#2$}\else\sbox\CBox{#2}\fi%
	\makebox[0pt][l]{\usebox\CBox}%
	\rule[0.5\ht\CBox-#1/2]{\wd\CBox}{#1}}

\numberwithin{equation}{section}

\begin{document}
	\title{A \L{}ojasiewicz inequality for ALE metrics}
	\date{today}
\author{Alix Deruelle}
\address{Institut de math\'ematiques de Jussieu, 4, place Jussieu, Boite Courrier 247 - 75252 Paris}
\email{alix.deruelle@imj-prg.fr}
\author{Tristan Ozuch}
\address{École Normale Supérieure, PSL Université, 45 rue d'Ulm - 75005 Paris}
\email{Tristan.Ozuch-Meersseman@ens.fr}
	\maketitle

	\begin{abstract}
		We introduce a new functional inspired by Perelman's $\lambda$-functional adapted to the asymptotically locally Euclidean (ALE) setting and denoted $\lambda_{\operatorname{ALE}}$. Its expression includes a boundary term which turns out to be the ADM-mass. We prove that $\lambda_{\operatorname{ALE}}$ is defined and analytic on convenient neighborhoods of Ricci-flat ALE metrics and we show that it is monotonic along the Ricci flow. This for example lets us establish that small perturbations of integrable and stable Ricci-flat ALE metrics with nonnegative scalar curvature have nonnegative mass. We then introduce a general scheme of proof for a \L{}ojasiewicz-Simon inequality on non-compact manifolds and prove that it applies to $\lambda_{\operatorname{ALE}}$ around Ricci-flat metrics. We moreover obtain an optimal weighted \L{}ojasiewicz exponent for metrics with integrable Ricci-flat deformations. 
	\end{abstract}
	\markboth{Alix Deruelle and Tristan Ozuch}{A \L{}ojasiewicz inequality for ALE metrics}
	\tableofcontents
	
	\section*{Introduction}
	
	The understanding of Ricci-flat metrics is a classical issue in Riemannian geometry. When they are non-compact, these metrics have at most Euclidean volume growth by Bishop-Gromov inequality and those which exactly have Euclidean volume growth are asymptotically conical. In dimension $4$, this implies that they are Asymptotically Locally Euclidean (ALE), that is they are asymptotic to some quotient of Euclidean space, see Definition \ref{defn-ALE}. 
	
	The class of ALE Ricci-flat metrics models the formation of singularities in various noncollapsed situations, in particular for spaces with Ricci curvature bounds \cite{and,Ban-Kas-Nak}. Moreover, these metrics model potential singularities of a $4$-dimensional Ricci flow with bounded scalar curvature \cite{bz}. Even more recently however, it has been shown that Ricci flows on a closed Riemannian manifold of dimension less than $8$ with bounded scalar curvature exist for all time: \cite{Buz-DiM}. Finally, ALE Ricci flat metrics actually appear as finite-time blow-up limits of some Ricci flows \cite{app-EH}. Their stability therefore becomes a crucial question for the Ricci flow.
	\subsection*{Adapting Perelman's $\lambda$-functional to the ALE setting}
	In \cite{Per-Ent}, Perelman introduced three functionals denoted $\lambda$, $\mu$ and $\nu$ of which the Ricci flow is the gradient flow. These functionals were the core of his spectacular proofs and revolutionized the understanding of the formation of singularities of the Ricci flow. 
		In particular, recall that if $(M^n,g)$ is a closed smooth Riemannian manifold, Perelman's energy, denoted by $\lambda(g)$, is defined as follows:
\begin{equation}
\lambda(g):=\inf_{\|\varphi\|_{L^2}=1}\int_M4|\nabla^g\varphi|^2_g+\R_g\varphi^2\,d\mu_g.\label{Per-Def-Lam}
\end{equation}
In other terms, $\lambda(g)$ is the bottom of the spectrum of the Schr\"odinger operator $-4\Delta_g+\R_g$. Then Perelman showed that $\lambda$ is monotone non-decreasing along the Ricci flow and it is constant on Ricci-flat metrics only. 

Let us notice some major inconvenients to use the same definition (\ref{Per-Def-Lam}) in a non-compact setting : if $(M^n,g)$ is an ALE metric in a neighborhood (with respect to some natural topology designed on polynomially decaying tensors at a certain rate) of a given ALE Ricci-flat metric then it can be shown that $\lambda(g)=0$. In other words, the $\lambda$-functional in its usual $L^2$-constrained form is not well-suited because of the lack of non trivial minimizers. 
	
	Our goal here is twofold. On the one hand, we aim at defining a functional on suitable neighborhoods of any ALE Ricci-flat metrics which detects Ricci-flat metrics only. On the other hand, we wish to define an adequate notion of linear stability for an ALE Ricci-flat metric tied to our functional in order to study the relation between linear stability and dynamical stability along the Ricci flow. The second goal will be addressed in a forthcoming paper.
	
	Our work relies on that of Haslhofer \cite{Has-Per-Fct} who introduced a functional which we denote by $\lambda_{\operatorname{ALE}}^0$ and where the minimization in (\ref{Per-Def-Lam}) takes place among test functions $\varphi$ such that $\varphi-1$ is compactly supported. It is a convenient functional when the scalar curvature is nonnegative and integrable: in \cite{Has-Per-Fct}, $\lambda_{\operatorname{ALE}}^0(g)$ is compared to the ADM mass $m_{\operatorname{ADM}}(g)$ of $g$ in order to give a quantitative version of the positive mass theorem on asymptotically flat manifolds (satisfying additional assumptions). 
	Despite its good properties with respect to the Ricci flow or its link to the ADM mass, the functional $\lambda_{\operatorname{ALE}}^0$ is only defined on a suitable neighborhood of metrics of a given ALE Ricci-flat metric $(N^n,g_b)$ whose scalar curvatures are either integrable or decay sufficiently fast at infinity. Such sets of metrics are not closed with respect to the topology induced by H\"older spaces modeled on polynomially decaying tensors with rate $\tau>\frac{n-2}{2}$. This classical restriction on the rate ensures that any deformation of the metric has its gradient lying in $L^2$. This observation is a major drawback to establish finer properties of the functional $\lambda_{\operatorname{ALE}}^0$ such as a \L ojasiewicz inequality we discuss below. 
	
		To remedy these issues, we refine the definition of $\lambda_{\operatorname{ALE}}^0$ by substracting the ADM mass. It gives in turn a new functional called $\lambda_{\operatorname{ALE}}$ and formally defined by:
		\begin{equation}
	\lambda_{\operatorname{ALE}}(g):= \lambda_{\operatorname{ALE}}^0(g)-m_{\operatorname{ADM}}(g).\label{formal-intro-lambda}
	\end{equation}
	At first sight, the functional $\lambda_{\operatorname{ALE}}$ still seems to require the integrability of the scalar curvature to make sense of each term: it turns out that (\ref{formal-intro-lambda}) can be reinterpreted as the limit of the difference of two a priori divergent integrals: see Section \ref{extension tilde lambda} for a precise statement. Again, in case a metric $g$ has integrable scalar curvature and lies in a neighborhood of an ALE Ricci-flat metric, the formula (\ref{formal-intro-lambda}) makes sense and lets one to compute $\lambda_{\operatorname{ALE}}$ more explicitly. We show that this functional fulfills our specifications: it is defined on a whole neighborhood of an ALE Ricci-flat metric whether its ADM mass is finite or not. Furthermore, it is analytic, its gradient vanishes on Ricci-flat ALE metrics only and the Ricci flow is moreover its gradient flow: see Section \ref{extension tilde lambda} for rigorous statements and proofs. 
	
	The second variation of $\lambda_{\operatorname{ALE}}$ at an ALE Ricci-flat metric $(N^n,g_b)$ along divergence-free variations is half the Lichnerowicz operator $L_{g_b}:=\Delta_{g_b}+2\Rm(g_b)\ast$, i.e. if $h$ is a smooth compactly supported $2$-tensor on $N$,
	\begin{equation}
	\delta^2_{g_b}\lambda(h,h)=\frac{1}{2}\left<L_{g_b}h,h\right>_{L^2},\quad \div_{g_b}h=0.
	\end{equation}
	This fact alone strongly suggests that the linear stability of an ALE Ricci-flat metric $(N^n,g_b)$ should be defined in terms of the non-positivity of the associated Lichnerowicz operator $L_{g_b}$ restricted to divergence-free variations: this guess is not new and was investigated in the setting of Ricci-flat cones by Haslhofer-Hall-Siepmann \cite{Hal-Has-Sie}. Linear stability actually gives some non trivial local information in the integrable case, i.e. in the case where the space of ALE Ricci-flat metrics in the neighborhood of a fixed ALE Ricci-flat metric is a smooth finite-dimensional manifold: see Proposition \ref{local maximum stable integrable} for a formal statement.
\begin{prop}\label{prop-sta-int}
Any ALE Ricci-flat metric $(N^n,g_b)$ which is locally stable and integrable is a local maximum for the functional $\lambda_{\operatorname{ALE}}$.
\end{prop}
The statement of Proposition \ref{prop-sta-int} echoes \cite{Has-Sta}: the proof is however technically different.

We also obtain in Section \ref{sec-covid-mass} new quantitative positive mass theorems in the same spirit of \cite{Has-Per-Fct} for small metric perturbations of \emph{stable} Ricci-flat ALE metrics: see Corollary \ref{local positive mass}. Note that the positive mass theorem generally does not hold on ALE manifolds \cite{LeB-Counter-Mass}. 

The positive mass theorem however holds for spin ALE manifolds, see \cite{nak}. Pushing Nakajima's estimate further similarly to \cite{Has-Per-Fct} in the asymptotically Euclidean setting, we prove the following \emph{global} property satisfied by the functional $\lambda_{\operatorname{ALE}}$ on spin manifolds.
	\begin{prop}[Proposition \ref{prop-spin-def-local}]
		Let $(N^4,g)$ be an ALE metric of order $\tau>1 = \frac{4}{2}-1$ on a spin manifold asymptotic to $\mathbb{R}^4\slash\Gamma$ for $\Gamma\subset SU(2)$.
		Assume the scalar curvature $\R_g$ is integrable and non-negative. Then, we have $$\lambda_{\operatorname{ALE}}(g)\leq 0,$$
		that is
		$$m_{\operatorname{ADM}}(g)\geqslant\lambda_{\operatorname{ALE}}^0(g)\geqslant 0,$$
		with equality if and only if $(N^4,g)$ is one of Kronheimer's gravitational instantons \cite{kro}.
	\end{prop}
	These well-known gravitational instantons are therefore the (only) maximizers of $\lambda_{ALE}$ with non negative and integrable scalar curvature.

	\subsection*{A \L{}ojasiewicz inequality for $\lambda_{\operatorname{ALE}}$}
	A certain number of difficulties arises when it comes to study the dynamical stability of ALE Ricci-flat metrics along the Ricci flow.
	A first obstacle is the presence of a non-trivial kernel of the Lichnerowicz operator. This issue already occurs in the case of a closed Ricci-flat metric. The non-compactness of the underlying space is an additional source of trouble since $0$ is not isolated in the spectrum of the linearized operator. This fact explains a polynomial-in-time convergence instead of an exponential rate in the case of an integrable Ricci-flat metric on a closed manifold, as it was demonstrated in \cite{Has-Sta}. 
	
	One tool that has been quite popular to study the stability of fixed points of  geometric evolution equations is the notion of \L{}ojasiewicz-Simon inequalities. Its name comes from both the classical work of \L{}ojasiewicz \cite{loj} on finite dimensional dynamical systems of gradient type and that of L. Simon \cite{sim} who extended systematically these inequalities to functionals defined on infinite dimensional spaces. The main geometric applications obtained in \cite{sim} concern the uniqueness of tangent cones of isolated singularities of minimal surfaces in Euclidean space together with the uniqueness of tangent maps of minimizing harmonic maps with values into an analytic closed Riemannian manifold. These geometric equations have the advantage to be strongly elliptic. Notice that all these results do not hold true if one drops the assumption on the analyticity of the data under consideration.

	 \L{}ojasiewicz-Simon inequalities have been extensively used these last years in the context of mean curvature flow, especially to prove the uniqueness of blow-ups \cite{Col-Min-Uni-MCF}. 
	
	In the compact setting, \L{}ojasiewicz inequalities have been proved for Perelman's $\lambda$-functional in the neighborhood of \emph{compact} Ricci-flat metrics in \cite{Has-Sta} in the integrable case, and in \cite{Has-Mul} in the general case, see also \cite{Kro-Sta-Ins} for Ricci solitons. They have been applied to characterize the stability and instability of the Ricci flow at compact Ricci-flat metrics.

	Our main application is to prove a similar result for Ricci-flat ALE metrics. We provide a general scheme of proof for \L{}ojasiewicz inequalities on non-compact manifolds based on the theory of elliptic operators between weighted Hölder spaces.
	
	We first show how such an inequality can be proved "by hand" in the integrable and stable situation, but we realize that this only leads to an actual \L{}ojasiewicz inequality in dimensions greater than or equal to $5$. We introduce a general scheme of proof, based on that of \cite{Col-Min-Ein-Tan-Con}, for weighted \L{}ojasiewicz inequalities for ALE metrics which holds without integrability or stability assumption in dimensions greater than or equal to $4$ and we provide an optimal exponent in the integrable case. By weighted \L{}ojasiewicz inequalities, we mean that the norm of the gradient of the corresponding functional is a weighted $L^2$-norm. More specifically here, we consider the space $L^2_{\frac{n}{2}+1}$ which, roughly speaking, is the space of tensors $T$ such that $r\cdot T$ belongs to $L^2$, $r$ being the distance from a fixed point: see Definition \ref{def-weighted-sobolev-norms} for a formal definition. The necessity of using such weighted norms rather than using the usual $L^2$ norm is explained below. 
	
Our main result is that the functional $\lambda_{\operatorname{ALE}}$ satisfies an $L^2_{\frac{n}{2}+1}$-\L ojasiewicz inequality in a neighborhood of any ALE Ricci-flat metric with respect to the topology of weighted H\"older spaces $C^{2,\alpha}_{\tau}$, $\alpha\in(0,1)$, with polynomial decay of rate $\tau\in (\frac{n-2}{2},n-2)$. 
	\begin{theo}\label{dream-thm-loja-intro}
		Let $(N^n,g_b)$ be an ALE Ricci-flat manifold of dimension $n\geq 4$. Let $\alpha\in(0,1)$ and $\tau\in(\frac{n-2}{2},n-2)$. Then there exist a neighborhood $B_{C^{2,\alpha}_{\tau}}(g_b,\varepsilon)$ of $g_b$, a constant $C>0$ and $\theta\in (0,1)$ such that for any metric $g\in B_{C^{2,\alpha}_{\tau}}(g_b,\varepsilon)$, we have the following $L^2_{\frac{n}{2}+1}$-\L ojasiewicz inequality,
		\begin{equation}
		|\lambda_{\operatorname{ALE}}(g)|^{2-\theta}\leq C\|\nabla \lambda_{\operatorname{ALE}}(g)\|_{L^2_{\frac{n}{2}+1}}^{2}.\label{loj-ineq-lambda-ALE}
		\end{equation}
		Moreover, if $(N^n,g_b)$ has integrable infinitesimal Ricci-flat deformations, then $\theta=1$.
		
		In particular, if $n\geq 5$, one has the following $L^2$-\L ojasiewicz inequality for integrable Ricci-flat ALE metrics: if $\tau\in(\frac{n}{2},n-2)$ then for any $0<\delta<\frac{2\tau-(n-2)}{2\tau-(n-4)}$, there exists $C>0$ such that for all $g\in B_{C^{2,\alpha}_\tau}(g_b,\epsilon)$, 
				$$ |\lambda_{\operatorname{ALE}}(g)|^{2-\theta_{L^2}}\leq C \|\nabla \lambda_{\operatorname{ALE}}(g)\|_{L^2}^{2}, \quad\theta_{L^2}:=2-\frac{1}{\delta}.$$

	\end{theo}
	Here $\nabla \lambda_{\operatorname{ALE}}$ denotes the gradient of $\lambda_{\operatorname{ALE}}$ in the $L^2$ sense.
	
	The fact that our spaces are non-compact induces quite a lot of new difficulties. In particular, the spectrum of the Lichnerowicz operator is not discrete anymore and $0$ belongs to the essential spectrum. This explains the need of considering weighted Sobolev spaces different from $L^2$ for which the differential of $\nabla\lambda_{\operatorname{ALE}}$ at a Ricci-flat ALE metric is Fredholm. We underline the fact that while Theorem \ref{dream-thm-loja-intro} gives an optimal $L^2_{\frac{n}{2}+1}$-\L ojasiewicz inequality, we cannot reach the usual optimal $L^2$-\L{}ojasiewicz exponent $\theta_{L^2} = 1$ (note that we approach it as $\tau$ is close to $n-2$ and the dimension tends to $+\infty$) in the integrable case: see also \cite[Theorem $2.1$]{Har-jen-Loja-Hil} for a proof of this fact in a general linear setting. This is consistent with the known fact that the DeTurck-Ricci flow only converges polynomially fast for perturbations of the Euclidean space: see for instance \cite{Sch-Sch-Sim} and \cite{app-scal}. Indeed, an exponent $\theta_{L^2} =1$ implies that the convergence is exponential.
	
	
	\begin{rk}
		Most of the analysis of this article should apply to the Ricci-flat asymptotically conical case. However, the other classical asymptotics in dimension $4$, namely ALF, ALG and ALH should require more involved arguments.
	\end{rk}
	
	In a forthcoming article, we use the $L^2_{\frac{n}{2}+1}$-\L{}ojasiewicz inequality \eqref{loj-ineq-lambda-ALE} to investigate the dynamical stability of Ricci-flat ALE metrics. 
\\


	\subsection*{Outline of paper}
	
	In Section \ref{sec-rel-ene-ALE}, we begin by recalling the basics of ALE Ricci-flat metrics including the definitions of polynomially weighted H\"older and Sobolev function spaces. Next, Section \ref{sec-mass-def} recalls the definition of the $\operatorname{ADM}$-mass of an ALE Ricci-flat metric and discusses various topologies on the set of nearby metrics which appear in the literature on the study of the mass on ALE metrics. Section \ref{sec-def-fun-lambda-0} introduces the $\lambda_{\operatorname{ALE}}^0$-functional and studies its basic properties in the previously mentioned topology: this is the content of Proposition \ref{existence propriete-wg}. Section \ref{sec-first-sec-var} is devoted to compute the first and second variations of $\lambda_{\operatorname{ALE}}^0$ which are summarized in Propositions \ref{first-var-prop} and  \ref{second-var-prop}. 
		
	In Section \ref{extension tilde lambda}, Proposition \ref{lambdaALE analytic} proves that substracting the ADM-mass to $\lambda_{\operatorname{ALE}}^0$ yields a much better-behaved and analytic functional that we denote $\lambda_{\operatorname{ALE}}$. The second variation of $\lambda_{\operatorname{ALE}}$ is computed in Proposition \ref{snd-var-gal-lambda}.  Moreover, the monotonicity of $\lambda_{\operatorname{ALE}}$ along the Ricci flow is established in Proposition \ref{prop-mono-lambda}. We also take the opportunity to define the linear stability of an ALE Ricci-flat metric: Lemma \ref{lemma-equiv-def-stable} gives two other  equivalent ways of defining the notion of linear stability for such a class of metrics.

	In Sections \ref{sec-ene-est-pot-fct} and \ref{fred-sec-prop}, we prove some technical results which are crucial for the rest of the paper. We first prove energy bounds for the potential function appearing in the definition of $\lambda_{\operatorname{ALE}}$ in Proposition \ref{prop-ene-pot-fct}. We then give the Fredholm properties of the Hessian of $\lambda_{\operatorname{ALE}}$, the Lichnerowicz Laplacian, in weighted H\"older and Sobolev spaces: this is the content of Proposition \ref{prop-lic-fred}.

	In Sections \ref{loj-sim-sec-int-case} and \ref{sec-loja-ineq-gal-case}, we prove \L{}ojasiewicz inequalities satisfied by the functional $\lambda_{\operatorname{ALE}}$. We first consider a na\"ive proof "by hand" in the integrable and stable case in dimension greater than or equal to $5$ in Section \ref{loj-sim-sec-int-case}. More specifically, Section \ref{sec-int-ric-fla} starts with a precise description of neighborhoods of integrable ALE Ricci-flat metrics: see Proposition \ref{gauge fixing ALE integrable}. Using it, we prove Proposition \ref{prop-sta-int} reformulated more accurately in Proposition \ref{local maximum stable integrable}. In Section \ref{naive-loja-sec}, a first attempt to prove a \L ojasiewicz inequality for integrable ALE Ricci-flat metrics is given by following an idea due to Haslhofer \cite{Has-Sta}: Proposition \ref{prop-baby-loja} yields an infinitesimal version of the \L ojasiewicz inequality for $\lambda_{\operatorname{ALE}}$ with a non-trivial \L ojasiewicz exponent $\theta$ in dimension greater than or equal to $5$. This can be interpreted as a preliminary step to the general case developed in Section \ref{sec-loja-ineq-gal-case}. 
	
	We then prove a general \L{}ojasiewicz inequality in Section \ref{sec-loja-ineq-gal-case} by extending the classical reduction to the finite-dimensional situation of \cite{sim}: more precisely, we adapt the concise version of \cite{Col-Min-Ein-Tan-Con} to this non-compact setting in Sections \ref{sec-gal-loja} and \ref{sec-prop-loja-ineq}. This requires the study of the Fredholm properties of the Lichnerowicz operator between weighted H\"older and Sobolev spaces. The proof of a general \L ojasiewicz inequality (\ref{loj-ineq-lambda-ALE}) needs a priori energy estimates which are taken care by Proposition \ref{prop-energy-est} in Section \ref{sec-proof-gal-loja} as required by Proposition \ref{Lojasiewicz ineq weighted}. Section \ref{sec-proof-gal-loja} ends with the proof of Theorem \ref{dream-thm-loja-intro} restated as Theorem \ref{theo-loja-ALE} (the general case) and Theorem \ref{theo-loja-int-opt} in the integrable case.
	
	In Section \ref{sec-covid-mass}, we give some connections between the mass $m_{\operatorname{ADM}}$ and the functional $\lambda_{\operatorname{ALE}}$. We deduce that the positive mass theorem holds in neighborhoods of Ricci-flat ALE metrics which are local maximizers of $\lambda_{\operatorname{ALE}}$. In particular, any compactly supported (or sufficiently decaying) deformation with nonnegative scalar curvature of a given integrable and locally stable Ricci-flat ALE metric is Ricci-flat: see Corollary \ref{local positive mass}. Proposition \ref{prop-spin-def-local} shows that ALE Ricci-flat metrics on spin manifolds are global maximizers of $\lambda_{\operatorname{ALE}}$.
	
	We then recall some basic results about real analytic maps between Banach spaces in Appendix \ref{app-A}, we prove a divergence-free gauge-fixing for ALE metrics in Appendix \ref{app-B}. Finally, we also recall the variations of some geometric quantities in Appendix \ref{app-C}.
	
	\subsection{Acknowledgements}
The first author is supported by grant ANR-17-CE40-0034 of the French National Research Agency ANR (Project CCEM).	
	\section{A relative energy for ALE Ricci-flat metrics}\label{sec-rel-ene-ALE}
	
	Let us introduce a non-compact version of Perelman's $\lambda$-functional and study its properties on ALE metrics.
	
	\subsection{Main definitions}~~\\
	
	We start by defining the class of metrics as well as the function spaces we will be interested in.
	
	\begin{defn}[Asymptotically locally Euclidean (ALE) manifolds]\label{defn-ALE}
		We will call a Riemannian manifold $(N^n,g)$ \emph{asymptotically locally Euclidean} (ALE) of order $\tau>0$ if the following holds: there exists a compact set $K\subset N$, a radius $R>0$, $\Gamma$ a subgroup of $SO(n)$ acting freely on $\mathbb{S}^{n-1}$ and a diffeomorphism $\Phi : (\RR^n\slash\Gamma )\backslash B_e(0,R)\mapsto N\backslash K$ such that, if we denote $g_e$ the Euclidean metric on $\mathbb{R}^n\slash\Gamma$, we have, for all $k\in \mathbb{N}$,
		$$ \rho^k\big|\nabla^{g_e,k}(\Phi^*g-g_e)\big|_e = O(\rho^{-\tau}),$$
		on $\big(\RR^n\slash\Gamma\big) \backslash B_e(0,R)$, where $\rho = d_e(.,0)$.
	\end{defn}
	
	If $g$ is an ALE metric on $N$, then we denote by $\rho_g$ any smooth positive extension of (the push-forward of) the radial distance on $\RR^n$, $\Phi_{\ast}\rho$. In particular, we will use the fact that the level sets $\{\rho_g=R\}$ of $\rho_g$ are smooth closed connected hypersurfaces for sufficiently large height $R$ constantly.
	
	We will study ALE metrics in a neighborhood of a Ricci-flat ALE metric. Let us start by defining this neighborhood thanks to weighted norms :
	
	\begin{defn}[Weighted Hölder norms for ALE metrics]\label{def-weighted-norms}
		Let $(N,g,p)$ be an ALE manifold of dimension $n$, $\beta>0$. For any tensor $s$, we define the following weighted $C^{k,\alpha}_\beta$-norm :
		$$\| s \|^g_{C^{k,\alpha}_\beta} := \sup_{N}\rho_g^\beta\Big( \sum_{i=0}^k\rho_g^{i}|\nabla^{g,i} s|_{g} + \rho_g^{k+\alpha}[\nabla^{g,k}s]_{C^{0,\alpha}}\Big).$$
	\end{defn}
	
	\begin{defn}[Weighted Sobolev norms for ALE metrics]\label{def-weighted-sobolev-norms}
		Let $\beta>0$, and $(N,g,p)$ an ALE manifold of dimension $n$. For any tensor $s$, we define the following weighted $L_\beta^2$-norm :
		$$\| s \|_{L^{2}_{\beta}} ^g:= \Big(\int_N |s|^2 \rho_{g}^{2\beta-n}d\mu_{g}\Big)^\frac{1}{2}.$$
		We moreover define the $H^k_{\beta}$-norm of $s$ as 
		$$\|s\|_{H^k_\beta}^g:= \sum_{i= 0}^k {\|\nabla^i s \|^g_{L^{2}_{\beta+i}}}.$$	
	\end{defn}
	
	\begin{rk}
		The usual $L^2$ space equals $L^2_\frac{n}{2}$ with the above definition. The intuition behind these norms is that $ \rho_{g}^{-\beta}\in C^{k,\alpha}_\beta$ and for all $\beta'>\beta$, $\rho_{g}^{-\beta'}\in H^k_\beta$. 
	\end{rk}
	
	\begin{rk}\label{sobolev embeddings}
		Moreover, we have the following embeddings: for $k$ large enough depending on the dimension: $$H^k_\beta\subset C^{0,\alpha}_\beta,$$ and for $\beta<\beta'$, we have $$C^{k,\alpha}_{\beta'}\subset H^k_\beta,$$ see \cite[Theorem $1.2$]{Bart-Mass}. Notice that in \cite{Bart-Mass}, $W^{k,2}_{\beta}$ coincides with our space $H^k_{-\beta}$ for $\beta\in\RR$.
	\end{rk}
	
	\begin{note}
		In the rest of this article, we will almost always work in a neighborhood of a fixed Ricci-flat metric with a given topology. Since the above definitions do not formally depend on the type of tensor, we will often abusively omit to mention these informations and simply denote these spaces $C^{k,\alpha}_\tau$ for instance.
	\end{note}
	Finally, we state a version of Hardy's inequality proved by Minerbe \cite{Min-Har-Ine} for Riemannian metrics $(N^n,g)$ with nonnegative Ricci curvature and maximal volume growth, i.e. $\AVR(g):=\lim_{r\rightarrow+\infty}r^{-n}\vol_gB_g(p,r)>0$ for some point $p$ and hence for all points by Bishop-Gromov Theorem:
	
	\begin{theo}[Minerbe's Hardy's inequality]\label{thm-min-har-inequ}
		Let $(N^n,g)$ be a complete Riemannian manifold such that $\Ric(g)\geq 0$ and $\AVR(g)>0$. Then for some point $p\in N$,
		\begin{equation*}
		\int_Nr_p^{-2}\varphi^2\,d\mu_g\leq C(n,\AVR(g))\int_N|\nabla^{g}\varphi|_{g}^2\,d\mu_{g},\quad \forall \varphi\in C_c^{\infty}(N),
		\end{equation*}
		where $r_p(x):=d_g(p,x)$ if $x\in N$.
	\end{theo}
	
	\subsection{The mass of an ALE metric}\label{sec-mass-def}~~\\
	
	Next, we define the class of ALE metrics (with a finite amount of regularity at infinity)  we will focus on at the beginning of this article:	
	let $(N^n,g_b)$ be an ALE Ricci-flat metric and let us consider for $\tau>\frac{n-2}{2}$ and $\alpha\in(0,1)$, the following space of metrics 
	\begin{equation}
	\mathcal{M}^{2,\alpha}_{\tau}(g_b):= \left\{\text{$g$ is a metric on $N$}\,|\,g-g_b\in C^{2,\alpha}_{\tau}(S^2T^*N)\,,\, \R_{g} = O(\rho_{g_b}^{-\tau'})\,\text{for some $\tau'>n$}\right\}.
	\end{equation}
	It turns out that this space is convex:
	\begin{lemma}\label{lemm-charac-M}
		We have the following characterization of the space $\mathcal{M}^{2,\alpha}_{\tau}(g_b)$: a metric $g\in \mathcal{M}^{2,\alpha}_{\tau}(g_b)$ if and only if $g-g_b\in C^{2,\alpha}_{\tau}(S^2T^*N)$ and the function $\div_{g_b}\div_{g_b}(g-g_b)-\Delta_{g_b}\tr_{g_b}(g-g_b)$ is in $C^0_{\tau'}(N)$ for some $\tau'>n$. 
		
		In particular, the space $\mathcal{M}^{2,\alpha}_{\tau}(g_b)$ is convex.
	\end{lemma}
	\begin{proof}
		If $g_1$ and $g_2$ belong to $\mathcal{M}^{2,\alpha}_{\tau}(g_b)$, it is straightforward that any convex combination $\lambda_1g_1+\lambda_2g_2$, $\lambda_i\geq 0$, $i=1,2$, $\lambda_1+\lambda_2=1$, is again a metric on $N$ such that $\lambda_1g_1+\lambda_2g_2-g_b\in C_{\tau}^{2,\alpha}(S^2T^*N)$. Now, by linearizing the scalar curvatures of the metrics $g_i=:g_b+h_i$, $i=1,2$ with respect to the metric $g_b$ thanks to [(\ref{lem-lin-equ-scal-first-var}, Lemma \ref{lem-lin-equ-Ric-first-var}], one gets:
		\begin{equation*}
		\R_{g_b+h_i}=\R_{g_b}+\div_{g_b}\div_{g_b}h_i-\Delta_{g_b}\tr_{g_b}h_i+O(\rho_{g_b}^{-2\tau-2}).
		\end{equation*}
		Since $\R_{g_b+h_i}=O(\rho_{g_b}^{-\tau_i'})$, for some $\tau'_i>n$, $i=1,2$, one concludes that 
		\begin{equation}
		\div_{g_b}\div_{g_b}h_i-\Delta_{g_b}\tr_{g_b}h_i=O\left(\rho_{g_b}^{-\min\{2\tau+2,\tau_i'\}}\right).\label{decay-lin-scal-cur}
		\end{equation}
		In particular, by summing the previous Taylor expansions of the scalar curvatures of $g_b+h_i$, $i=1,2$ together with (\ref{decay-lin-scal-cur}), one observes that:
		\begin{equation*}
		\begin{split}
		\R_{g_b+\lambda_1h_1+\lambda_2h_2}&=\sum_{i=1}^2\lambda_i\left(\div_{g_b}\div_{g_b}h_i-\Delta_{g_b}\tr_{g_b}h_i\right)+O(\rho_{g_b}^{-2\tau-2})\\
		&=O\left(\rho_{g_b}^{-\min\{2\tau+2,\tau_1',\tau'_2\}}\right).
		\end{split}
		\end{equation*}
		This proves the convexity of the space under consideration since $2\tau+2>n$.

	\end{proof}
	
\begin{rk}	
	We already see the importance of the assumption $\tau>\frac{n-2}{2}$ which will be crucial all along this paper: it ensures that the nonlinear terms in the expansion of the Ricci or scalar curvature around a given ALE Ricci-flat metric decay faster than the linear ones and are integrable. When dealing with improper integrals for instance, 
	only the linear terms have to be taken care of.
	\end{rk}
	We endow the space $\mathcal{M}^{2,\alpha}_{\tau}(g_b)$ with the distance induced by the norm $\|\cdot\|_{C^{2,\alpha}_{\tau}}$ as a convex subspace of $C^{2,\alpha}_{\tau}(S^2T^*N)$ and write $\mathcal{M}^{2,\alpha}_{\tau}(g_b,\varepsilon)$ for $\mathcal{M}^{2,\alpha}_{\tau}(g_b)\,\cap B_{C^{2,\alpha}_{\tau}}(0_{S^2T^*N},\varepsilon).$ 
	Our choice of notations follow that of Dai-Ma \cite{Dai-Ma-Mass}, Lee-Parker \cite{Lee-Parker} and Bartnik \cite{Bart-Mass} where they consider the more classical space of metrics 
	\begin{equation}
	\mathcal{M}_\tau:= \left\{\text{$g$ is a metric on $N$}\,|\, g-g_b\in C^{1,\alpha}_{\tau}(S^2T^*N)\,|\,\R_g \in L^1\right\},
	\end{equation}
	on which the mass of an ALE metric is well-defined:
	\begin{equation}
	m_{\operatorname{ADM}}(g):=\lim_{R\rightarrow+\infty}\int_{\{\rho_{g_b}=R\}}\left<\div_{g_b}(g-g_b)-\nabla^{g_b}\tr_{g_b}(g-g_b),\mathbf{n}\right>_{g_b}\,d\sigma_{g_b},\label{def-mass}
	\end{equation}
	where $\mathbf{n}$ denotes the outward unit normal of the closed smooth hypersurfaces $\{\rho_{g_b}=R\}$ for $R$ large.

	
	
	\subsection{Definition of the functional $\lambda_{\operatorname{ALE}}^0$ and its main properties}\label{sec-def-fun-lambda-0}~~\\

	We will use the renormalized Perelman's functional introduced by Haslhofer in \cite{Has-Per-Fct} to study ALE metrics that are close to Ricci-flat metrics. It can be defined in the following way for ALE metrics.
	
	\begin{defn}[$\lambda_{\operatorname{ALE}}^0$, a first renormalized Perelman's functional]
		Let $(N^n,g_b)$ be an ALE Ricci-flat metric and let $g\in \mathcal{M}^{2,\alpha}_{\tau}(g_b,\varepsilon)$. Define the $\mathcal{F}_{\operatorname{ALE}}$-energy by:
		\begin{eqnarray}
		\mathcal{F}_{\operatorname{ALE}}(w,g):=\int_N\big(4|\nabla^g w|_g^2 +\R_g w^2 \big)\,d\mu_g, 
		\end{eqnarray}
		where $w-1\in C^{\infty}_c(N)$, where $C^{\infty}_c(N)$ is the space of compactly supported smooth functions.
		The $\lambda_{\operatorname{ALE}}^0$-functional associated to the $\mathcal{F}_{\operatorname{ALE}}$-energy is:
		$$\lambda_{\operatorname{ALE}}^0(g) := \inf_w \mathcal{F}_{\operatorname{ALE}}(w,g),$$
		where the infimum is taken over functions $w:N\rightarrow \RR$ such that $w-1\in C_c^\infty(N)$.
	\end{defn}
	
	\begin{rk}
		By testing the infimum condition with $w \equiv 1$, we get the upper bound 
		\begin{equation}
		\lambda_{\operatorname{ALE}}^0(g)\leq \int_N\R_g\,d\mu_g.\label{borne sup lambda ALE}
		\end{equation}
		
		An assumption on the convergence rate $\tau$ and $\tau'$ in the definition of the space $\mathcal{M}^{2,\alpha}_{\tau}(g_b)$ is crucial to make sense of the functional $\lambda_{\operatorname{ALE}}^0(g)$: in particular, it ensures the integrability of the scalar curvature $\R_g$ of such an ALE metric $g$. The fact that $\lambda_{\operatorname{ALE}}^0(g)>-\infty$ is not trivial will be established in the following section: it depends on Hardy's inequality (Theorem \ref{thm-min-har-inequ}). See the proof of Proposition \ref{existence propriete-wg}.
	\end{rk}

	Let us now prove that the functional $\lambda_{\operatorname{ALE}}^0$ has nice properties in sufficiently small neighborhoods of Ricci-flat ALE metrics. Let us mention that according to \cite{Ban-Kas-Nak,Che-Tian-Ric-Fla}, any $n$-dimensional Ricci-flat ALE metrics is ALE of order $n$.

	\begin{prop}\label{existence propriete-wg} 
		Let $(N^n,g_b)$ be an ALE Ricci-flat metric asymptotic to $\RR^n\slash\Gamma$, for some finite subgroup $\Gamma$ of $SO(n)$ acting freely on $\mathbb{S}^{n-1}$. Let $\tau\in(\frac{n-2}{2},n-2)$ and $\alpha\in (0,1)$.
		
		Then, there exists some positive $\varepsilon$ such that for any metric $g$ in a neighborhood $\mathcal{M}^{2,\alpha}_{\tau}(g_b,\varepsilon)= \mathcal{M}^{2,\alpha}_{\tau}(g_b)\cap B_{C^{2,\alpha}_{\tau}}(g_b,\varepsilon)$ of $g_b$, the infimum defining the functional $\lambda_{\operatorname{ALE}}^0(g)$ is attained by the unique solution $w_g$ to the following equation,
		\begin{equation}
		\left\{
		\begin{aligned}
		-4\Delta_g w_g + \R_g w_g =0, \label{equ-criti-lambda}\\
		w_g-1 \in C^{2,\alpha}_{\tau}(N)\cap C^{1,\alpha}_{n-2}(N).
		\end{aligned}
		\right.
		\end{equation}
		
		Moreover, $w_g$ is positive on $N$ and we have the following expansion of $w_g$ at infinity :
		\begin{equation}
		w_g = 1 - \frac{\lambda_{\operatorname{ALE}}^0(g) |\Gamma|}{4(n-2)\vol{\mathbb{S}^{n-1}}}\frac{1}{\rho_{g_b}^{n-2}} +O(\rho_{g_b}^{-n+2-\gamma}),  \label{developpement w}
		\end{equation}
		for some positive $\gamma$ and where $|\Gamma|$ is the cardinal of $\Gamma$. 
		
		Next, we have the equalities :
		\begin{equation}
		\lambda_{\operatorname{ALE}}^0(g)=\int_N \big(4|\nabla^g w_g|_g^2 +\R_g w_g^2 \big)\,d\mu_g= \int_N \R_g w_g \,d\mu_g,\label{egalite lambda 1}
		\end{equation}
		and,
		\begin{equation}
		\lambda_{\operatorname{ALE}}^0(g) = \lim_{R\to \infty} 4\int_{\{\rho_{g_b} = R\}} \langle\nabla^gw_g,\mathbf{n}_{g_b}\rangle\, d\sigma_g,\label{egalite lambda 2}
		\end{equation}
		 where $\mathbf{n}_{g_b}$ denotes the outward unit normal of $\{\rho_{g_b}=R\}$.
		
		Finally, the map $g\in B_{C^{2,\alpha}_{\tau}}(g_b,\varepsilon)\rightarrow w_g-1\in C^{2,\alpha}_{\tau}(N)$ is analytic in the sense of Definition \ref{def-analytic}. \end{prop}
	
	\begin{proof}
		First of all, let us show that $\lambda_{\operatorname{ALE}}^0(g)$ is finite, i.e. $\lambda_{\operatorname{ALE}}^0(g)>-\infty$.
		
		Since $(N^n,g_b)$ is a Ricci-flat ALE metric $(N^n,g_b)$, Theorem \ref{thm-min-har-inequ} ensures that the following Hardy inequality holds true:
		\begin{equation}
		C_H\int_N\frac{\varphi^2}{\rho_{g_b}^2}d\mu_{g_b}\leq \int_N|\nabla^{g_b}\varphi|_{g_b}^2\,d\mu_{g_b},\quad \varphi\in C_c^{\infty}(N),\label{har-inequ}
		\end{equation}
		for some positive constant $C_H$ depending on $g_b$, the dimension $n$ and the base point $p\in N$ used in Definition \ref{def-weighted-norms} of $\rho_g$.
		Since the metrics $g$ and $g_b$ are equivalent, i.e. $C^{-1}g_b\leq g\leq Cg_b$ for some positive constant depending on the neighborhood $B_{C^{2,\alpha}_{\tau}}(g_b,\varepsilon)$, the same Hardy inequality holds with a positive constant $C_H(g_b)/2$ if $\varepsilon$ is chosen small enough. Moreover, (\ref{har-inequ}) is valid for functions $w$ on $N$ such that $w-1\in C_{\tau}^{2,\alpha}(N)$. This implies that:
		\begin{equation*}
		\begin{split}
		\int_N4|\nabla^gw|^2_g+\R_gw^2d\mu_g&=\int_N4|\nabla^g(w-1)|^2_g+\R_g(w-1+1)^2d\mu_g\\
		&\geq2C_H\int_N\frac{(w-1)^2}{\rho_g^2}d\mu_g-\varepsilon\int_N\frac{(w-1)^2}{\rho_g^2}d\mu_g-c\int_N|\R_g|d\mu_g\\
		&\geq-c\int_N|\R_g|d\mu_g,
		\end{split}
		\end{equation*}
		
		if $\varepsilon$ is chosen not greater than $2C_H$ and where $c$ is a universal positive constant that may vary from line to line. 
		This proves the finiteness of $\lambda_{\operatorname{ALE}}^0(g)$ together with the fact that the operator $-4\Delta_g+\R_g$ is non-negative and dominates $-\Delta_g$ in the $L^2$ sense, i.e. if $g\in B_{C^{2,\alpha}_{\tau}}(g_b,\varepsilon)$ then
		\begin{equation}
		\langle-4\Delta_g\varphi+\R_g\varphi,\varphi\rangle\geq c\|\nabla^g\varphi\|^2_{L^2},\quad \forall \varphi\in C_c^{\infty}(N).\label{dom-lap-per-ope}
		\end{equation}
		In particular, by density, inequality (\ref{dom-lap-per-ope}) holds for functions in $C^2_{\tau}(N)$ with $2\tau>n-2$. 
		
		\begin{claim}\label{claim-iso}
			The operator $-4\Delta_g+\R_g: C^{2,\alpha}_{\tau}(N)\rightarrow C^{0,\alpha}_{\tau+2}(N)$ is an isomorphism of Banach spaces for all $\alpha\in(0,1)$. Moreover, the map $g\in B_{C^{2,\alpha}_{\tau}}(g_b,\varepsilon)\mapsto (-4\Delta_g+\R_g)^{-1}\R_g \in C^{2,\alpha}_{\tau}(N) $ is analytic.
			
		\end{claim}
		\begin{proof}[Proof of Claim \ref{claim-iso}]
			Consider the map $\Psi:B_{C^{2,\alpha}_{\tau}}(g_b,\varepsilon)\times C^{2,\alpha}_{\tau}(N)\rightarrow C^{0,\alpha}_{\tau+2}(N)$ defined by $\Psi(g,v):=-4\Delta_gv+\R_gv$. The map $\Psi$ is analytic in the sense of Definition \ref{def-analytic}.
			
			According to [Theorem $8.3.6$ $(a)$, \cite{Joy-Book}], $\Delta_g: C^{2,\alpha}_{\tau}(N)\rightarrow C^{0,\alpha}_{\tau+2}(N)$ is an isomorphism of Banach spaces for all $\alpha\in(0,1)$. Fix $\alpha\in(0,1)$. Since $\R_g: C^{2,\alpha}_{\tau}(N)\rightarrow C^{0,\alpha}_{\tau+2}(N)$ is a compact operator, the operator $-4\Delta_g+\R_g: C^{2,\alpha}_{\tau}(N)\rightarrow C^{0,\alpha}_{\tau+2}(N)$ is a Fredholm operator of index $0$. In particular, it is an isomorphism if (and only if) it is injective. This in turn is ensured by (\ref{dom-lap-per-ope}) since $(N,g)$ has infinite volume.
			Therefore, the analytic version of the implicit function Theorem given by Lemma \ref{th fcts implicites} applied to the map $\Psi$ gives us the expected result.
		\end{proof}
		Now, let $\alpha\in(0,1)$ such that Claim \ref{claim-iso} holds: since $\R_g\in C^{0,\alpha}_{\tau+2}(N)$, there exists a unique solution $v_g\in C^{2,\alpha}_{\tau}(N)$ to $-4\Delta_gv_g+\R_gv_g=-\R_g$ and the map $g\in B_{C^{2,\alpha}_{\tau}}(g_b,\varepsilon)\rightarrow v_g\in C^{2,\alpha}_{\tau}(N)$ is analytic. In particular, if $w_g:=1+v_g$ then $-4\Delta_gw_g+\R_gw_g=0$. Let us show that this implies (\ref{egalite lambda 1}) and (\ref{egalite lambda 2}) by integrating by parts over sublevel sets $\{\rho_{g_b}\leq R\}$ of large radii $R$ whose boundary is $\{\rho_{g_b}= R\}$:
		\begin{equation*}
		\begin{split}
		\int_{\{\rho_{g_b}\leq R\}}4|\nabla^gw_g|^2+\R_gw_g^2\,d\mu_g=\,&\int_{\{\rho_{g_b}\leq R\}}-4\Delta_gw_g\cdot w_g+\R_gw_g\cdot w_g\,d\mu_g\\
		&+4\int_{\{\rho_{g_b}= R\}}\langle\nabla^gw_g,\mathbf{n}_{g_b}\rangle \cdot w_g\,d\sigma_g\\
		=\,&0+4\int_{\{\rho_{g_b}= R\}}\langle\nabla^gw_g,\mathbf{n}_{g_b}\rangle\, d\sigma_g +4\int_{\{\rho_{g_b}= R\}}\langle\nabla^gw_g,\mathbf{n}_{g_b}\rangle \cdot v_g\,d\sigma_g\\
		=\,&4\int_{\{\rho_{g_b}= R\}}\langle\nabla^gw_g,\mathbf{n}_{g_b}\rangle\,d\sigma_g+\textit{o}(1),
		\end{split}
		\end{equation*}
		as $R$ tends to $+\infty$. Similarly, by using (\ref{equ-criti-lambda}):
		\begin{equation*}
		\begin{split}
		\int_{\{\rho_{g_b}\leq R\}}\R_gw_g\,d\mu_g&=4\int_{\{\rho_{g_b}\leq R\}}\Delta_gw_g\,d\mu_g=4\int_{\{\rho_{g_b}= R\}}\langle\nabla^gw_g,\mathbf{n}_{g_b}\rangle \,d\sigma_g.
		\end{split}
		\end{equation*}
		Since $\R_g$ is integrable, taking a limit in the previous identity as $R$ tends to $+\infty$ is meaningful. To sum it up:
		\begin{equation}
		\int_{N}\R_gw_g\,d\mu_g=\lim_{R\rightarrow+\infty}4\int_{\{\rho_{g_b}= R\}}\langle\nabla^gw_g,\mathbf{n}_{g_b}\rangle \,d\sigma_g=\int_{N}4|\nabla^gw_g|^2+\R_gw_g^2\,d\mu_g.\label{equ-diff-form-lambda}
		\end{equation}
		
		Finally, to end the proof of (\ref{egalite lambda 1}) and (\ref{egalite lambda 2}), it suffices to show that:
		\begin{equation}
		\int_N4|\nabla^g(w_g+\varphi)|^2_g+\R_g(w_g+\varphi)^2\,d\mu_g\geq \int_N4|\nabla^gw_g|^2_g+\R_gw_g^2\,d\mu_g, \quad \forall \varphi\in C^{\infty}_{c}(N).\label{w_g-min-F-ALE}
		\end{equation}
		This amounts to proving that:
		\begin{equation}
		\int_N4|\nabla^g\varphi|^2_g+\R_g\varphi^2\,d\mu_g+2\int_N4\langle\nabla^gw_g,\nabla^g\varphi\rangle+\R_gw_g\varphi \,d\mu_g\geq 0,\quad\forall \varphi\in C^{\infty}_{c}(N).\label{w_g-min-F-ALE-bis}
		\end{equation}
		This is implied by (\ref{dom-lap-per-ope}) together with (\ref{equ-criti-lambda}) after an integration by parts on the second integral of the lefthand side of the previous inequality (\ref{w_g-min-F-ALE-bis}).
		
		We are left with proving the positivity of $w_g$ and the asymptotic expansion (\ref{developpement w}). 
		
		By Kato inequality, one can check that 
		\begin{eqnarray}
		\lambda_{\operatorname{ALE}}^0(g)=\mathcal{F}_{\operatorname{ALE}}(w_g,g)\geq \mathcal{F}_{\operatorname{ALE}}(|w_g|,g).\label{min-F-bis}
		\end{eqnarray}
		Notice that (\ref{w_g-min-F-ALE}) still holds for functions $\varphi\in H_c^1(N)$ (the completion of compactly supported functions for the $H^1$-norm), by taking the completion with respect to the norm $\varphi\rightarrow\|\nabla^g\varphi\|_{L^2}$. In particular, the function $|w_g|$ is a test function, i.e. $|w_g|-1\in H_c^1(N)$ and it is a minimizer of $\lambda_{\operatorname{ALE}}^0(g)$ by (\ref{min-F-bis}). As such, $|w_g|$ is a continuous weak solution to (\ref{equ-criti-lambda}). Elliptic regularity together with elliptic Schauder estimates imply that $|w_g|$ is a $C^{2,\alpha}_{loc}$-solution, so that $w_g$ has a sign and tends to $1$ at infinity, i.e. $w_g$ is nonnegative. Now, $w_g$ is positive by the strong maximum principle for parabolic equations. Indeed, if $$W_g(x,t):=\exp\left(\frac{\sup_N\R_g}{4} t\right)w_g(x),$$ for $x\in N$ and $t\in\RR$, then $W_g$ is a super-solution to the heat equation: 
		\begin{equation*}
		\left(\partial_t-\Delta_g\right)W_g\geq 0.
		\end{equation*}
		Therefore, the maximum principle leads to $W_g(x,t)\geq \int_NK(x,y,t)w_g(y)d\mu_g(y)>0$ for $t\geq 0$ where $K(x,y,t)$ denotes the positive heat kernel associated to $\Delta_g$.\\
		
		Let us prove the asymptotic estimate (\ref{developpement w}).\\
		
		Observe that if $w_g=1-c\rho_{g_b}^{2-n}+\textit{O}(\rho_{g_b}^{2-n-\gamma})$ for some $\gamma>0$ up to first order for some constant $c$ then (\ref{egalite lambda 2}) implies necessarily that 
		\begin{eqnarray}
		4c(n-2)\vol\left(\mathbb{S}^{n-1}/\Gamma\right)=\lim_{R\rightarrow+\infty}4\int_{\{\rho_{g_b}= R\}}\langle\nabla^gw_g,\mathbf{n}_{g_b}\rangle\,d\sigma_g=\lambda_{\operatorname{ALE}}^0(g).\label{identification-lambda-sol-inf}
		\end{eqnarray}
		
		Let us observe that $-4\Delta_gv_g=-\R_gw_g=\textit{O}(\rho_{g_b}^{-\tau'})$ with $\tau'>n$. Therefore, by [Theorem $8.3.6$ $(b)$, \cite{Joy-Book}], there exists a unique solution $u_g\in C^{1,\alpha}_{n-2}(N)$ to $-4\Delta_gu_g=-\R_gw_g$: here $\R_g$ is only assumed to lie in $C^0_{\tau'}$ for some $\tau'>n$ which does not ensure higher regularity on $u_g$ in $C^{2,\alpha}_{n-2}$. Moreover, it is shown that:
		\begin{equation*}
		\begin{split}
		u_g&=-\left(\frac{|\Gamma]}{4(n-2)\vol\mathbb{S}^{n-1}}\int_N\R_gw_g\,d\mu_g\right)\rho_{g_b}^{2-n}+\textit{O}(\rho_{g_b}^{2-n-\gamma})\\
		&=-\frac{\lambda_{\operatorname{ALE}}^0(g)|\Gamma]}{4(n-2)\vol\mathbb{S}^{n-1}}\rho_{g_b}^{2-n}+\textit{O}(\rho_{g_b}^{2-n-\gamma}),
		\end{split}
		\end{equation*}
		for some positive $\gamma$ where we used (\ref{identification-lambda-sol-inf}) together with the first equality of (\ref{equ-diff-form-lambda}) in the second line. To conclude, one gets that $v_g-u_g$ is a harmonic function on $N$ that converges to $0$ at infinity. These facts imply that $v_g=u_g$ by the maximum principle.
	\end{proof}
	\begin{rk}
		Under the assumptions of Proposition \ref{existence propriete-wg}, the usual $L^2$-constrained $\lambda$-functional, i.e. the bottom of the $L^2$-spectrum of the operator $-4\Delta_g+\R_g$ is vanishing because the mass can escape to infinity, and therefore this functional does not give any information: indeed, by (\ref{dom-lap-per-ope}), $\lambda_1(-4\Delta_g+\R_g)\geq c\lambda_1(-\Delta_g)=0$. One can then argue by contradiction as in \cite{Cheng-Yau} to show that if $\lambda_1(-4\Delta_g+\R_g)$ were positive then geodesic balls of large radii would grow faster than any polynomial. 
	\end{rk}
	
	The functional $\lambda_{\operatorname{ALE}}^0$ is moreover invariant by diffeomorphisms decaying at infinity.
	\begin{prop}\label{inv-diff-lambda}
		Let $(N^n,g_b)$ be an ALE Ricci-flat metric asymptotic to $\RR^n\slash\Gamma$, for some finite subgroup $\Gamma$ of $SO(n)$ acting freely on $\mathbb{S}^{n-1}$. Let $\alpha\in (0,1)$ and $\tau\in(\frac{n-2}{2},n-2)$.
		
		Let $g$ be a metric in $ \mathcal{M}^{2,\alpha}_{\tau}(g_b,\varepsilon)$ and let $w_g$ be the minimizer of $\lambda_{\operatorname{ALE}}^0(g)$ whose existence is ensured by Proposition \ref{existence propriete-wg}. Consider $\phi$ a diffeomorphism close to the identity in the $C^{3,\alpha}_{\tau-1}(TN)$ topology. Then, we have $\lambda_{\operatorname{ALE}}^0(\phi^*g) = \lambda_{\operatorname{ALE}}^0(g)$ and $w_{\phi^*g} = \phi^*w_g$.
	\end{prop}
	\begin{proof}
		Let $g=:g_b+h$ be a metric such that $g_b+h\in \mathcal{M}^{2,\alpha}_{\tau}(g_b)$, and consider $w_g$ the minimizer of $\lambda_{\operatorname{ALE}}^0(g)$. Let $\phi : N\to N$ be a diffeomorphism defined as $\phi :x \mapsto \exp_x^{g}(X(x))$ for a vector field $X\in C^{3,\alpha}_{\tau-1}(TN)$ close to $0_{TN}$. Consider $\phi^*g$ which is also a  metric on $N$ of order $\tau$ lying in a small neighborhood of $g_b$ in the $C^{2,\alpha}_{\tau}$ topology. Moreover, since $\R_{\phi^*g}= \phi^*\R_g$, we still have $\R_{\phi^*g}\in C^0_{\tau'}(N)$ for some $\tau'>n$. Then, for any $w$ such that $w-1$ is compactly supported, $\phi^*w-1$ is also compactly supported and we clearly have $\mathcal{F}_{\operatorname{ALE}}(\phi^*w,\phi^*g) = \mathcal{F}_{\operatorname{ALE}}(w,g)$ by the change of variables Theorem. Therefore we have $\lambda_{\operatorname{ALE}}^0(\phi^*g) = \lambda_{\operatorname{ALE}}^0(g)$, and finally, Proposition \ref{existence propriete-wg} ensures that $w_{\phi^*g} = \phi^*w_g$.
	\end{proof}
	We end this section by giving another sufficient condition to ensure the finiteness of $\lambda_{\operatorname{ALE}}^0$ together with its behavior under scalings of the metrics: 	
	\begin{lemma}\label{scaling lambdaALE}
		Let $(N^n,g)$ be a complete metric with $\R_g\in L^1(N)$ and non-negative scalar curvature. Then, $$0\leq \lambda_{\operatorname{ALE}}^0(g)\leq \|\R_g\|_{L^1}.$$
		
		Moreover, in case $(N^n,g)$ is a complete metric such that $\lambda_{\operatorname{ALE}}^0(g)$ is finite then for any $s>0$, we have $$\lambda_{\operatorname{ALE}}^0(sg) = s^{\frac{n-2}{2}}\lambda_{\operatorname{ALE}}^0(g).$$
	\end{lemma}
	\begin{proof}
		By testing the function $w\equiv1$ in the definition of $\lambda_{\operatorname{ALE}}^0(g)$, one gets the upper bound. Moreover, one has the straightforward inequality for any function $w$ such that $w-1\in C_c^{\infty}(N)$: $$0\leq \int_N|\nabla^g w|^2_g\,d\mu_g\leq \int_N|\nabla^g w|^2_g+\R_gw^2\,d\mu_g.$$ By considering the infimum over such functions $w$, one gets the expected lower bound on $\lambda_{\operatorname{ALE}}^0(g)$.
		
		For any smooth $w$ such that outside a compact set we have $w\equiv 1$,
		$\mathcal{F}_{\operatorname{ALE}}(w,sg) = s^{\frac{n}{2}- 1}\mathcal{F}_{\operatorname{ALE}}(w,g)$, because of the scaling behavior of the different operations: $|\nabla^{sg}f|_{sg}^2=s^{-1}|\nabla^{g}f|_{g}^2$, $\R_{sg}= s^{-1}\R_g$ and $d\mu_{sg} = s^{\frac{n}{2}}d\mu_g$.
		We therefore have $ \lambda_{\operatorname{ALE}}^0(sg) = s^{\frac{n}{2}- 1}\lambda_{\operatorname{ALE}}^0(g)$.
		
	\end{proof}

	\section{First and second variations of $\lambda_{\operatorname{ALE}}^0$}\label{sec-first-sec-var}
	In this section, we compute the first and second variations of the functional $\lambda_{\operatorname{ALE}}^0$ introduced in Section \ref{sec-rel-ene-ALE}. Before doing so, we define the notion of a potential function associated to a metric $g$ lying in a $C^{2,\alpha}_{\tau}$-neighborhood of an ALE Ricci-flat metric $g_b$.
	\begin{defn}
		For a metric $g$ in $B_{C^{2,\alpha}_{\tau}}(g_b,\varepsilon)$, with $\tau\in(\frac{n-2}{2},n-2)$, let us define the potential function associated to $g$ by $$f_g:=-2\ln w_g,$$
		where $w_g$ is defined as in Proposition \ref{existence propriete-wg} .
	\end{defn}
	Notice that $f_g$ is well-defined by the positivity of $w_g$ ensured by Proposition \ref{existence propriete-wg}. Moreover, we sum up the properties shared by $f_g$ in the next proposition which follow in a straightforward way from Proposition \ref{existence propriete-wg}:
	
	\begin{prop}\label{prop-pot-fct}
		Let $(N^n,g_b)$ be an ALE Ricci-flat metric asymptotic to $\RR^n\slash\Gamma$, for some finite subgroup $\Gamma$ of $SO(n)$ acting freely on $\mathbb{S}^{n-1}$. Let $\tau\in(\frac{n-2}{2},n-2)$ and $\alpha\in(0,1)$. 
		
		Then there exists some positive $\varepsilon$ such that $g\in B_{C^{2,\alpha}_{\tau}}(g_b,\varepsilon)\rightarrow f_g\in C^{2,\alpha}_{\tau}(N)$ is analytic and satisfies on $N$,
		\begin{equation}
		2\Delta_gf_g-|\nabla^gf_g|^2_g+\R_g =0. \label{equ-criti-lambda-pot}
				\end{equation}
		Moreover, the asymptotic expansion holds true for $g\in\mathcal{M}^{2,\alpha}_{\tau}(g_b,\varepsilon)$,
		\begin{equation}
		f_g=\frac{\lambda_{\operatorname{ALE}}^0(g) |\Gamma|}{2(n-2)\vol{\mathbb{S}^{n-1}}}\frac{1}{\rho_{g_b}^{n-2}}+\textit{O}(\rho_{g_b}^{-n+2-\gamma}),
\end{equation}
for some positive real number $\gamma$.

		Finally, if $g\in\mathcal{M}^{2,\alpha}_{\tau}(g_b,\varepsilon)$,
		\begin{equation}
		\begin{split}
		\lambda_{\operatorname{ALE}}^0(g) &= \int_N \big(|\nabla^g f_g|_g^2 +\R_g  \big)e^{-f_g}\,d\mu_g\\
		& =2\int_N\left(|\nabla^gf_g|^2_g-\Delta_gf_g\right)\,e^{-f_g}d\mu_g\\
		&= -\lim_{R\to \infty} 2\int_{\{\rho_{g_b} = R\}} \langle\nabla^gf_g,\mathbf{n}_{g_b}\rangle\, d\sigma_g. \label{egalite-lambda-2-bis}
		\end{split}
		\end{equation}
	\end{prop}
	
	Before stating the first variation of $\mathcal{F}_{\operatorname{ALE}}$ for arbitrary variations, we introduce several notions associated to a smooth metric measure space $(N^n,g,\nabla^gf)$ where $f$ is a given $C^1_{loc}$ function on $N$. The \textbf{weighted laplacian} of a tensor $T$ on $N$ denoted by $\Delta_fT$ is defined by:
	\begin{equation}
	\Delta_fT:=\Delta_gT-\nabla^g_{\nabla^gf}T,
	\end{equation}
	where $\Delta_g$ denotes the rough Laplacian associated to the Riemannian metric $g$.
	
	The \textbf{weighted divergence} of a $C^1_{loc}$ vector field $X$ on $N$ is defined by:
	\begin{equation}
	\div_fX:=\Delta_gX-g(\nabla^gf,X).\label{def-wei-div-vec}
	\end{equation}
	Finally, the \textbf{weighted divergence} of a $C^1_{loc}$ symmetric $2$-tensor $T$ on $N$ is defined by:
	\begin{equation}
	\div_fT:=\div_gT-T(\nabla^gf).\label{def-wei-div-sym}
	\end{equation}
	
	\begin{prop}[First variation of $\mathcal{F}_{\operatorname{ALE}}$]\label{first-var-prop} Let $(N^n,g_b)$ be an ALE Ricci-flat metric asymptotic to $\RR^n\slash\Gamma$, for some finite subgroup $\Gamma$ of $SO(n)$ acting freely on $\mathbb{S}^{n-1}$. Let $\tau\in(\frac{n-2}{2},n-2)$ and $\alpha\in(0,1)$. 
		
		The first variation of $\mathcal{F}_{\operatorname{ALE}}$ at a couple $(g,f)\in  \mathcal{M}^{2,\alpha}_{\tau}(g_b,\varepsilon)\times C^{2}_{\tau}(N)$ along directions $(h,\varphi)\in C^{2}_{\tau}(S^2T^*N)\times C^{2}_{\tau}(N)$ such that $g+h\in \mathcal{M}^{2,\alpha}_{\tau}(g_b,\varepsilon)$ is
		\begin{equation}
		\begin{split}
		\delta_{g,f}\mathcal{F}_{\operatorname{ALE}}(h,\varphi)&=-\int_N\langle h,\Ric(g)+\nabla^{g,2}f\rangle_g\, e^{-f}d\mu_g\\
		&+\int_N\left(2\Delta_gf-|\nabla^gf|^2_g+\R_g\right)\left(\frac{\tr_gh}{2}-\varphi\right)e^{-f}d\mu_g+m_{\operatorname{ADM}}(g_b+h),
		\label{first-var-F}
		\end{split}
		\end{equation}
		where $m_{\operatorname{ADM}}(g_b+h)$ is the mass of the metric $g_b+h$ defined in (\ref{def-mass}). 
		
		Finally, the first variation of $\lambda_{\operatorname{ALE}}^0$ on a neighborhood of $\mathcal{M}^{2,\alpha}_{\tau}(g_b,\varepsilon)$ is:
		\begin{equation}
		\begin{split}
		\delta_g \lambda_{\operatorname{ALE}}^0(h)=&-\int_N\langle h,\Ric(g)+\nabla^{g,2}f_g\rangle_g \,e^{-f_g}d\mu_g+m_{\operatorname{ADM}}(g_b+h).\label{first-var-lambda}
		\end{split}
		\end{equation}
	\end{prop}
	\begin{rk}
		Notice that (\ref{first-var-lambda}) gives a link between the variation of the functional $\lambda_{\operatorname{ALE}}^0$, the mass of an ALE metric with integrable scalar curvature and its associated Bakry-\'Emery tensor.
	\end{rk}
	\begin{proof}
		We follow [Chapter $2$, \cite{Cho-Boo}] closely by using (\ref{lem-lin-equ-scal-first-var}) from Lemma \ref{lem-lin-equ-Ric-first-var}:
		\begin{equation*}
		\begin{split}
		\delta_{g,f}\big[\big(|\nabla^gf|^2_g+\R_g\big)\,&e^{-f}d\mu_g\big](h,\varphi)\\
		=\,&\left(\div_g(\div_gh)-\Delta_g\tr_gh-\langle h,\Ric(g)\rangle-h(\nabla^gf,\nabla^gf)\right) e^{-f}d\mu_g\\
		&+\left(2\langle \nabla^gf,\nabla^g\varphi\rangle+(\R_g+|\nabla^gf|^2_g)\left(\frac{\tr_gh}{2}-\varphi\right)\right)e^{-f}d\mu_g.
		\end{split}
		\end{equation*}
		Now, by integrating by parts twice on the domain $\{\rho_{g_b}\leq R\}$ with $R$ sufficiently large such that $\{\rho_{g_b}= R\}$ is a smooth compact hypersurface:
		\begin{equation*}
		\begin{split}
		\int_{\{\rho_{g_b}\leq R\}}\div_g(\div_gh)\,e^{-f}d\mu_g=\,&\int_{\{\rho_{g_b}\leq R\}}\left(h(\nabla^gf,\nabla^gf)-\langle h,\nabla^{g,2}f\rangle\right)e^{-f}d\mu_g\\
		&+\int_{\{\rho_{g_b}=R\}}\left<\div_gh+h(\nabla^gf),\mathbf{n}_{g_b}\right>\,e^{-f}d\sigma_g.
		\end{split}
		\end{equation*}
		Moreover,
		\begin{equation*}
		\begin{split}
		-\int_{\{\rho_{g_b}\leq R\}}\Delta_g\tr_gh\,e^{-f}d\mu_g=\,&-\int_{\{\rho_{g_b}\leq R\}}\langle\nabla^g\tr_gh,\nabla^gf\rangle \,e^{-f}d\mu_g\\
		&-\int_{\{\rho_{g_b}= R\}}\left<\nabla^g\tr_gh,\mathbf{n}_{g_b}\right>\,e^{-f}d\sigma_g\\
		=\,&\int_{\{\rho_{g_b}\leq R\}}\left(\Delta_gf-|\nabla^gf|^2_g\right)\tr_gh\,e^{-f}d\mu_g\\
		&-\int_{\{\rho_{g_b}= R\}}\left<\nabla^g\tr_gh,\mathbf{n}_{g_b}\right>\,e^{-f}d\sigma_g\\
		&-\int_{\{\rho_{g_b}=R\}}\tr_{g}h\langle\nabla^gf,\mathbf{n}_{g_b}\rangle\,e^{-f}d\sigma_g,\\
		2\int_{\{\rho_{g_b}\leq R\}}\langle \nabla^gf,\nabla^g\varphi\rangle\,e^{-f}d\mu_g=\,& -2\int_{\{\rho_{g_b}\leq R\}}\left(\Delta_gf-|\nabla^gf|^2_g\right)\varphi\,e^{-f}d\mu_g\\
		&+2\int_{\{\rho_{g_b}=R\}}\varphi\langle\nabla^gf,\mathbf{n}_{g_b}\rangle\,e^{-f}d\sigma_g.
		\end{split}
		\end{equation*}
		Therefore the expected result follows by summing the previous equalities and let $R$ go to $+\infty$ by using the asymptotics assumed on $h$, $f$ and $\varphi$ together with the definition of the mass of $g_b+h$ given in (\ref{def-mass}).
		
		 In order to compute the first variation of $\lambda_{\operatorname{ALE}}^0$, one proceeds similarly by setting $\varphi:=\delta_gf(h)$ and by using 
		(\ref{equ-criti-lambda-pot}) to cancel the second integral term on the righthand side of (\ref{first-var-F}). By density of $C_c^{\infty}$ in $L^2_{\frac{n}{2}-1}$, notice that (\ref{first-var-lambda}) still holds true for variations $h\in L^2_{\frac{n}{2}-1}$ since $\Ric(g)+\nabla^{g,2}f_g=\textit{O}(\rho_{g_b}^{-\tau-2})\in L^2_{\frac{n}{2}+1}$.

	\end{proof}

	
	As a first consequence of Proposition \ref{first-var-prop}, we recover (in our non-compact setting) the fact stated in [Remark $4.6$, \cite{Has-Sta}] that the weighted $L^2$-norm of the Bakry-\'Emery tensor $\Ric(g)+\nabla^{g,2}f_g$ associated to the functional $\lambda_{\operatorname{ALE}}^0$ is dominated by that of the Ricci curvature. More precisely, we have the following result:
	\begin{coro}
		Let $(N^n,g_b)$ be an ALE Ricci-flat metric asymptotic to $\RR^n\slash\Gamma$, for some finite subgroup $\Gamma$ of $SO(n)$ acting freely on $\mathbb{S}^{n-1}$. Let $\tau\in(\frac{n-2}{2},n-2)$ and $\alpha\in(0,1)$. 
		
		Then there exists $\varepsilon>0$ such that if $g\in B_{C^{2,\alpha}_{\tau}}(g_b,\varepsilon)$, the tensor $\Ric(g)+\nabla^{g,2}f_g$ is weighted divergence-free, i.e. 
		\begin{equation}
		\div_{f_g}\left(\Ric(g)+\nabla^{g,2}f_g\right)=0.\label{wei-div-free-obs-ten}
		\end{equation}
		In particular, if $g\in B_{C^{2,\alpha}_{\tau}}(g_b,\varepsilon)$,
		\begin{equation}
		\|\Ric(g)+\nabla^{g,2}f_g\|_{L^2(e^{-f_g}d\mu_g)}\leq \|\Ric(g)\|_{L^2(e^{-f_g}d\mu_g)}.\label{easy-inequ-ric-bak-eme}
		\end{equation}
		
	\end{coro}
	Notice that the quantitative version of the reverse inequality of (\ref{easy-inequ-ric-bak-eme}) is more delicate to prove in general: see \cite[Theorem C]{Has-Sta} for closed Ricci-flat metrics in the integrable case.
	
	\begin{proof}
	Let us prove that $\Ric(g)+\nabla^{g,2}f_g$ is divergence-free in the (weighted) sense of (\ref{def-wei-div-sym}):
		\begin{equation}
		\begin{split}
		2\div_{f_g}\left(\Ric(g)+\nabla^{g,2}f_g\right)&=2\div_g\left(\Ric(g)+\nabla^{g,2}f_g\right)-2\left(\Ric(g)+\nabla^{g,2}f_g\right)(\nabla^gf_g)\\
		&=\nabla^g\R_g+\div_g\Li_{\nabla^gf_g}(g)-2\left(\Ric(g)+\nabla^{g,2}f_g\right)(\nabla^gf_g)\\
		&=\nabla^g\R_g+\frac{1}{2}\nabla^g\tr_g\Li_{\nabla^gf_g}(g)+\Delta_g\nabla^gf_g+\Ric(g)(\nabla^gf_g)\\
		&-2\left(\Ric(g)+\nabla^{g,2}f_g\right)(\nabla^gf_g)\\
		&=\nabla^g\left(\R_g+2\Delta_gf_g-|\nabla^gf_g|^2_g\right)\\
		&=0.
		\end{split}
		\end{equation}
		Here, we have used the Bianchi identity (its traced version) in the second line together with the Bochner formula for vector fields in the third line and the one for functions in the fourth line. The last line comes from [(\ref{equ-criti-lambda-pot}), Proposition \ref{prop-pot-fct}].\\
		
		The proof of (\ref{easy-inequ-ric-bak-eme}) is essentially due to (\ref{wei-div-free-obs-ten}).
		Indeed, if $X$ is any smooth vector field which is compactly supported (or decaying faster than $\rho_{g_b}^{-\frac{n}{2}+2}$) on $N$, then 
		\begin{equation*}
		\begin{split}
		\left<\Ric(g)+\nabla^{g,2}f_g,\Li_X(g)\right>_{L^2(e^{-f_g}d\mu_g)}=0,
		\end{split}
		\end{equation*}
		by integration by parts. In particular, by applying this fact to $X=\nabla^{g}f_g=O(\rho_{g_b}^{-\tau-1})$ by Proposition \ref{prop-pot-fct}, one gets, 
		\begin{equation}
		\begin{split}
		\|\Ric(g)+\nabla^{g,2}f_g\|_{L^2(e^{-f_g}d\mu_g)}^2=\,&\|\Ric(g)\|^2_{L^2(e^{-f_g}d\mu_g)}+2\left<\Ric(g),\nabla^{g,2}f_g\right>_{L^2(e^{-f_g}d\mu_g)}\\
		&+\|\nabla^{g,2}f_g\|^2_{L^2(e^{-f_g}d\mu_g)}\\
		=\,&\|\Ric(g)\|^2_{L^2(e^{-f_g}d\mu_g)}+2\left<\Ric(g)+\nabla^{g,2}f_g,\nabla^{g,2}f_g\right>_{L^2(e^{-f_g}d\mu_g)}\\
		&-\|\nabla^{g,2}f_g\|^2_{L^2(e^{-f_g}d\mu_g)}\\
		=\,&\|\Ric(g)\|^2_{L^2(e^{-f_g}d\mu_g)}-\|\nabla^{g,2}f_g\|^2_{L^2(e^{-f_g}d\mu_g)}\\
		\leq\,& \|\Ric(g)\|^2_{L^2(e^{-f_g}d\mu_g)}.
		\end{split}
		\end{equation}
		All the integrals and integration by parts are justified here by the sufficiently fast decays at infinity satisfied by $\Ric(g)$ and $f_g$ and their covariant derivatives.

	\end{proof}

	We are in a good position to compute the second variation of $\lambda_{\operatorname{ALE}}^0$. We first need one more definition:
	\begin{defn}\label{defn-Lic-Op}
		Let $(N^n,g)$ be a Riemannian metric. Then the \emph{Lichnerowicz operator} associated to $g$ acting on symmetric $2$-tensors, denoted by $L_{g}$, is defined by:
		\begin{equation}
		L_{g}h:=\Delta_{g}h + 2\Rm(g)(h)-\Ric(g)\circ h-h\circ\Ric(g),\quad h\in C_{loc}^2(S^2T^*N),\label{defn-Lic-op-eq}
		\end{equation}
		where $\Delta_{g}=-\nabla^*\nabla$ and where $\Rm(g)(h)(X,Y) := h(\Rm(g)(e_i,X)Y,e_i)$ for an orthonormal basis $(e_i)_{i=1}^n$ with respect to $g$. In particular, if $(N^n,g)$ is a Ricci-flat metric, then,
		\begin{equation}
		L_{g}h:=\Delta_{g}h + 2\Rm(g)(h),\quad h\in C_{loc}^2(S^2T^*N).\label{defn-Lic-op-eq}
		\end{equation}
	\end{defn}
	
	With this definition in hand, we are able to identify the second variation of $\lambda_{\operatorname{ALE}}^0$  at an ALE Ricci-flat metric as follows.
	\begin{prop}[Second variation of $\lambda_{\operatorname{ALE}}^0$ at a Ricci-flat metric]\label{second-var-prop}
		Let $(N^n,g_b)$ be an ALE Ricci-flat metric asymptotic to $\RR^n\slash\Gamma$, for some finite subgroup $\Gamma$ of $SO(n)$ acting freely on $\mathbb{S}^{n-1}$. Let $\tau\in(\frac{n-2}{2},n-2)$ and $\alpha\in(0,1)$. Then the second variation of $\lambda_{\operatorname{ALE}}^0$ at $g_b$ along a divergence-free variation $h\in S^2T^*N$ such that $g_b+h\in \mathcal{M}^{2,\alpha}_{\tau}(g_b,\varepsilon)$ is:
		\begin{equation}
		\delta^2_{g_b}\lambda_{\operatorname{ALE}}^0(h,h)=\frac{1}{2}\langle L_{g_b}h,h\rangle_{L^2}.
		\end{equation}
	
	\end{prop}
	
	\begin{proof}
		
		Recall by Lemma \ref{lem-lin-equ-Ric-first-var} that, if we denote the Bianchi operator $B_g(h):= \div_g\big(h-\frac{1}{2}(\tr_g h)g\big)$, we have
		\begin{equation*}
		\begin{split}
		\delta_{g_b}(-2\Ric)(h)&=L_{g_b}h-\Li_{B_{g_b}(h)}(g_b)\\
		&=L_{g_b}h+\frac{1}{2}\Li_{\nabla^{g_b}\tr_{g_b}h}(g_b),
		\end{split}
		\end{equation*}
		if $\div_{g_b}h=0$. Since $f_{g_b}=0$ then, denoting $\delta_{g_b}f(h)$ the first order variation of $g\mapsto f_g$ at $g_b$ in the direction $h$, we have:
		\begin{equation*}
		\delta_{g_b}\Li_{\nabla^gf_g}(h)=\Li_{\nabla^{g_b}\delta_{g_b}f(h)}(g_b).
		\end{equation*}
		Therefore, according to Proposition \ref{first-var-prop},
		\begin{equation*}
		2\delta_{g_b}^2\lambda_{\operatorname{ALE}}^0(h,h)=\langle L_{g_b}h,h\rangle_{L^2}+ \frac{1}{2}\langle\Li_{\nabla^{g_b}\tr_{g_b}h}(g_b),h\rangle_{L^2}-\langle \Li_{\nabla^{g_b}\left(\delta_{g_b}f(h)\right)}(g_b),h\rangle_{L^2}.
		\end{equation*}
		By integrating by parts:
		\begin{equation*}
		\begin{split}
		2\delta_{g_b}^2\lambda_{\operatorname{ALE}}^0(h,h)&=\langle L_{g_b}h,h\rangle_{L^2}-\langle\nabla^{g_b}\tr_{g_b}h,\div_{g_b}h\rangle_{L^2}+2\langle \delta_{g_b}f(h),\div_{g_b}h\rangle_{L^2}\\
		&=\langle L_{g_b}h,h\rangle_{L^2},
		\end{split}
		\end{equation*}
		if $\div_{g_b}h=0$.
	\end{proof}
	\begin{rk}\label{remark jauge div}
		Notice that if $h$ is a symmetric $2$-tensor $h$ on $N$ then by differentiating (\ref{equ-criti-lambda-pot}) at $g_b$ gives:
		\begin{eqnarray}
		2\Delta_{g_b}\left(\delta_{g_b}f_{g}(h)-\frac{\tr_{g_b}h}{2}\right)&=&-\div_{g_b}(\div_{g_b}h).\label{equ-first-var-pot-fct}
		\end{eqnarray}
		Now, if $\div_{g_b}h=0$ then the function $\delta_{g_b}f_{g}(h)-\frac{\tr_{g_b}h}{2}$ is a harmonic function on $N$. In the setting of Proposition \ref{second-var-prop}, the maximum principle applied to the previous function shows that it vanishes identically, i.e. the pointwise volume $e^{-f_g}d\mu_g$ is preserved under such a variation. On the other hand, if $B_{g_b}(h)=0$ then one gets by (\ref{equ-first-var-pot-fct}) that $\delta_{g_b}f_{g}(h)-\frac{\tr_{g_b}h}{4}=0$. This implies that the second variation of $\lambda_{\operatorname{ALE}}^0$ at $g_b$ along $h\in B_{g_b}^{-1}(0)$ satisfies:
		\begin{equation*}
		2\delta^2_{g_b}\lambda_{\operatorname{ALE}}^0(h,h)=\langle L_{g_b}h,h\rangle_{L^2}+\frac{1}{2}\|\nabla^{g_b}\tr_{g_b}h\|^2_{L^2}.
		\end{equation*}
		This leaves us with a less tractable formula for the second derivative of $\lambda_{\operatorname{ALE}}^0$.
	\end{rk}

	\section{A functional defined on a $C^{2,\alpha}_\tau$-neighborhood of Ricci-flat ALE metrics}\label{extension tilde lambda}
	\subsection{First and second variations of $\lambda_{\operatorname{ALE}}$}~~\\
	
	Let $(N^n,g_b)$ be an ALE Ricci-flat metric. We first recall that there exists a sequence of metrics $C^{2,\alpha}_\tau$-converging to $ g_b $ while having unbounded mass and $\lambda_{\operatorname{ALE}}^0$-functional. 
		\begin{exmp}\label{exemple masse infinie}
		Denote for $A>0$ large enough, $\chi_A$ a cut-off function supported in $\{\rho_{g_b}>A\}$ where $g_b$ has ALE coordinates, and constant equal to $1$ on $\{\rho_{g_b}>2A\}$ and assume that for $c>0$ uniform and all $k\in \{1,2,3\}$, its $k$-th derivative is bounded by $cA^{-k}$. Let us define the metric
		$$g_{A,m}:= \left(1+\chi_A\frac{m}{\rho_{g_b}^{n-2}}\right)^{\frac{4}{n-2}}g_b,$$
		whose scalar curvature vanishes on $\{\rho_{g_b}>2A\}$ by the usual variation of the scalar curvature for conformal changes of metric.
		
		Then, for some constant $C>0$ we have $$\|g_{A,m}-g_e\|_{C^{2,\alpha}_\tau}\leq C |m|A^{\tau-(n-2)}, $$
		and for some $c_n>0$, we have $m_{\operatorname{ADM}}(g_{A,m}) = c_n m$. As a consequence, by choosing $m\to \pm \infty$ while $|m|A^{\tau-(n-2)}\to 0$, we get a sequence of metrics $C^{2,\alpha}_\tau$-converging to $g_b$ while its mass tends to $\pm\infty$.
	\end{exmp}

	The computation of the first variation of $\lambda_{\operatorname{ALE}}^0$ in Proposition \ref{first-var-prop} motivates the study of the difference $\lambda_{\operatorname{ALE}}^0-m_{\operatorname{ADM}}$.
	
	\begin{defn}[$\lambda_{\operatorname{ALE}}$, a renormalized Perelman's functional]\label{defn-ale-lambda}
		Let $(N^n,g_b)$ be an ALE Ricci-flat metric and let $g\in \mathcal{M}^{2,\alpha}_{\tau}(g_b,\varepsilon)$ for $\tau>\frac{n-2}{2}$. We define 
		$$\lambda_{\operatorname{ALE}}(g) := \lambda_{\operatorname{ALE}}^0(g)-m_{\operatorname{ADM}}(g).$$
	\end{defn}
	
	\begin{rk}
		This is reminiscent of the introduction of the mass in General Relativity in order to replace the Hilbert-Einstein functional $\int_N\R_gd\mu_g$ by $\int_N\R_gd\mu_g-m_{\operatorname{ADM}}(g)$ which is better-behaved in the setting of  Asymptotically Euclidean metrics. Our functional $\lambda_{\operatorname{ALE}}$ is moreover an approximation up to third order:
		$$\int_N\R_{g_b+h}\,d\mu_{g_b+h}-m_{\operatorname{ADM}}(g_b+h)- \lambda_{\operatorname{ALE}}(g_b+h) = O\left(\|h\|^3_{C^{2,\alpha}_\tau}\right).$$
	\end{rk}
	Notice that Definition \ref{defn-ale-lambda} only makes sense for metrics lying in $\mathcal{M}^{2,\alpha}_{\tau}(g_b,\varepsilon)$ a priori. The following proposition ensures the functional $\lambda_{\operatorname{ALE}}$ is well-defined on a whole neighborhood of a given Ricci-flat ALE metric in the $C^{2,\alpha}_{\tau}$-topology.
	\begin{prop}\label{lambdaALE analytic}
		Let $(N^n,g_b)$ be an ALE Ricci-flat metric asymptotic to $\RR^n\slash\Gamma$, for some finite subgroup $\Gamma$ of $SO(n)$ acting freely on $\mathbb{S}^{n-1}$. Let $\tau\in(\frac{n-2}{2},n-2)$ and $\alpha\in(0,1)$. Then, the functional $\lambda_{\operatorname{ALE}}$, initially defined on $\mathcal{M}^{2,\alpha}_{\tau}(g_b,\varepsilon)$, extends to a $C^{2,\alpha}_\tau$-neighborhood of $ g_b $ as an analytic functional
		\begin{itemize}
			\item whose $L^2(e^{-f_g}d\mu_g)$-gradient at $g$ is $-(\Ric(g)+\nabla^{g,2}f_g),$
			\item and whose second variation at $g_b$ for divergence-free $2$-tensors is $\frac{1}{2}L_{g_b}$.
		\end{itemize}
		Moreover, if $g\in B_{C^{2,\alpha}_{\tau}}(g_b,\varepsilon)$,
		\begin{equation}
		\begin{split}
\lambda_{\operatorname{ALE}}(g)=\lim_{R\rightarrow+\infty}\Bigg(\int_{\{\rho_{g_b}\leq R\}}&\left(|\nabla^{g}f_{g}|^2_{g}+\R_{g}\right)\,e^{-f_{g}}d\mu_{g}\\
&-\int_{\{\rho_{g_b}=R\}}\left<\div_{g_b}(g)-\nabla^{g_b}\tr_{g_b}(g),\mathbf{n}_{g_b}\right>_{g_b}\,d\sigma_{g_b}\Bigg).\label{true-def-lambda}
\end{split}
\end{equation}
Finally, the functional $\lambda_{\operatorname{ALE}}$ is kept unchanged on $C^{2,\alpha'}_{\tau'}$-neighborhoods of $g_b$ for $\tau'\in (\frac{n-2}{2},\tau]$ and $\alpha'\in(0,\alpha]$.
	\end{prop}
	Before proving Proposition \ref{lambdaALE analytic}, we make a couple of remarks.
	\begin{rk}
		The largest definition space seems to be $h\in H^2_{\frac{n}{2}-1}$ where the first derivative $\langle h, \Ric_g+\nabla^{g,2}f_g\rangle_{L^2(e^{-f_g}d\mu_g)}$ is well defined for $g-g_b\in H^2_{\frac{n}{2}-1}$. The second derivative $\frac{1}{2}\langle h, L_{g_b}h\rangle_{L^2(e^{-f_g}d\mu_g)}$ is also well defined for $2$-tensors in $H^2_{\frac{n}{2}-1}$. However, it is not clear if the definition of the mass $m_{\operatorname{ADM}}$ is invariant by changes of coordinates under these assumptions.
	\end{rk}
	\begin{rk}
	As already noticed in the Introduction, $\lambda_{\operatorname{ALE}}(g)$ is the limit of the difference of two integrals which could be divergent in general. However, if $g\in B_{C^{2,\alpha}_{\tau}}(g_b,\varepsilon)$ is such that its scalar curvature is integrable, then $\lambda_{\operatorname{ALE}}(g)$ really is the difference of $\lambda_{\operatorname{ALE}}^0(g)$ with the mass $m_{\operatorname{ADM}}(g)$.
	\end{rk}
	\begin{proof}[Proof of Proposition \ref{lambdaALE analytic}]
		Let us consider a small $h$ in $ C^{2,\alpha}_\tau$ for $\tau>\frac{n-2}{2}$ such that $g_b+h$ is a metric on $N$. Thanks to [(\ref{lem-lin-equ-scal-first-var}), Lemma \ref{lem-lin-equ-Ric-first-var}], we have
		\begin{align}
		\R_{g_b+h} =& \int_0^1 \div_{g_b+th}( \div_{g_b+th}h - \nabla^{g_b+th}\textup{tr}_{g_b+th}h )\,dt\\
		=& \div_{g_b}( \div_{g_b}h - \nabla^{g_b}\tr_{g_b}h ) \nonumber\\
		&+\int_0^1 \Big[\big(\div_{g_b+th}( \div_{g_b+th} - \nabla^{g_b+th}\textup{tr}_{g_b+th})\big)-\big(\div_{g_b}( \div_{g_b} - \nabla^{g_b}\textup{tr}_{g_b})\big)\Big](h)\, dt\nonumber\\
		=&:\div_{g_b}( \div_{g_b}h - \nabla^{g_b}\textup{tr}_{g_b}h ) + Q_{g_b}(h),\label{dvp scal}
		\end{align}
		where $Q_{g_b}:C^{2,\alpha}_\tau\to \mathbb{R}$ is analytic and satisfies for some positive constant $C$, for any two symmetric $2$-tensors $h$ and $h'$, $$\|Q_{g_b}(h)- Q_{g_b}(h)\|_{C^{0,\alpha}_{2\tau +2}}\leq C\|h-h'\|_{C^{2,\alpha}_{\tau}} \left(\|h\|_{C^{2,\alpha}_{\tau}}+\|h'\|_{C^{2,\alpha}_{\tau}}\right).$$
		
		For a symmetric $2$-tensor $h\in B_{C^{2,\alpha}_{\tau}}(g_b,\varepsilon)$, denote $v_{g_b+h}\in C^{2,\alpha}_\tau$ to be the unique solution to 
		$$-4\Delta_{g_b+h}v_{g_b+h} + \R_{g_b+h}v_{g_b+h} = -\R_{g_b+h} =- \div_{g_b}\left(\div_{g_b}h-\nabla^{g_b}\textup{tr}_{g_b}h\right) - Q_{g_b}(h) \in C^{0,\alpha}_{\tau+2}.$$
		Its existence is ensured because $-4\Delta_{g_b+h} + \R_{g_b+h}: C^{2,\alpha}_\tau\to C^{0,\alpha}_{\tau+2}$ is invertible: indeed, we are in the invertibility range $0<\tau<n-2$ of the Laplacian as already noticed in Claim \ref{claim-iso}.
		
		We have already seen in (the proof of) Proposition \ref{existence propriete-wg}  by integration by parts against $1+v_{g_b+h}$ that we actually have 
		$$\lambda_{\operatorname{ALE}}^0(g_b+h) = \int_N (1+v_{g_b+h})\R_{g_b+h}\,d\mu_{g_b+h},$$
		if $g_b+h\in\mathcal{M}^{2,\alpha}(g_b,\varepsilon)$.
		
		Let us now consider the following expression, for $g_b+h\in\mathcal{M}^{2,\alpha}(g_b,\varepsilon)$, $$ \int_N (1+v_{g_b+h})\R_{g_b+h}\,d\mu_{g_b+h} - m_{\operatorname{ADM}}(g_b+h). $$ Use \eqref{dvp scal} together with the fact that 
		\begin{equation*}
		\begin{split}
		&\int_N\div_{g_b}( \div_{g_b}h - \nabla^{g_b}\textup{tr}_{g_b}h)\,d\mu_{g_b} - m_{\operatorname{ADM}}(g_b+h) =\\
		&\lim_{R\rightarrow+\infty}\Bigg(\int_{\{\rho_{g_b}\leq R\}}\div_{g_b}( \div_{g_b}h - \nabla^{g_b}\textup{tr}_{g_b}h)\,d\mu_{g_b}-\int_{\{\rho_{g_b}=R\}}\left<\div_{g_b}h-\nabla^{g_b}\tr_{g_b}h,\mathbf{n}_{g_b}\right>_{g_b}\,d\sigma_{g_b}\Bigg)\\
		&=0,
		\end{split}
		\end{equation*}
		 noticed in (\ref{def-mass}), to obtain
		\begin{align}
		\lambda_{\operatorname{ALE}}(g_b+h)=&\;\int_N v_{g_b+h}\R_{g_b+h}\,d\mu_{g_b+h} + \int_N Q_{g_b}(h)\,d\mu_{g_b+h} \\
		&+ \Big(\int_N\div_{g_b}( \div_{g_b}h - \nabla^{g_b}\tr_{g_b}h) - m_{\operatorname{ADM}}(g_b+h) \Big)\\
		=&\; \int_N v_{g_b+h}\R_{g_b+h}\,d\mu_{g_b+h} + \int_N Q_{g_b}(h)\,d\mu_{g_b+h}.\label{last-exp-make-sense-lambda}
		\end{align}
		This last expression (\ref{last-exp-make-sense-lambda}) is well-defined and analytic on a $C^{2,\alpha}_\tau$-neighborhood of $g_b$. A similar argument leads to the proof of (\ref{true-def-lambda}) by using the first expression of $\lambda_{\operatorname{ALE}}^0(g)$ given in [(\ref{egalite-lambda-2-bis}), Proposition \ref{prop-pot-fct}], $\lambda_{\operatorname{ALE}}^0(g)=\int_N(|\nabla^gf_g|^2_g+\R_g)\,e^{-f_{g}}d\mu_g$ instead.
		
		 Moreover, for any metric $g$ in this $C^{2,\alpha}_\tau$-neighborhood of $g_b$ and from the computations of Proposition \ref{first-var-prop}, the $L^2(e^{-f_g}d\mu_g)$-gradient of $\lambda_{\operatorname{ALE}}$ at $g$ is $-(\Ric(g)+\nabla^{g,2} f_g)$. Since $h\mapsto m_{\operatorname{ADM}}(g_b+h)$ is linear, the Hessian of $\lambda_{\operatorname{ALE}}$ at $g_b$ for divergence-free deformations is $\frac{1}{2}L_{g_b}$. 
		 
		 Finally, $\lambda_{\operatorname{ALE}}$ is independent of $(\tau',\alpha')\in(\frac{n-2}{2},\tau]\times(0,\alpha]$ since the potential function $f_g$ is. Indeed, let $\frac{n-2}{2}<\tau'\leq \tau<n-2$ and $0<\alpha'\leq \alpha<1$ and let us show that the map $g\rightarrow f_g$ is constant as $g$ varies in $B_{C^{2,\alpha'}_{\tau'}}(g_b,\varepsilon)\,\cap\,B_{C^{2,\alpha}_{\tau}}(g_b,\varepsilon)$. Recall from Proposition \ref{existence propriete-wg} that $w_g:=e^{-\frac{f_g}{2}}$ satisfies $-4\Delta_gw_g+\R_gw_g=0$ and $w_g-1\in C^{2,\alpha}_{\tau}$ by definition. Let $w:=w_g$ (respectively $w':=w_g$) if $g\in B_{C^{2,\alpha}_{\tau}}(g_b,\varepsilon)$ (respectively if $g\in B_{C^{2,\alpha'}_{\tau'}}(g_b,\varepsilon)$). Then the difference $w-w'\in C^{2,\alpha'}_{\tau'}$ lies in the kernel of $-4\Delta_g+\R_g$. Claim \ref{claim-iso} of the proof of Proposition \ref{existence propriete-wg}  leads to $w=w'$.

	\end{proof}
	
	The next proposition computes the second variation of $\lambda_{\operatorname{ALE}}$ at any metric $C^{2,\alpha}_{\tau}$-close to a given ALE Ricci flat metric. This is used in the proof of Theorem \ref{theo-loja-ALE} in order to check condition [(\ref{item-0-bis}), Proposition \ref{Lojasiewicz ineq weighted}].   
	\begin{prop}\label{snd-var-gal-lambda}
		Let $(N^n,g_b)$ be an ALE Ricci-flat metric asymptotic to $\RR^n\slash\Gamma$, for some finite subgroup $\Gamma$ of $SO(n)$ acting freely on $\mathbb{S}^{n-1}$. Let $\tau\in(\frac{n-2}{2},n-2)$ and $\alpha\in(0,1)$. Then there exists $\varepsilon>0$ such that for any $g\in B_{C^{2,\alpha}_{\tau}}(g_b,\varepsilon)$ and any $h\in C^{2,\alpha}_{\tau}$,
		\begin{equation}
		\begin{split}\label{snd-var-gal-lambda-formula}
\delta^2_g\lambda_{\operatorname{ALE}}&(h,h)=\\
&\frac{1}{2}\int_N\left\langle\Delta_{f_g}h+2\Rm(g)(h)-\Li_{B_{f_g}(h)}(g),h\right\rangle_g\,e^{-f_g}d\mu_g\\
&\quad+\frac{1}{2}\int_N\left\langle h\circ\Ric_{f_g}(g)+\Ric_{f_g}(g)\circ h-2\left(\frac{\tr_gh}{2}-\delta_gf(h)\right)\Ric_{f_g}(g),h\right\rangle_g\,e^{-f_g}d\mu_g.
		\end{split}
		\end{equation}
		Here, $\Ric_{f_g}(g)$ denotes the Bakry-\'Emery tensor $\Ric(g)+\nabla^{g,2}f_g$ associated to the smooth metric measure space $(N^n,g,\nabla^gf_g)$ and $B_{f_g}(h)$ denotes the weighted linear Bianchi gauge defined by 
		$$B_{f_g}(h):=\div_{f_g}h-\nabla^g\left(\frac{\tr_gh}{2}-\delta_gf(h)\right).$$
		\end{prop}
	\begin{rk}
	In \eqref{snd-var-gal-lambda-formula}, notice that the function $\frac{\tr_gh}{2}-\delta_gf(h)$ is nothing but the infinitesimal variation of the weighted volume $e^{-f_g}d\mu_g$ and the weighted linear Bianchi gauge $B_{f_g}(h)$ differs from the linear Bianchi gauge defined in [\ref{defn-bianchi-op}, Lemma \ref{lem-lin-equ-Ric-first-var}] by $-h(\nabla^gf_g)+\nabla^g\delta_gf(h)$. This vector field is in turn the variation of the vector field $\nabla^g f_g$.
		\end{rk}
		\begin{rk}
		Proposition \ref{snd-var-gal-lambda} recovers the second variation of $\lambda_{\operatorname{ALE}}$ at $g=g_b$ along divergence-free variations.
		\end{rk}
		\begin{proof}
		We consider $\varepsilon>0$ so small that $f_g$ is well-defined by Proposition \ref{prop-pot-fct}. Now, as the $L^2(e^{-f_g}d\mu_g)$-gradient of $\lambda_{\operatorname{ALE}}$ is $-\Ric(g)-\nabla^{g,2}f_g=:-\Ric_{f_g}(g)$ by Proposition \ref{lambdaALE analytic}, we deduce the following formula with the help of Lemma \ref{lem-lin-equ-Ric-first-var} and [\eqref{first-var-lie-der-app}, Lemma \ref{lemma-app-lie-der-lin}]:
		\begin{equation}
\begin{split}\label{form-var-bak-eme}
2\delta_g(-\Ric_{f_g}(g))(h)&=\left(L_gh-\Li_{B_g(h)}(g)\right)-\Li_{\nabla^g(\delta_gf(h))}(g)-\Li_{\nabla^gf_g}(h)+\Li_{h(\nabla^gf_g)}(g)\\
&=\Delta_gh+2\Rm(g)(h)-\Ric(g)\circ h-h\circ\Ric(g)\\
&\quad-\nabla^g_{\nabla^gf}h-\nabla^{g,2}f_g\circ h-h\circ\nabla^{g,2}f_g-\Li_{B_{f_g}(h)}(g)\\
&=\Delta_{f_g}h+2\Rm(g)(h)-\Ric_{f_g}(g)\circ h-h\circ\Ric_{f_g}(g)-\Li_{B_{f_g}(h)}(g).
\end{split}
\end{equation}
Here we have used the general fact that $\Li_{\nabla^gf}T=\nabla^g_XT+T\circ\nabla^{g,2}f+\nabla^{g,2}f\circ T$ for any symmetric $C^1_{loc}$ $2$-tensor and any $C^1_{loc}$ function $f$.

Next, we observe that the volume variation is 
\begin{equation}\label{form-weig-var}
\delta_g(e^{-f_g}d\mu_g)(h)=\left(\frac{\tr_gh}{2}-\delta_gf(h)\right)e^{-f_g}d\mu_g.
\end{equation}
Finally, we compute the variation with respect to the norm on symmetric $2$-tensors induced by the metric $g$ as follows:
\begin{equation}
\begin{split}\label{gal-form-var-norm}
\delta_g\left(\left\langle h,T\right\rangle_g\right)(h)&=-g^{im}h_{mn}g^{nk}g^{jl}h_{ij}T_{kl}-g^{ik}g^{jm}h_{mn}g^{nl}h_{ij}T_{kl}\\
&=-\left\langle h\circ T+T\circ h,h\right\rangle_g,
\end{split}
\end{equation}
for any symmetric $2$-tensor $T$. Then \eqref{snd-var-gal-lambda-formula} follows by considering the linear combination $\frac{1}{2}$\eqref{form-var-bak-eme}+\eqref{form-weig-var} to which we add \eqref{gal-form-var-norm} applied to $T:=-\Ric_{f_g}(g)$.

		\end{proof}
		
		We end this section by establishing the weighted elliptic equation satisfied by the volume variation at a metric $C^{2,\alpha}_{\tau}$-close to a given ALE Ricci flat metric.
		Again, this is used in the proof of Theorem \ref{theo-loja-ALE} in order to check conditions [(\ref{item-0}), (\ref{item-0-bis}), Proposition \ref{Lojasiewicz ineq weighted}].   
	\begin{prop}\label{var-vol-var-ell-eqn-prop}
		Let $(N^n,g_b)$ be an ALE Ricci-flat metric asymptotic to $\RR^n\slash\Gamma$, for some finite subgroup $\Gamma$ of $SO(n)$ acting freely on $\mathbb{S}^{n-1}$. Let $\tau\in(\frac{n-2}{2},n-2)$ and $\alpha\in(0,1)$. Then there exists $\varepsilon>0$ such that for any $g\in B_{C^{2,\alpha}_{\tau}}(g_b,\varepsilon)$ and any $h\in C^{2,\alpha}_{\tau}$,
		\begin{equation}
		\begin{split}\label{var-vol-var-ell-eqn-for}
		\Delta_{f_g}\left(\frac{\tr_gh}{2}-\delta_gf(h)\right)=\frac{1}{2}\left(\div_{f_g}(\div_{f_g}h)-\langle h,\Ric_{f_g}(g)\rangle_g\right).
		\end{split}
		\end{equation}
		\end{prop}
		\begin{proof}
		We consider $\varepsilon>0$ so small that $f_g$ is well-defined by Proposition \ref{prop-pot-fct}. In particular, the potential function $f_g$ satisfies the Euler-Lagrange equation \eqref{equ-criti-lambda-pot} that we differentiate along a variation $h\in C^{2,\alpha}_{\tau}$ as follows:
		\begin{equation}
		\begin{split}\label{diff-var-for-eul-lag}
2\delta_g\left(\Delta_gf_g\right)(h)&=2\Delta_g\delta_gf(h)-2\langle B_g(h),\nabla^gf_g\rangle_g-2\langle h,\nabla^{g,2}f_g\rangle_g,\\
\delta_g(|\nabla^gf_g|^2_g)(h)&=-h(\nabla^gf_g,\nabla^gf_g)+2\langle\nabla^gf_g,\nabla^g(\delta_gf(h))\rangle_g,\\
\delta_g\R(h)&=\div_{g}\div_{g}h-\Delta_{g}\tr_gh-\left<h,\Ric(g)\right>_g.
\end{split}
\end{equation}
Here, we have used [\eqref{first-var-lie-der-app}, Lemma \ref{lemma-app-lie-der-lin}] in the first line and the last line is simply [\eqref{lem-lin-equ-scal-first-var}, Lemma \ref{lem-lin-equ-Ric-first-var}]. By considering a suitable linear combination of the first variations described in \eqref{diff-var-for-eul-lag} and using \eqref{equ-criti-lambda-pot} leads to:
\begin{equation*}
\begin{split}
0&=\delta_g\left(2\Delta_gf_g-|\nabla^gf_g|^2_g+\R_g\right)(h)\\
&=2\Delta_{f_g}\left(\delta_gf(h)-\frac{\tr_gh}{2}\right)-\langle h,\Ric_{f_g}(g)\rangle_g\\
&\quad-2\langle \div_gh,\nabla^gf_g\rangle_g-\langle h,\nabla^{g,2}f_g\rangle_g+h(\nabla^gf_g,\nabla^gf_g)+\div_g\div_gh.
\end{split}
\end{equation*}
This ends the proof of the desired equation satisfied by $\delta_gf(h)-\frac{\tr_gh}{2}$ once we observe that:
\begin{equation*}
\div_{f_g}(\div_{f_g}h)=-2\langle \div_gh,\nabla^gf_g\rangle_g-\langle h,\nabla^{g,2}f_g\rangle_g+h(\nabla^gf_g,\nabla^gf_g)+\div_g\div_gh.
\end{equation*}

		\end{proof}
		\subsection{Further properties of $\lambda_{\operatorname{ALE}}$}~~\\
		
	We can approximate $C^{2,\alpha}$-perturbations by metrics which are Ricci flat outside a compact subset: this is the content of the following lemma that we state without proof.
	\begin{lemma}\label{cutoffC2alphatau}
		Let $(N^n,g_b)$ be an ALE Ricci-flat metric asymptotic to $\RR^n\slash\Gamma$, for some finite subgroup $\Gamma$ of $SO(n)$ acting freely on $\mathbb{S}^{n-1}$. Let $\tau\in\left(\frac{n-2}{2},n-2\right)$ and $\alpha\in (0,1)$. Let $g$ be a metric in $B_{C^{2,\alpha}_{\tau}}(g_b,\varepsilon)$ for $\varepsilon$ sufficiently small. Then, for a sequence of cut-off functions $\chi_s$ for $s>1$ vanishing in smaller and smaller neighborhoods of infinity, we have for any $\alpha'\in(0,\alpha)$ and $\tau'<\tau$,
		\begin{equation*}
		\begin{split}
		&\chi_sg + (1-\chi_s)g_b\xrightarrow[s\to +\infty]{C^{2,\alpha'}_{\tau'}} g,\\
		&\lambda_{\operatorname{ALE}}^0(\chi_sg + (1-\chi_s)g_b)-m_{\operatorname{ADM}}(\chi_sg + (1-\chi_s)g_b)\xrightarrow[s\to +\infty]{} \lambda_{\operatorname{ALE}}(g).
		\end{split}
		\end{equation*}
	\end{lemma}
	
	The following proposition sums up the scaling properties and the diffeomorphism invariance of the functional $\lambda_{\operatorname{ALE}}$: it echoes Proposition \ref{inv-diff-lambda} and Lemma \ref{scaling lambdaALE} established for $\lambda_{\operatorname{ALE}}^0$.
	\begin{prop}\label{scaling diffeo tildelambda}
		Let $(N^n,g_b)$ be an ALE Ricci-flat metric asymptotic to $\RR^n\slash\Gamma$, for some finite subgroup $\Gamma$ of $SO(n)$ acting freely on $\mathbb{S}^{n-1}$. Let $\tau\in\left(\frac{n-2}{2},n-2\right)$ and $\alpha\in (0,1)$.
		Let $g$ be a metric in $B_{C^{2,\alpha}_{\tau}}(g_b,\varepsilon)$ and consider $\phi$ a diffeomorphism close to the identity in the $C^{3,\alpha}_{\tau-1}(TN)$ topology. Then, we have $\lambda_{\operatorname{ALE}}(\phi^*g) = \lambda_{\operatorname{ALE}}(g)$.
		
		Moreover, if $s>0$,
		$$\lambda_{\operatorname{ALE}}(sg) = s^{\frac{n-2}{2}}\lambda_{\operatorname{ALE}}(g).$$
	\end{prop}
	\begin{proof}
		Let us prove the result assuming that $\R_g=0$ in a neighborhood of infinity: the general case is obtained by approximation thanks to Lemma \ref{cutoffC2alphatau}. Under this assumption, we know that $\lambda_{\operatorname{ALE}}^0$ is bounded and has the expected behavior by scaling and action of diffeomorphism thanks to Proposition \ref{inv-diff-lambda} and Lemma \ref{scaling lambdaALE}.
		
		The mass behaves the same way by rescaling. Indeed, if $g$ is ALE asymptotic to $g_e$ at infinity, then $sg$ is ALE asymptotic to $sg_e$ at infinity, and we therefore have 
		$$m_{\operatorname{ADM}}(sg) = s^{\frac{n}{2}- 1} m_{\operatorname{ADM}}(g),$$
		by studying the scaling of the operators involved.
		
		The invariance of $m_{\operatorname{ADM}}$ by $C^{3,\alpha}_{\tau-1}$-diffeomorphisms was proved in \cite{Bart-Mass}.
	\end{proof}
	
	We end this section with the proof of the monotonicity of the functional $\lambda_{\operatorname{ALE}}$ along the Ricci flow in case it stays in a neighborhood of a Ricci-flat metric. 	
		\begin{prop}\label{prop-mono-lambda}
		Let $(N^n,g_b)$ be an ALE Ricci-flat metric asymptotic to $\RR^n\slash\Gamma$, for some finite subgroup $\Gamma$ of $SO(n)$ acting freely on $\mathbb{S}^{n-1}$. Let $\tau\in\left(\frac{n-2}{2},n-2\right)$ and $\alpha\in (0,1)$ and let $(g(t))_{t\in[0,T]}$ be a solution to the Ricci flow on $N$ starting from $g(0)\in B_{C^{2,\alpha}_{\tau}}(g_b,\varepsilon)$.
		
		Then, $t\in[0,T]\rightarrow \lambda_{\operatorname{ALE}}(g(t))\in\RR$ is non-decreasing along the Ricci flow as long as $g(t)\in B_{C^{2,\alpha}_{\tau}}(g_b,\varepsilon)$ for every $t\in[0,T]$ and,
		\begin{equation}
		\frac{d}{dt}\lambda_{\operatorname{ALE}}(g(t))=2\|\Ric(g(t))+\nabla^{g(t),2}f_{g(t)}\|_{L^2(e^{-f_{g(t)}})}^2.\label{first-mono-lambda}
		\end{equation}
		Moreover, $\lambda_{\operatorname{ALE}}(g(\cdot))$ is constant in time on Ricci-flat metrics only.
		
	\end{prop}
	\begin{rk}
	Notice that Proposition \ref{prop-mono-lambda} is not a direct consequence of  Proposition \ref{lambdaALE analytic} since in general, the curve $t\in[0,T]\rightarrow g(t)\in B_{C^{2,\alpha}_{\tau}}(g_b,\varepsilon)$ induced by a solution to the Ricci flow is only $C^0$ continuous a priori when interpreted with values into the space $C^{2,\alpha}_{\tau}$.
	\end{rk}
	
	\begin{rk}
		To check that a solution to the Ricci flow stays in a neighborhood $B_{C^{2,\alpha}_{\tau}}(g_b,\varepsilon)$ of $g_b$ is a delicate problem. In a forthcoming paper, we will investigate this question in case $(N^n,g_b)$ is stable.
	\end{rk}
	\begin{proof}
	Let $R\geq R_0>0$ where $R_0$ is sufficiently large so that the level sets $\{\rho_{g_b}=R\}$ are closed smooth hypersurfaces. Then if $(g(t))_{t\in[0,T]}$ is a Ricci flow in $B_{C^{2,\alpha}_{\tau}}(g_b,\varepsilon)$, thanks to the proof of Proposition \ref{first-var-prop} by taking into account that $h(t):=-2\Ric(g(t))$ and $\varphi(t):=\delta_{g(t)}f(-2\Ric(g(t)))$,
	\begin{equation}
	\begin{split}\label{delicate-first-var}
&\int_{\{\rho_{g_b}\leq R\}}\left(|\nabla^{g(t)}f_{g(t)}|^2_{g(t)}+\R_{g(t)}\right)\,e^{-f_{g(t)}}d\mu_{g(t)}\\
&-\int_{\{\rho_{g_b}=R\}}\left<\div_{g_b}(g(t)-g_b)-\nabla^{g_b}\tr_{g_b}(g(t)-g_b),\mathbf{n}_{g_b}\right>_{g_b}\,d\sigma_{g_b}\Bigg\rvert_{t=t_1}^{t_2}=\\
&2\int_{t_1}^{t_2}\int_{\{\rho_{g_b}\leq R\}}|\Ric(g(t))+\nabla^{g(t),2}f_{g(t)}|^2_{g(t)}\,e^{-f_{g(t)}}d\mu_{g(t)}dt\\
&+\int_{t_1}^{t_2}\int_{\{\rho_{g_b}=R\}}\left<\nabla^{g(t)}\R_{g(t)}+2\Ric(g(t))(\nabla^{g(t)}f_{g(t)}),\mathbf{n}_{g(t)}\right>_{g(t)}\,e^{-f_{g(t)}}d\sigma_{g(t)}dt\\
&+\int_{t_1}^{t_2}\int_{\{\rho_{g_b}= R\}}2\left(\R_{g(t)}+\delta_{g(t)}f(-2\Ric(g(t)))\right)\langle\nabla^{g(t)}f_{g(t)},\mathbf{n}_{g(t)}\rangle_{g(t)}\,e^{-f_{g(t)}}d\sigma_{g(t)}dt\\
&+2\int_{t_1}^{t_2}\int_{\{\rho_{g_b}=R\}}\left<\div_{g_b}(\Ric(g(t))-\nabla^{g_b}\tr_{g_b}(\Ric(g(t))),\mathbf{n}_{g_b}\right>_{g_b}\,d\sigma_{g_b}dt.
\end{split}
\end{equation}
All we need to check is that the boundary integrals go to $0$ as $R$ tends to $+\infty$ uniformly in time. First, observe that by properties of the map $g\rightarrow f_g$ summarized in Proposition \ref{prop-pot-fct},
	\begin{equation}
	\begin{split}\label{easy-peasy-est}
&\left|\int_{t_1}^{t_2}\int_{\{\rho_{g_b}=R\}}\left<2\Ric(g(t))(\nabla^{g(t)}f_{g(t)}),\mathbf{n}_{g(t)}\right>_{g(t)}\,e^{-f_{g(t)}}d\sigma_{g(t)}dt\right|\leq CR^{n-2\tau-4},\\
&\left|\int_{t_1}^{t_2}\int_{\{\rho_{g_b}= R\}}2\left(\R_{g(t)}+\delta_{g(t)}f(-2\Ric(g(t)))\right)\langle\nabla^{g(t)}f_{g(t)},\mathbf{n}_{g(t)}\rangle_{g(t)}\,e^{-f_{g(t)}}d\sigma_{g(t)}dt\right|\leq CR^{n-2\tau-2},
\end{split}
\end{equation}
for some positive constant $C=C(n,g_b,\varepsilon)$. Since $\tau>\frac{n-2}{2}$, the righthand sides of (\ref{easy-peasy-est}) decay to $0$ as $R$ tends to $+\infty$. Now, the remaining boundary integrals in (\ref{delicate-first-var}) would cancel each other if the last integral was expressed in terms of the evolving metric $g(t)$ thanks to the Bianchi identity. To conclude, it is sufficient to notice that: 
\begin{equation*}
|\mathbf{n}_{g(t)}-\mathbf{n}_{g_b}|_{g_b}+\left|e^{-f_{g(t)}}\frac{d_{\sigma_{g(t)}}}{d_{\sigma_{g_b}}}-1\right|\leq C\rho^{-\tau},
\end{equation*}
for some positive constant $C=C(n,g_b,\varepsilon)$. Indeed, this implies that:
\begin{equation*}
\begin{split}
\Bigg\lvert\int_{t_1}^{t_2}&\int_{\{\rho_{g_b}=R\}}\Bigg(\left<\nabla^{g(t)}\R_{g(t)},\mathbf{n}_{g(t)}\right>_{g(t)}\,e^{-f_{g(t)}}\frac{d\sigma_{g(t)}}{d\sigma_{g_b}}\\
&+2\left<\div_{g_b}(\Ric(g(t))-\nabla^{g_b}\tr_{g_b}(\Ric(g(t))),\mathbf{n}_{g_b}\right>_{g_b}\Bigg)\,d\sigma_{g_b}dt\Bigg\rvert\leq CR^{n-2-2\tau}.
\end{split}
\end{equation*}
Here we have used that $\nabla^{g(t)}\Ric(g(t))=O(\rho_{g_b}^{-\tau-2})$ for $t>0$ by Shi's estimates \cite{Shi-Def}.

By letting $R$ tend to $+\infty$ in (\ref{delicate-first-var}) gives the expected result:
\begin{equation*}
\lambda_{\operatorname{ALE}}(g(t_2))-\lambda_{\operatorname{ALE}}(g(t_1))=2\int_{t_1}^{t_2}\int_{N}|\Ric(g(t))+\nabla^{g(t),2}f_{g(t)}|^2_{g(t)}\,e^{-f_{g(t)}}d\mu_{g(t)}dt.
\end{equation*}

	\end{proof}
	\begin{rk}
		Under our assumptions, the mass $m_{\operatorname{ADM}}$, when it is defined, is constant along the Ricci flow by \cite{Dai-Ma-Mass} (see also \cite{Li-Yu-ALE}), the variations of the functional $\lambda_{\operatorname{ALE}}$ therefore only come from those of $\lambda_{\operatorname{ALE}}^0$.
	\end{rk}
	\subsection{Local properties of stable ALE Ricci flat metric}~~\\
	
	Propositions \ref{second-var-prop} and \ref{lambdaALE analytic} justify the following notion of stability for an ALE Ricci-flat metric.
	\begin{defn}\label{defn-stable}
		An ALE Ricci-flat metric $(N^n,g_b)$ asymptotic to $\RR^n\slash\Gamma$, for some finite subgroup $\Gamma$ of $SO(n)$ acting freely on $\mathbb{S}^{n-1}$ is said to be \emph{linearly stable} if the second variation of $\lambda_{\operatorname{ALE}}$ at $g_b$ along a divergence-free variation $h\in C^{2,\alpha}_{\tau}(S^2T^*N)$, $\tau\in\left(\frac{n-2}{2},n-2\right)$, $\alpha\in(0,1)$, is nonpositive, i.e. if $L_{g_b}$ is a nonpositive operator in the $L^2$ sense when restricted to divergence-free variations in $ C^{2,\alpha}_{\tau}(S^2T^*N)$.
	\end{defn}
	
	Definition \ref{defn-stable} is relevant with respect to the functional $\lambda_{\operatorname{ALE}}$ but it has the apparent disadvantage that it depends on a choice of parameters $(\tau,\alpha)\in\left(\frac{n-2}{2},n-2\right)\times(0,1)$ a priori. The following lemma shows that Definition \ref{defn-stable} is actually independent of this choice of parameters.
	
	\begin{lemma}\label{lemma-equiv-def-stable}
	Let $(N^n,g_b)$, $n\geq 4$, be an ALE Ricci-flat metric. Then the following assertions are equivalent:
	\begin{enumerate}
	\item \label{first-def-sta}$(N^n,g_b)$ is linearly stable in the sense of Definition \ref{defn-stable}.\\
	
	\item \label{sec-def-sta}$\left<-L_{g_b}h,h\right>_{L^2}\geq 0$ for all $h\in C_c^{\infty}(S^2T^*N)$.\\
	
	\item \label{thir-def-sta}$\left<-L_{g_b}h,h\right>_{L^2}\geq 0$ for all $h\in H^2_{\frac{n}{2}-1}(S^2T^*N)$.
	\end{enumerate} 
	\end{lemma}
	
	\begin{proof}
	We proceed by proving the implications $(\ref{first-def-sta})\Rightarrow (\ref{sec-def-sta})\Rightarrow (\ref{thir-def-sta})\Rightarrow (\ref{first-def-sta}).$
	
	Notice first that we have the following inclusions $$C^{\infty}_c(S^2T^*N)\subset C^{2,\alpha}_{\tau}(S^2T^*N)\subset H^2_{\frac{n}{2}-1}(S^2T^*N),$$ for any $\tau>\frac{n}{2}-1$ and $\alpha\in(0,1)$ according to Remark \ref{sobolev embeddings}. Moreover, $C_c^{\infty}(S^2T^*N)$ is dense in $H^2_{\frac{n}{2}-1}(S^2T^*N)$. In particular, the implications $(\ref{sec-def-sta})\Rightarrow (\ref{thir-def-sta})\Rightarrow (\ref{first-def-sta})$ are straightforward.
	
	We claim that if $(\ref{first-def-sta})$ holds true then 
	\begin{equation}
	\left<-L_{g_b}h,h\right>_{L^2}\geq 0,\quad\text{for all $h\in C^{2,\alpha}_{\tau}(S^2T^*N)$.}\label{cond-1-bis-lin-sta}
	\end{equation}
	  Taken (\ref{cond-1-bis-lin-sta}) for granted, it is immediate to conclude the proof of the implication $(\ref{first-def-sta})\Rightarrow (\ref{sec-def-sta}).$
	Therefore, all is left to prove is (\ref{cond-1-bis-lin-sta}) under Condition $(\ref{first-def-sta})$. By Proposition \ref{prop-decomp-2-tensor}, if $h$ is a symmetric $2$-tensor in $C^{2,\alpha}_{\tau}$, $\beta:=\tau\in\left(\frac{n-2}{2},n-2\right)\subset(1,n-1)$ then there exist a symmetric $2$-tensor $h'$ in $C^{2,\alpha}_{\tau}$ and a vector field $X\in C^{3,\alpha}_{\tau-1}$ such that $h=h'+\Li_X(g_b)$. Now, observe by bilinearity that:
	\begin{equation}
\begin{split}\label{easy-quad-form}
\left<-L_{g_b}h,h\right>_{L^2}=\,&\left<-L_{g_b}h',h'\right>_{L^2}+\left<-L_{g_b}h',\Li_X(g_b)\right>_{L^2}\\
&+\left<-L_{g_b}\Li_X(g_b),h'\right>_{L^2}+\left<-L_{g_b}\Li_X(g_b),\Li_X(g_b)\right>_{L^2}.
\end{split}
\end{equation}
On the one hand, an integration by parts shows that:
\begin{equation}\label{easy-IBP}
\left<L_{g_b}h',\Li_X(g_b)\right>_{L^2}=\left<h',L_{g_b}\Li_X(g_b)\right>_{L^2}.
\end{equation}
Here we have used the fact that $2\tau+1>n-1$ to handle the boundary term in the integration by parts.

On the other hand, by using the flow $(\phi^X_t)_{t\in \RR}$ generated by the vector field $X$, one observes that the one-parameter family $((\phi_t^X)^*g_b)_{t\in \RR}$ is a curve of Ricci-flat metrics. In particular, by differentiating the Ricci-flat equation at $t=0$ with the help of Lemma \ref{lem-lin-equ-Ric-first-var}, 
\begin{equation}\label{identif-im-lie-der}
L_{g_b}\Li_X(g_b)=\Li_{B^X}(g_b),\quad B^X:=B_{g_b}(\Li_X(g_b)).
\end{equation}

Plugging (\ref{identif-im-lie-der}) in (\ref{easy-IBP}) leads to:
\begin{equation}
\begin{split}
\left<L_{g_b}h',\Li_X(g_b)\right>_{L^2}=\,&\left<h',\Li_{B^X}(g_b)\right>_{L^2}\\
=\,&-2\left<\div_{g_b}h',B^X\right>_{L^2}=0,\label{delicate-orth-decom}
\end{split}
\end{equation}
since by definition, $h'$ is divergence-free. Here again, the integration by parts is legitimated by the fact that $B^X=O(\rho_{g_b}^{-\tau-1})$. 

Going back to (\ref{easy-quad-form}), the vanishing (\ref{delicate-orth-decom}) shows that it is sufficient to prove that 
\begin{equation}
\left<-L_{g_b}\Li_X(g_b),\Li_X(g_b)\right>_{L^2}\geq 0,\label{last-cond-fulfill}
\end{equation}
 since Condition (\ref{first-def-sta}) is assumed to hold. Because of (\ref{identif-im-lie-der}), it is equivalent to check the following:
\begin{equation}
\begin{split}\label{IBP-Bianchi-lemma-loc-stab}
-\left<\Li_X(g_b),\Li_{B^X}(g_b)\right>_{L^2}=\,&2\left<\div_{g_b}\Li_X(g_b),B^X\right>_{L^2}\\
=\,&2\|B^X\|_{L^2}^2+\left<\nabla^{g_b}\tr_{g_b}\Li_X(g_b),B^X\right>_{L^2}\\
=\,&2\|B^X\|_{L^2}^2-\left<\tr_{g_b}\Li_X(g_b),\div_{g_b}B^X\right>_{L^2}\\
=\,&2\|B^X\|_{L^2}^2+\frac{1}{2}\left<\tr_{g_b}\Li_X(g_b),-\Delta_{g_b}\Li_{X}(g_b)\right>_{L^2}\\
=\,&2\|B^X\|_{L^2}^2+\frac{1}{2}\|\nabla^{g_b}\tr_{g_b}\Li_{X}(g_b)\|_{L^2}^2\geq 0.
\end{split}
\end{equation}
Here, we have integrated by parts in the first, third and last lines. The second line uses the definition of $B^X$ given in (\ref{identif-im-lie-der}) only. Taking into account the Ricci-flatness of $g_b$, the penultimate line is obtained by considering the trace of (\ref{identif-im-lie-der}) with respect to the metric $g_b$. This concludes the proof of (\ref{last-cond-fulfill}).
	\end{proof}
	We end this section with the following strong positivity property shared by the Lichnerowicz operator associated to a stable ALE Ricci-flat metric $(N^n,g_b)$: this result established in \cite[Theorem $3.9$]{Der-Kro} has been proved to be useful for the dynamical stability of integrable Ricci-flat ALE metrics. The proof of this result is essentially due to Devyver \cite{Dev-Gau-Est} in a more general setting. 
	
	\begin{theo}\label{theo-der-kro}
		Let $(N^n,g_b)$ be a linearly stable ALE Ricci-flat metric. Then there exists some positive constant $\varepsilon(g_b)\in[0,1)$ such that 
		\begin{equation*}
		(1-\varepsilon(g_b))\left<-\Delta_{g_b}h,h\right>_{g_b}\leq \left<-L_{g_b}h,h\right>_{g_b}, 
		\end{equation*}
		for all $h\in H^2_{\frac{n}{2}-1}$ which is $L^2(g_b)$-orthogonal to $\ker_{L^2(g_b)}L_{g_b}.$
	\end{theo}

	\section{Energy estimates on the potential function}\label{sec-ene-est-pot-fct}
	If $(N^n,g_b)$ is an ALE Ricci-flat metric, we establish energy estimates on the gradient and the (weighted) laplacian of the potential function $f_g$ associated to a metric $g\in B_{C^{2,\alpha}_{\tau}}(g_b,\varepsilon)$ in terms of the norm $\|g-g_b\|_{H^2_{\frac{n}{2}-1}}$. These estimates will be crucially used in the proof of Proposition \ref{prop-energy-est}.
	
	\begin{prop}\label{prop-ene-pot-fct}
		Let $(N^n,g_b)$ be an ALE Ricci-flat metric asymptotic to $\RR^n\slash\Gamma$, for some finite subgroup $\Gamma$ of $SO(n)$ acting freely on $\mathbb{S}^{n-1}$. Let $\tau\in\left(\frac{n-2}{2},n-2\right)$ and $\alpha\in (0,1)$. Then there exists a neighborhood $B_{C^{2,\alpha}_{\tau}}(g_b,\varepsilon)$ of $g_b$ such that the following energy estimates on the potential function $v_g:=e^{-f_g}-1$ and its first variation hold true:
		\begin{equation}
		\|\nabla^gv_g\|_{L^2}\leq C(n,g_b,\varepsilon)\|\nabla^{g_b}(g-g_b)\|_{L^2}\label{grad-est-int-ene},
		\end{equation}
		and, if $g_1$ and $g_2$ are metrics in $B_{C^{2,\alpha}_{\tau}}(g_b,\varepsilon)$,
		\begin{equation}
		\left\|\nabla^{g_t}\left(\delta_{g_t}f(h)-\frac{\tr_{g_t}h}{2}\right)\right\|_{L^2}\leq C(n,g_b,\varepsilon)\|\nabla^{g_t}h\|_{L^2},\label{est-grad-first-der-pot-fct-ene}
		\end{equation}
		where $g_t:=g_1+(t-1)h:=g_1+(t-1)(g_2-g_1)$ for $t\in[1,2]$.
		
		Moreover,
		\begin{equation}\label{first-var-pot-fct-ell-equ-gal-case-prop}
		\left\|\Delta_{g_t,f_{g_t}}\left(\delta_{g_t}f(h)-\frac{\tr_{g_t}h}{2}\right)\right\|_{L^2_{\frac{n}{2}+1}}\leq C(n,g_b,\varepsilon)\left(\|\nabla^{g_t}h\|_{L^2}+\|\nabla^{g_t,2}h\|_{L^2_{\frac{n}{2}+1}}\right),\quad t\in[1,2].
		\end{equation}

	\end{prop}
	
	\begin{proof}
		Let us remark first that if $v_g:=w_g-1=e^{-f_g}-1$ then:
		\begin{equation*}
		\begin{split}
		\Delta_g(w_g-1)^2&=\Delta_gv_g^2=2|\nabla^gv_g|^2+2\Delta_gv_g\cdot v_g\\
		&=2|\nabla^gv_g|^2+\frac{1}{2}\R_gw_g\cdot v_g\\
		&=2|\nabla^gv_g|^2+\frac{1}{2}\R_gv_g^2+\frac{1}{2}\R_gv_g.
		\end{split}
		\end{equation*}
		Integrating by parts the previous identity gives:
		\begin{eqnarray}
		2\|\nabla^gv_g\|^2_{L^2}&\leq&\int_N|\R_g|v_g^2\,d\mu_g-\frac{1}{2}\int_N\R_gv_g\,d\mu_g.\label{first-int-est-nabla-v}
		\end{eqnarray}
		Notice that the last term on the righthand side is kept unchanged for the following reason: if $n=4$, $v_g$ or equivalently, $f_g$, is not in $L^2$ so one needs to proceed in a more subtle way than just using Young's inequality. By tracing Lemma \ref{Ric-lin-lemma-app} together with [(\ref{lem-lin-equ-scal-first-var}), Lemma \ref{lem-lin-equ-Ric-first-var}], recall that pointwise:
		\begin{equation*}
		\left|-\R_g+\div_{g_b}(\div_{g_b}h)-\Delta_{g_b}\tr_{g_b}h\right|\leq C(n,g_b,\varepsilon)\left(|\nabla^{g_b}h|_{g_b}^2+|h|_{g_b}|\nabla^{g_b,2}h|_{g_b}\right).
		\end{equation*}
		
		In particular, by integrating by parts, if $\gamma>0$,
		\begin{equation}
		\begin{split}
		\left|\int_N\R_gv_g\,d\mu_g\right|&\leq C\left(\int_N|\nabla^{g_b}h|_{g_b}|\nabla^{g_b}v_g|_{g_b}\,d\mu_g\right)+C\int_N\left(|\nabla^{g_b}h|^2_{g_b}+|h|_{g_b}|\nabla^{g_b,2}h|_{g_b}\right)|v_g|d\mu_g\\
		&\leq\gamma\|\nabla^{g_b}v_g\|_{L^2}^2+C\left(\gamma^{-1}\|\nabla^{g_b}h\|_{L^2}^2+\|\rho_{g_b}^{-1}h\|_{L^2}\|\rho_{g_b}\cdot v_g\cdot |\nabla^{g_b,2}h|\|_{L^2}\right)\\
		&\leq \gamma\|\nabla^{g_b}v_g\|_{L^2}^2+C\left(\gamma^{-1}\|\nabla^{g_b}h\|_{L^2}^2+\|\rho_{g_b}^{-1}h\|_{L^2}\|\rho_{g_b}^{-1}\cdot v_g\|_{L^2}\|\rho_{g_b}^2\nabla^{g_b,2}h\|_{C^0}\right)\\
		&\leq \gamma\left(\|\rho_{g_b}^{-1}\cdot v_g\|_{L^2}^2+\|\nabla^{g_b}v_g\|_{L^2}^2\right)+C\gamma^{-1}\left(\|\nabla^{g_b}h\|_{L^2}^2+\|\rho_{g_b}^{-1}h\|_{L^2}^2\right),
		\label{sec-int-est-nabla-v}
		\end{split}
		\end{equation}
		where $C=C(n,g_b,\varepsilon)$ is a positive constant that may vary from line to line, and where we used the inequality $2\|\rho_{g_b}^{-1}h\|_{L^2}\|\rho_{g_b}^{-1}\cdot v_g\|_{L^2}\leq \gamma\|\rho_{g_b}^{-1}\cdot v_g\|_{L^2}^2+\gamma^{-1}\|\rho_{g_b}^{-1}h\|_{L^2}^2$. Here we have used Cauchy-Schwarz inequality together with the fact that $\nabla^{g_b,2}h$ decays at least quadratically since $\|h\|_{C^2_0}$ is finite. To sum it up, (\ref{first-int-est-nabla-v}) and (\ref{sec-int-est-nabla-v}) give by Hardy's inequality for any $\gamma>0$:
		\begin{equation*}
		\begin{split}
		\|\nabla^gv_g\|^2_{L^2}&\leq\int_N|\R_g|v_g^2\,d\mu_g+C\gamma\|\nabla^gv_g\|_{L^2}^2+C\gamma^{-1}\|\nabla^{g_b}h\|_{L^2}^2.
		\end{split}
		\end{equation*}
		By choosing $\gamma$ small enough together with the fact that $\sup_{N}\rho^2_{g_b}|\R_g|$ can be made arbitrarily small by shrinking $B_{C^{2,\alpha}_{\tau}}(g_b,\varepsilon)$ if necessary, Hardy's inequality yields:
		\begin{equation}
		\begin{split}\label{est-grad-pot-fct-ene}
		\|\nabla^gv_g\|^2_{L^2}&\leq\frac{1}{2}\|\nabla^gv_g\|^2_{L^2}+C(n,g_b,\varepsilon,\gamma)\|\nabla^{g_b}h\|_{L^2}^2,
		\end{split}
		\end{equation}
		i.e. one gets (\ref{grad-est-int-ene}) as expected.
		\\
		
		Now, let us turn to the proof of (\ref{est-grad-first-der-pot-fct-ene}).\\
		
		We invoke the variation of equation (\ref{equ-criti-lambda-pot}) satisfied by the potential function $f_g$ established in [\eqref{var-vol-var-ell-eqn-for}, Proposition \ref{var-vol-var-ell-eqn-prop}].
		
		Consequently, $\delta_{g_t}f(h)$ satisfies for $t\in[1,2]$:
		\begin{equation}
		\begin{split}\label{first-var-pot-fct-ell-equ-gal-case}
		\Delta_{f_{g_t}}\left(\frac{\tr_{g_t}h}{2}-\delta_{g_t}f(h)\right)&=\frac{1}{2}\left(\div_{f_{g_t}}(\div_{f_{g_t}}h)-\langle h,\Ric_{f_{g_t}}(g_t)\rangle_{g_t}\right).
		\end{split}
		\end{equation}
		Now, thanks to Proposition \ref{prop-pot-fct} and the fact that $g\in B_{C^{2,\alpha}_{\tau}}(g_b,\varepsilon)$, 
		\begin{equation}
		\rho_{g_b}|\nabla^{g_t}f_{g_t}|_{g_t}+\rho^2_{g_b}\left(|\Ric(g_t)|_{g_t}+|\nabla^{g_t,2}f_{g_t}|_{g_t}\right)\leq C(n,g_b,\varepsilon).\label{est-pot-fct-ene-est}
		\end{equation}

		Once we multiply the previous elliptic equation (\ref{first-var-pot-fct-ell-equ-gal-case}) by $\delta_{g_t}f(h)-\frac{\tr_{g_t}h}{2}$, let us integrate by parts as follows:
		\begin{equation}
		\begin{split}\label{grad-first-var-pot-fct-ene-est}
		\left\|\nabla^{g_t}\left(\delta_{g_t}f(h)-\frac{\tr_{g_t}h}{2}\right)\right\|^2_{L^2(e^{-f_{g_t}}d\mu_{g_t})}&=-\frac{1}{2}\int_N\left\langle\nabla^{g_t}\left(\delta_{g_t}f(h)-\frac{\tr_{g_t}h}{2}\right),\div_{f_{g_t}} h\right\rangle_{g_t}\,e^{-f_{g_t}}d\mu_{g_t}\\
		&\quad-\frac{1}{2}\int_N\left(\delta_{g_t}f(h)-\frac{\tr_{g_t}h}{2}\right)\left\langle  h,\Ric_{f_{g_t}}(g_t)\right\rangle_{g_t}\,e^{-f_{g_t}}d\mu_{g_t}\\
		&=:I_1+I_2.
		\end{split}
		\end{equation}
		
		The first integrals $I_1$ on the righthand side of the previous computation can be handled as follows:
		\begin{equation}
		\begin{split}\label{I_1-est}
		|I_1|&\leq \frac{1}{4}\left\|\nabla^{g_t}\left(\delta_{g_t}f(h)-\frac{\tr_{g_t}h}{2}\right)\right\|^2_{L^2(e^{-f_{g_t}}d\mu_{g_t})}+C\|\div_{f_{g_t}}h\|^2_{L^2(d\mu_{g_t})}\\
		&\leq \frac{1}{4}\left\|\nabla^{g_t}\left(\delta_{g_t}f(h)-\frac{\tr_{g_t}h}{2}\right)\right\|^2_{L^2(e^{-f_{g_t}}d\mu_{g_t})}\\
		&\quad+C(n,g_b,\varepsilon)\left(\|\nabla^{g_t}h\|^2_{L^2(d\mu_{g_t})}+\|\rho_{g_b}^{-1}|h|_{g_b}\|_{L^2(d\mu_{g_t})}^2\right)\\
		&\leq \frac{1}{4}\left\|\nabla^{g_t}\left(\delta_{g_t}f(h)-\frac{\tr_{g_t}h}{2}\right)\right\|^2_{L^2(e^{-f_{g_t}}d\mu_{g_t})}+C(n,g_b,\varepsilon)\|\nabla^{g_t}h\|^2_{L^2(d\mu_{g_t})},
		\end{split}
		\end{equation}
		where we have used Young's inequality in the first line. The second inequality follows from (\ref{est-pot-fct-ene-est}) and the third inequality uses Hardy's inequality from Theorem \ref{thm-min-har-inequ} together with the fact that $g_t$ (and therefore $f_{g_t}$ by Proposition \ref{prop-pot-fct}) is arbitrary close to $0$ in the $C^{2,\alpha}_{\tau}$-topology.
		
		The integral $I_2$ can be estimated from above in a similar way for any $t\in[1,2]$ and $\gamma\in(0,1)$:
		\begin{equation}
		\begin{split}\label{I_2-est}
		|I_2|&\leq C(n,g_b,\varepsilon)\int_N\rho_{g_b}^{-2}\left|\delta_{g_t}f(h)-\frac{\tr_{g_t}h}{2}\right||h|_{g_t}\,d\mu_{g_t}\\
		&\leq \gamma \int_N\rho_{g_b}^{-2}\left|\delta_{g_t}f(h)-\frac{\tr_{g_t}h}{2}\right|^2\,d\mu_{g_t}+C(\gamma,n,g_b,\varepsilon)\int_N\rho_{g_b}^{-2}|h|^2_{g_t}\,d\mu_{g_t}\\
		&\leq \frac{1}{4}\left\|\nabla^{g_t}\left(\delta_{g_t}f(h)-\frac{\tr_{g_t}h}{2}\right)\right\|^2_{L^2(e^{-f_{g_t}}d\mu_{g_t})}+C(\gamma,n,g_b,\varepsilon)\|\nabla^{g_t}h\|_{L^2(d\mu_{g_t})},
		\end{split}
		\end{equation}
		if $\gamma=\gamma(n,g_b,\varepsilon)$ is chosen sufficiently small. Here, we have used Hardy's inequality in the last inequality.
		
		Putting (\ref{grad-first-var-pot-fct-ene-est}), (\ref{I_1-est}) and (\ref{I_2-est}) altogether lead to the expected estimate (\ref{est-grad-first-der-pot-fct-ene}).

		By considering the $L^2_{\frac{n}{2}+1}$ norm of (\ref{first-var-pot-fct-ell-equ-gal-case}), (\ref{first-var-pot-fct-ell-equ-gal-case-prop}) follows by using Hardy's inequality and (\ref{est-pot-fct-ene-est}).
	\end{proof}
	
	We conclude this section by stating a quantitative version of Proposition \ref{prop-ene-pot-fct} whose proof is very similar and is therefore omitted:
	\begin{prop}\label{prop-ene-pot-fct-bis}
		Let $(N^n,g_b)$ be an ALE Ricci-flat metric asymptotic to $\RR^n\slash\Gamma$, for some finite subgroup $\Gamma$ of $SO(n)$ acting freely on $\mathbb{S}^{n-1}$. Let $\tau\in\left(\frac{n-2}{2},n-2\right)$ and $\alpha\in (0,1)$. Then there exists a neighborhood $B_{C^{2,\alpha}_{\tau}}(g_b,\varepsilon)$ of $g_b$ such that if $g_1$ and $g_2$ are metrics in $B_{C^{2,\alpha}_{\tau}}(g_b,\varepsilon)$ and $h\in C^{2,\alpha}_{\tau}$,
		\begin{equation}
		\begin{split}
		\left\|\nabla^{g_1}\left(\delta_{g_2}f(h)-\delta_{g_1}f(h)\right)\right\|_{L^2}&\leq C(n,g_b,\varepsilon)\|g_2-g_1\|_{C^{2,\alpha}_{\tau}}\|h\|_{H^2_{\frac{n}{2}-1}},\\
		\left\|\Delta_{g_1,f_{g_1}}\left(\delta_{g_2}f(h)-\delta_{g_1}f(h)\right)\right\|_{L^2_{\frac{n}{2}+1}}&\leq C(n,g_b,\varepsilon)\|g_2-g_1\|_{C^{2,\alpha}_{\tau}}\|h\|_{H^2_{\frac{n}{2}-1}}.
		\end{split}
		\end{equation}
			\end{prop}

	\section{Fredholm properties of the Lichnerowicz operator in weighted spaces}\label{fred-sec-prop}
	
	Another motivation for our weighted H\"older spaces is that they ensure that the Lichnerowicz operator which controls the second variations of $\lambda_{\operatorname{ALE}}$ has adequate Fredholm properties. Indeed, the Lichnerowicz operator $L_{g_b}:H^{2}(S^2T^*N)\rightarrow L^2(S^2T^*N)$ is symmetric and bounded, but is \emph{not} Fredholm and does not have satisfying analytical properties. This becomes the case when considering weighted spaces.

	\begin{prop}[Fredholm properties of the Lichnerowicz operator] \label{prop-lic-fred}
		Let $(N^n,g_b)$ be an ALE Ricci-flat metric, asymptotic to $\RR^n\slash\Gamma$, for some finite subgroup $\Gamma\neq \{\textup{Id}\}$ of $SO(n)$ acting freely on $\mathbb{S}^{n-1}$.  If $\beta\in(0,n-2)\cup(n-2,n)$, then $$L_{g_b}:C^{2,\alpha}_{\beta}(S^2T^*N)\rightarrow C^{0,\alpha}_{\beta+2}(S^2T^*N)$$
		is Fredholm for every $\alpha\in(0,1)$, and
		$$L_{g_b}:H^2_{\beta}(S^2T^*N)\rightarrow L^2_{\beta+2}(S^2T^*N)$$
		is Fredholm.
		
		Moreover, if a $2$-tensor $h\in C^{2,\alpha}_{\beta}(S^2T^*N)$  is in the kernel of $L_{g_b}$, then $h\in C^{\infty}_{n}(S^2T^*N)$ and is divergence-free.
		
		In dimension $n\geq 3$, the operator $L_{g_b}:H^2_{\frac{n}{2}-1}(S^2T^*N)\rightarrow L^2_{\frac{n}{2}+1}(S^2T^*N)$ is Fredholm and both its kernel and $L^2$-cokernel equal $\ker_{L^2}L_{g_b}$. As a consequence, there exists $C>0$ such that for any $h\perp_{L^2} \ker_{L^2}L_{g_b}$, we have the following control:
		\begin{equation}
		\|\nabla^{g_b, 2}h\|_{L^2_{\frac{n}{2} +1}}\leq\|h\|_{H^2_{\frac{n}{2}-1}} \leq C \|L_{g_b}h\|_{L^2_{\frac{n}{2}+1}}.\label{controle hessienne lic n}
		\end{equation}	
	\end{prop}
	\begin{rk}
		The statement holds for $\mathbb{R}^n$ if we additionally assume $\beta<n-1$.
	\end{rk}
	\begin{proof}
		The fact that the two operators are Fredholm is a consequence of the theory of elliptic operators between weighted Hölder spaces. Indeed, the elements of the kernel of the operator $-\nabla^*\nabla$ on a flat nontrivial quotient of $\mathbb{R}^n$ are sums of homogeneous $2$-tensors of order $k$ or $-n+2-k$ for $k\in \mathbb{N}\backslash\{1\}$ where $1$ is not in the set of possible values because there are no nonvanishing linear functions on $\mathbb{R}^n$ invariant by the group action induced by $\Gamma$.\\
		
		Let us now consider a $2$-tensor $h\in C^{2,\alpha}_{\beta}(S^2T^*N)$ which satisfies $L_{g_b}h = 0$, then, since $\div_{g_b}L_{g_b}=\frac{1}{2}(\nabla^{g_b})^{\ast}\nabla^{g_b}\div_{g_b}$, by the maximum principle, we have $\div_{g_b}h=0$.  At infinity, we have $h = H^{n-2} + O(\rho_{g_b}^{-n+2-\epsilon})$ for a harmonic homogeneous $2$-tensor $H^{n-2}\sim \rho_{g_b}^{-n+2}$. Now, such a divergence-free $2$-tensor $H^{n-2}$ must vanish by \cite[Lemma 4.1]{ozu2}. We therefore have $ h =  O(\rho_{g_b}^{-n+2-\epsilon}) $, and since the next decay rate in the kernel of $-\nabla^*\nabla$ is $\rho_{g_b}^{-n}$, we have $h\in C^{\infty}_{n}(S^2T^*N)$.
		\\
		
		For the weighted Sobolev spaces, the operator $$L_{g_b}:H^2_{\frac{n}{2}-1}(S^2T^*N)\rightarrow L^2_{\frac{n}{2}+1}(S^2T^*N)$$ whose kernel is the kernel of $L_{g_b}$ on $H^2_{\frac{n}{2}-1}(S^2T^*N)$ is reduced to the $L^2$-kernel of $L_{g_b}$ because there is no exceptional value between $\frac{n}{2}-1$ and $n-2$. Its cokernel is the kernel of $L_{g_b}$ on $L^2_{\frac{n}{2}-1}(S^2T^*N)\approx \left(L^2_{\frac{n}{2}+1}(S^2T^*N)\right)^*$ (see Note \ref{note L2 cokernel} below) which is also equal to $\ker_{L^2}L_{g_b}$. This operator is in particular of index $0$. By Banach bounded inverse theorem, this implies that there exists $C>0$ depending on $g_b$ such that for any $h\perp\ker_{L^2}L_{g_b}$, we have
		$$\|\nabla^{g_b,2}h\|_{L^2_{\frac{n}{2}+1}}\leq \|h\|_{H^2_{\frac{n}{2}-1}}\leq C \|L_{g_b} h \|_{L^2_{\frac{n}{2}+1}}.$$
		Note that the $L^2$-product is well-defined between elements of $L^2_{\frac{n}{2}-1}(S^2T^*N)$ and elements of $\ker_{L^2}L_{g_b} \subset C^0_n(S^2T^*N)\subset L^2_{\frac{n}{2}+1}(S^2T^*N)$.
	\end{proof}
	
	\begin{note}\label{note L2 cokernel}
		For any $s\in \mathbb{R}$, the dual of $ L^2_{\frac{n}{2}+s} $ classically identifies with $ L^2_{\frac{n}{2}-s} $ because by definition, this is the set of tensors for which the $L^2$-product is defined for any $2$-tensor in $ L^2_{\frac{n}{2}+s} $. We therefore define the $L^2$-cokernel of a symmetric operator $H^2_{\frac{n}{2}+s-2}\mapsto L^2_{\frac{n}{2}+s}$ as its kernel on the dual of its image: $L^2_{\frac{n}{2}-s}$. This is the $L^2$-orthogonal of its image. 
		
		We keep this definition on subsets of $ L^2_{\frac{n}{2}+s}$, for instance, consider a smooth elliptic operator $L$ asymptotic to the Euclidean Laplacian at infinity between $ C^{2,\alpha}_{\frac{n}{2}+s-2} $ and $ C^{0,\alpha}_{\frac{n}{2}+s} \subset L^2_{\frac{n}{2}+s-\epsilon}$ for all $\epsilon>0$. Assume moreover that ${\frac{n}{2}+s}$ is not a critical exponent of the Laplacian. Then, the $L^2$-cokernel of $L$ is the kernel of $L$ on $C^{0,\alpha}_{\frac{n}{2}-s}$. Indeed, by the above discussion, we first identify it with the kernel on $L^2_{\frac{n}{2}-s+\epsilon}$ which by elliptic regularity is the kernel on $H^k_{\frac{n}{2}-s+\epsilon}$ for all $k$. For $k$ large enough, this embeds in $C^{0,\alpha}_{\frac{n}{2}-s+\epsilon}$, and finally, by choosing $\epsilon$ small enough so that there is no critical exponent of the Laplacian in $[\frac{n}{2}-s,\frac{n}{2}-s+\epsilon] $, this is also the kernel of $L$ on $C^{0,\alpha}_{\frac{n}{2}-s}$.
	\end{note}
	
	
	\section{Properties of $\lambda_{\operatorname{ALE}}$ in the integrable case}\label{loj-sim-sec-int-case}
	Now that we have proved that the functional is well-defined and well-behaved in relevant function spaces, we start by investigating the stable integrable case (which corresponds to all known examples) first where a \L{}ojasiewicz inequality can be proved ``by hand".
	
	\subsection{Integrability of Ricci-flat ALE metrics}\label{sec-int-ric-fla}

	We say that a Ricci-flat ALE metric $(N^n,g_b)$ is integrable if the moduli space of Ricci-flat ALE metric on $N$ with the same cone at infinity is a smooth manifold around $g_b$.
	\begin{defn}[Integrable Ricci-flat ALE metric]\label{definition integrable}
		A Ricci-flat ALE metric $(N^n,g_b)$ is \emph{integrable} if for all $v\in\ker_{L^2}L_{g_b}$ small enough, there exists a (unique) Ricci-flat ALE metric $\Bar{g}_v$ satisfying $\Bar{g}_v-(g_b+v)\perp \ker_{L^2}L_{g_b}$, and such that $\div_{g_b}\Bar{g}_v=0$ and $\|\Bar{g}_v-g_b\|_{C^{2,\alpha}_{n}}\leq 2 \|v\|_{C^{2,\alpha}_{n}}$.
	\end{defn}
	
	We will need the following description of the neighborhood of an \emph{integrable} Ricci-flat ALE metric to restrict ourselves to deformations which are transverse to the Ricci-flat deformations.
	\begin{prop}\label{gauge fixing ALE integrable}
		Let $n\geqslant 4$ and $(N^n,g_b)$ be an integrable ALE Ricci-flat metric, asymptotic to $\RR^n\slash\Gamma$, for some finite subgroup $\Gamma$ of $SO(n)$ acting freely on $\mathbb{S}^{n-1}$. Let $\tau\in(1,n)$ and $\alpha\in(0,1)$. Then there exist $C>0$ and $\varepsilon>0$ such that for any metric $g$ satisfying $ \|g-g_b\|_{C^{2,\alpha}_\tau(g_b)}\leq \epsilon $, there exists a Ricci-flat ALE metric $ g'_b $ such that
		\begin{itemize}
			\item $\| g_b - g'_b \|_{C^{2,\alpha}_n(g_b)}\leq C \|g-g_b\|_{C^{2,\alpha}_\tau(g_b)}$,
			\item $ g - g'_b \perp_{L^2(g'_b)}\ker_{L^2(g'_b)}L_{g'_b} $, and
			\item $\div_{g'_b}g = 0$.
		\end{itemize}
	\end{prop}
	\begin{proof}
		According to \cite[Corollary 5.16]{ozu2} and Definition \ref{definition integrable}, the integrability assumption rewrites in the following way. For any $v\in\ker_{L^2(g_b)}L_{g_b}$, there exists a unique Ricci-flat ALE metric $\bar{g}_v$ satisfying
		\begin{enumerate}
			\item $\div_{g_b}\Bar{g}_v=0$,
			\item $\Bar{g}_v-(g_b+v)\perp \ker_{L^2}L_{g_b}$,
			\item $\Ric(\bar{g}_v) = 0$.
		\end{enumerate}
		Moreover, since these metrics are obtained from the implicit function theorem, Lemma \ref{th fcts implicites}, applied to $(v,g)\mapsto \Ric(g) + \frac{1}{2}\mathcal{L}_{\div_{g_b}\Bar{g}_v}g$ seen as an operator from $(\ker_{L^2}L_{g_b})\times C^{2,\alpha}_\tau$ to $C^{0,\alpha}_{\tau+2}$ for $\tau\in (\frac{n-2}{2},n-2)$, they consequently vary analytically in $v$. Similarly, the elements of $\ker_{L^2(\bar{g}_v)}L_{\bar{g}_v}$ vary analytically in $v$ as solutions $h\in L^2(g_b)$ of the parametrized equation $(v,h)\mapsto L_{\bar{g}_v} h = 0$.
		
		According to Proposition \ref{prop-gauge-div-free} which is proved by implicit function theorem, for any $\alpha\in(0,1)$, there exists $\varepsilon>0$ for which for any metric $g\in B_{C^{2,\alpha}_\tau}(g_b,\varepsilon)$ and any $v\in \ker_{L^2}L_{g_b}$, there exists a unique vector field $X(g,v)\in C^{2+1,\alpha}_{\tau-1}(TN)$ depending analytically on $g$ and $v$ for which
		$$\div_{(\exp_{X(g,v)})_*\bar{g}_v}g =0,$$
		where $\exp_{X}g:x \mapsto \exp_x^{g}(X(x))$. Define $\phi_{g,v} := (\exp_{X(g,v)})^{-1}$. We will naturally look for a $2$-tensor $v$ such that $g - \phi_{g,v}^*\bar{g}_v\perp \ker_{L^2(\phi_{g,v}^*\bar{g}_v)}L_{\phi_{g,v}^*\bar{g}_v}$, where the $L^2$-scalar product is defined with respect to $\phi_{g,v}^*\bar{g}_v$ and such that the elements of $\ker_{L^2(\phi_{g,v}^*\bar{g}_v)}L_{\phi_{g,v}^*\bar{g}_v}$ vary analytically in $v$. 
		
		Let us therefore consider the analytic map
		$$F:(g,v)\mapsto \pi^{\phi_{g,v}^*\bar{g}_v}\big(\phi_{g,v}^*\bar{g}_v-g\big),$$
		where $\pi^{\phi_{g,v}^*\bar{g}_v}$ is the $L^2(\phi_{g,v}^*\bar{g}_v)$-orthogonal projection on $\ker_{L^2(\phi_{g,v}^*\bar{g}_v)}L_{\phi_{g,v}^*\bar{g}_v}$. Note that this projection is smooth on $C^{2,\alpha}_\tau$ for $\tau>0$ since the elements of $\ker_{L^2(\phi_{g,v}^*\bar{g}_v)}L_{\phi^*_{g,v}g_{\bar{v}}}$ decay like $\rho_{g_b}^{-n}$ at infinity. 
		
		We consider $F$ on a neighborhood of $(g_b,0)$ in $C^{2,\alpha}_\tau\times \ker_{L^2(g_b)}L_{g_b}$ with values in \newline$\ker_{L^2(\phi_{g,v}^*\bar{g}_v)}L_{\phi_{g,v}^*\bar{g}_v}$. Here the map $F$ is analytic. We can apply the implicit function theorem as stated in Lemma \ref{th fcts implicites} to the map $F$. Indeed, $F(g_b,0)=0$ and $d_{(g_b,0)}F(0,v) = v$ is an isomorphism, and the spaces $C^{2,\alpha}_\tau\times \ker_{L^2(g_b)}L_{g_b}$ and $\ker_{L^2(\phi_{g,v}^*\bar{g}_v)}L_{\phi_{g,v}^*\bar{g}_v}$ are Banach spaces. We then conclude that there exists an analytic map (unique as a continuous map) $V$ such that for all metrics $g$ in a $C^{2,\alpha}_\tau$-neighborhood of $g_b$, we have 
		$$ F(g,V(g))=0. $$
		Now, for any $g$ satisfying $ \|g-g_b\|_{C^{2,\alpha}_\tau(g_b)}\leq \epsilon $ for $\epsilon>0$ small enough, we consider $g'_b = \phi_{g,V(g)}^*\bar{g}_{V(g)}$ which satisfies the desired properties.
	\end{proof}
	\begin{rk}
		The integrability of the Ricci-flat ALE metric $(N^n,g_b)$ is crucial to obtain the above statement. 
	\end{rk}
	
	Next, we prove that if a Ricci-flat ALE metric is stable and integrable then it is a local maximum of $\lambda_{\operatorname{ALE}}$: this result echoes [Theorem $A$, \cite{Has-Sta}]. Before stating and proving this result, we make a pause to discuss the relevant notion of stability we need here: 
	\begin{defn}[Locally stable integrable Ricci-flat ALE metrics]\label{definition integrable-loc-stable}
		An integrable Ricci-flat ALE metric $(N^n,g_b)$ is \emph{locally stable} if for any metric $g$ satisfying $ \|g-g_b\|_{C^{2,\alpha}_\tau(g_b)}\leq \epsilon $ for $\epsilon>0$ small enough, there exists a linearly stable Ricci-flat ALE metric $ g'_b $ satisfying the conclusions of Proposition \ref{gauge fixing ALE integrable} .
			\end{defn}
			
	Let us show that being linearly stable and integrable implies being locally stable.
		
	\begin{prop}\label{prop-int-lin-sta-loc-sta}
	Let $(N^n,g_b)$ be an integrable and linearly stable ALE Ricci-flat metric. Then it is locally stable in the sense of Definition \ref{definition integrable-loc-stable}.
	\end{prop}
	
	\begin{proof}
	Denote $\Bar{g}_0:=g_b$ and for $v\in \ker_{L^2(\Bar{g}_0)} L_{\Bar{g}_0}$ small enough, let $\Bar{g}_v$ be a Ricci-flat ALE metric satisfying $\Bar{g}_v-(\Bar{g}_0+v)\perp \ker_{L^2}L_{\Bar{g}_0}$ and $\textup{div}_{\bar{g}_0}\bar{g}_v=0$ since $\Bar{g}_0$ is assumed to be integrable. The map $v\in C^{2,\alpha}_\tau\rightarrow \Bar{g}_v\in C^{2,\alpha}_\tau$ is analytic: as already seen in the proof of Proposition \ref{gauge fixing ALE integrable}, this implies that there exists a basis of each $\ker_{L^2(\bar{g}_v)} L_{\Bar{g}_v}$ which depends analytically on $v$. Therefore, the $L^2(\bar{g}_v)$-projection on $\ker_{L^2(\bar{g}_v)}L_{\bar{g}_v}$ denoted by $\pi_v: H_{\frac{n}{2}-1}^1(\bar{g}_0)\to H^1_{\frac{n}{2}-1}(\bar{g}_0)$ depends analytically on $v$ since $\ker_{L^2(\bar{g}_v)}L_{\bar{g}_v}\subset C^\infty_n$ by Theorem \ref{prop-lic-fred}.
	
As $\Bar{g}_0$ is assumed to be linearly stable, thanks to Lemma \ref{lemma-equiv-def-stable}, there exists $c>0$ such that if $h_0\perp \ker_{L^2(\bar{g}_0)} L_{\Bar{g}_0}$ and $h_0\in H^1_{\frac{n}{2}-1}$, then one has
\begin{equation}
    \left<-L_{\Bar{g}_0} h_0,h_0\right>_{L^2}\geq c\|\nabla^{\Bar{g}_0}h_0\|_{L^2}^2,\label{stability at bar g0}
\end{equation}
by Theorem \ref{theo-der-kro}.

Now, if $h\perp \ker_{L^2(\bar{g}_v)} L_{\Bar{g}_v}$ and $h\in C_c^\infty(S^2T^*N)$, we decompose it as $h = h_0+h'$ where
$h_0\perp \ker_{L^2(\bar{g}_0)} L_{\Bar{g}_0}$ and $h'\in \ker_{L^2(\bar{g}_0)} L_{\Bar{g}_0}$. 

Let us first show that $h'$ is small when $v$ is. We have $0 = \pi_vh$ and $h' = \pi_0 h$, and since $v\mapsto \pi_v$ is analytic, there exists $C>0$ such that for $v$ small enough, one has $\|h'\|_{H^1_{\frac{n}{2}-1}(\bar{g}_0)}\leq \|v\|_{C^{2,\alpha}_{\tau}}\|h\|_{H^1_{\frac{n}{2}-1}(\bar{g}_0)}$. 

Using the fact that $h'\in\ker_{L^2(\bar{g}_0)} L_{\bar{g}_0}$, one gets immediately that
\begin{equation*}
\begin{split}
\left<-L_{\Bar{g}_0} h,h\right>_{L^2(\bar{g}_0)}=\,&\left<-L_{\Bar{g}_0} h_0,h_0\right>_{L^2(\bar{g}_0)}\\
\geq\,& c\|\nabla^{\Bar{g}_0}h_0\|_{L^2(\bar{g}_0)}^2,
\end{split}
\end{equation*}
where we used the inequality \eqref{stability at bar g0}.

In particular, if $\|v\|_{C^{2,\alpha}_\tau}$ is chosen small enough, then we have $\|\nabla^{\Bar{g}_0}h_0\|_{L^2(\bar{g}_0)}\geq \|\nabla^{\Bar{g}_0}h\|_{L^2(\bar{g}_0)}-\|\nabla^{\Bar{g}_0}h'\|_{L^2(\bar{g}_0)}\geq \frac{1}{\sqrt{2}}\|\nabla^{\Bar{g}_0}h\|_{L^2(\bar{g}_0)}$ and $\|\nabla^{\Bar{g}_0}h\|^2_{L^2(g_0)}\geqslant \frac{1}{2}\|\nabla^{\Bar{g}_v}h\|_{L^2(\bar{g}_v)}^2$ since $\|\Bar{g}_v-\Bar{g}_0\|_{C_{n}^{2,\alpha}}\leq 2\|v\|_{C_{n}^{2,\alpha}}$ by Definition \ref{definition integrable}. Therefore,
\begin{equation}
\left<-L_{\Bar{g}_0} h,h\right>_{L^2(\bar{g}_0)}\geq c\|\nabla^{\Bar{g}_0}h\|^2_{L^2(\bar{g}_0)}\geqslant \frac{c}{4}\|\nabla^{\Bar{g}_v}h\|_{L^2(\bar{g}_v)}^2,\label{fir-est-int-lin-sta}
\end{equation}
where $c$ is a positive constant independent of $v$ and $h$ that may vary from line to line. On the other hand, by linearizing $L_{\Bar{g}_0}$ with respect to $L_{\Bar{g}_v}$, we have
\begin{equation}
\begin{split}\label{sec-est-int-lin-sta}
\left|\left<(L_{\Bar{g}_0}-L_{\Bar{g}_v}) h,h\right>_{L^2(\bar{g}_v)}\right|&\leq\left|\left<(\Delta_{\Bar{g}_0}-\Delta_{\Bar{g}_v}) h,h\right>_{L^2(\bar{g}_v)}\right|+2\left|\left<\left(\Rm(\Bar{g}_0)-\Rm(\Bar{g}_v)\right)\ast h,h\right>_{L^2(\bar{g}_v)}\right|\\
&\leq\sum_{i=0}^2\left\langle\nabla^{\bar{g}_v,i}(\bar{g}_0-\bar{g}_v)\ast\nabla^{\bar{g}_v,2-i}h, h \right\rangle_{L^2(\bar{g}_v)}\\
&\leq C\|\bar{g}_v-\bar{g}_0\|_{C^{2,\alpha}_{n}}\sum_{i=1}^2\left\|\rho_{\bar{g}_0}^{-i}\nabla^{\bar{g}_0,2-i}h\ast h \right\|_{L^1(\bar{g}_0)}\\
&\quad+\left\langle\nabla^{\bar{g}_v}(\bar{g}_0-\bar{g}_v)\ast\nabla^{\bar{g}_v}h, h\right\rangle_{L^2(\bar{g}_v)}+\left\langle(\bar{g}_0-\bar{g}_v)\ast\nabla^{\bar{g}_v}h\ast \nabla^{\bar{g}_v}h\right\rangle_{L^2(\bar{g}_v)}\\
&\leq C\|v\|_{C^{2,\alpha}_{n}}\|\nabla^{\Bar{g}_v}h\|^2_{L^2(\bar{g}_v)},
\end{split}
\end{equation}
for some positive constant $C$ independent of $v$ and $h$.
 Here we have used the density of $C_c^\infty$ in $H^1_{\frac{n}{2}-1}$, integration by parts in the third inequality and Hardy's inequality (Theorem \ref{thm-min-har-inequ}) in the last line.

Combining (\ref{fir-est-int-lin-sta}) and (\ref{sec-est-int-lin-sta}), one gets for $v$ small enough, and for some constant $c>0$ independent on $v$ and $h$:
\begin{equation}
\left<-L_{\Bar{g}_v} h,h\right>_{L^2(\bar{g}_v)}\geq c\|\nabla^{\Bar{g}_v}h\|^2_{L^2(\bar{g}_v)},
\end{equation}
for any $h\perp\ker_{L^2(\Bar{g}_v)} L_{\Bar{g}_v}$. This shows that $\bar{g}_v$ is also linearly stable.
	\end{proof}
	
	Notice that all known examples of $4$-dimensional Ricci flat ALE metrics are hyperk\"ahler which implies that they are integrable. Furthermore, each infinitesimal deformation lying in the kernel of the corresponding Lichnerowicz operator is the first jet of a curve of hyperk\"ahler metrics. Since a hyperk\"ahler metric is linearly stable, an ALE hyperk\"ahler metric is integrable and locally stable in the sense of Definition \ref{definition integrable-loc-stable}: see \cite[Section $11$, Chapter $12$]{Besse} for more details.

	\begin{prop}\label{local maximum stable integrable}
		Let $(N^n,g_b)$, $n\geq 4$, be an ALE Ricci-flat metric asymptotic to $\RR^n\slash\Gamma$, for some finite subgroup $\Gamma$ of $SO(n)$ acting freely on $\mathbb{S}^{n-1}$. Let $\tau\in\left(\frac{n-2}{2},n-2\right)$ and let $\alpha\in(0,1)$. If $(N^n,g_b)$ is assumed to be integrable and linearly stable then it is a local maximum for the energy $\lambda_{\operatorname{ALE}}$ with respect to the topology defined by $C^{2,\alpha}_{\tau}(S^2T^*N)$. 
	\end{prop}
	
	\begin{proof}
		Consider the following Taylor expansion of the functional $\lambda_{\operatorname{ALE}}$ at $g_b$ of order $3$:
		\begin{equation}
		\lambda_{\operatorname{ALE}}(g_b+h)=\lambda_{\operatorname{ALE}}(g_b)+\delta_{g_b}\lambda_{\operatorname{ALE}}(h)+\delta^2_{g_b}\lambda_{\operatorname{ALE}}(h,h)+\int_0^1\frac{(1-t)^2}{2}\delta_{g_b+th}^3\lambda_{\operatorname{ALE}}(h,h,h)\,dt.\label{tay-exp-ord-3}
		\end{equation}
		Now, by definition of $\lambda_{\operatorname{ALE}}$, $\lambda_{\operatorname{ALE}}(g_b)=0$ and by (\ref{first-var-lambda}), one has $\delta_{g_b}\lambda_{\operatorname{ALE}}(h)=0$ as well by Proposition \ref{lambdaALE analytic}. Finally, by Proposition \ref{second-var-prop}, $\delta^2_{g_b}\lambda_{\operatorname{ALE}}(h,h)=\frac{1}{2}\langle L_{g_b}h,h\rangle_{L^2}$ if $\div_{g_b}h=0$. Since $(N^n,g_b)$ is integrable and linearly stable, it is integrable and locally stable by Proposition \ref{prop-int-lin-sta-loc-sta} and thanks to Theorem \ref{theo-der-kro} and Proposition \ref{gauge fixing ALE integrable}, it suffices to estimate the integral on the righthand side of (\ref{tay-exp-ord-3}) in such a way that it can be absorbed by $\|\nabla^{g_b} h\|^2_{L^2}$. More precisely, we claim the following:
		\begin{claim}\label{third-der-est}
			\begin{equation*}
			|\delta_{g_b+th}^3\lambda_{\operatorname{ALE}}(h,h,h)|\leq C\|h\|_{C^{2,\alpha}_{\tau}}\|\nabla^{g_b} h\|_{L^2}^2,\quad \div_{g_b}h=0,
			\end{equation*}
			for some positive constant $C=C(n,g_b,\varepsilon)$ uniform in $t\in[0,1]$.
		\end{claim}
		\begin{proof}[Proof of Claim \ref{third-der-est}]
			Recall from (\ref{first-var-lambda}) that:
			\begin{equation*}
			\delta^3_{g_b+th} \lambda_{\operatorname{ALE}}(h,h,h)=-\frac{d^2}{dt^2}\int_N\langle \Ric(g_b+th)+\nabla^{g_b+th,2}f_{g_b+th},h\rangle_{g_b+th} \,e^{-f_{g_b+th}}d\mu_{g_b+th}.
			\end{equation*}
			Denote the curve of metrics $(g_b+th)_{t\in[0,1]}$ by $(g_t)_{t\in[0,1]}$: the family $(g_t)_{t\in[0,1]}$ is uniformly equivalent in $t\in[0,1]$ since $g_t$ lies in an arbitrarily small neighborhood of $g_b$ in the $C^{2,\alpha}_{\tau}$ topology. Moreover, Proposition \ref{prop-pot-fct} ensures that $\|f_{g_t}\|_{C^{2,\alpha}_{\tau}}\leq \varepsilon\left(\|g_t-g_b\|_{C^{2,\alpha}_{\tau}}\right)$, where $\varepsilon(\cdot)$ is a positive function on $[0,+\infty)$ that tends to $0$ as its argument goes to $0$. Notice by [(\ref{lin-Bian-app}), Lemma \ref{Ric-lin-lemma-app}] applied to $g_1:=g_b$ and $g_2:=g_b+h$, that the Bianchi gauge satisfies:
			\begin{equation*}
			\begin{split}
			B&=\div_{g_1}(g_t-g_1)-\frac{1}{2}\nabla^{g_1}\tr_{g_1}(g_t-g_1)+g_t^{-1}\ast(g_t-g_1)\ast\nabla^{g_1}g_t\\
			&=-\frac{t}{2}\nabla^{g_b}\tr_{g_b}h+t^2g_t^{-1}\ast h\ast\nabla^{g_b}h,
			\end{split}
			\end{equation*}
			since $\div_{g_b}h=0$. Finally, since $\nabla^{g_t}T=\nabla^{g_b}T+g_t^{-1}\ast \nabla^{g_b}(g_t-g_b)\ast T$ for any tensor $T$, one gets 
			\begin{equation}
			\begin{split}
			\Li_B(g_t)&=-t\nabla^{g_b,2}\tr_{g_b}h+t^2\nabla^{g_b}(g_t^{-1}\ast h\ast\nabla^{g_b}h)\\
			&+tg_t^{-1}\ast \nabla^{g_b}h\ast \left(-\frac{t}{2}\nabla^{g_b}\tr_{g_b}h+t^2g_t^{-1}\ast h\ast\nabla^{g_b}h\right).\label{lie-der-gauge-bianchi}
			\end{split}
			\end{equation}
			
			Moreover, it can be shown with the help of Lemmata \ref{Ric-lin-lemma-app} and \ref{lem-lin-equ-Ric-first-var} that:
			\begin{equation}
			\left|-2\Ric(g_t)-tL_{g_b}h-t\nabla^{g_b,2}\tr_{g_b}h\right|\lesssim |\Rm(g_b)|_{g_b}|h|^2_{g_b}+|\nabla^{g_b}h|^2_{g_b}+|h|_{g_b}|\nabla^{g_b,2}h|_{g_b},
			\end{equation}
			and similarly, 
			\begin{equation}
			\left|-2\partial_t\Ric(g_t)-L_{g_b}h-\nabla^{g_b,2}\tr_{g_b}h\right|\lesssim |\Rm(g_b)|_{g_b}|h|^2_{g_b}+|\nabla^{g_b}h|^2_{g_b}+|h|_{g_b}|\nabla^{g_b,2}h|_{g_b},
			\end{equation}
			where the symbol $\lesssim$ denotes less than or equal to up to a positive multiplicative constant uniform in $t\in[0,1]$ which might depend on $n$, $g_b$, $\varepsilon$.

			By [(\ref{sec-der-Ric-rough}), Lemma \ref{Ric-lin-lemma-app}],  
			\begin{equation}
			\begin{split}
			\left|\int_N\left<\frac{\partial^2}{\partial t^2}\Ric(g_t),h\right>_{g_t}\,e^{-f_{g_t}}d\mu_{g_t}\right|&\lesssim \int_N|h|^2_{g_b}|\nabla^{g_b,2}h|_{g_b}+|\Rm(g_b)||h|_{g_b}^3+|\nabla^{g_b}h|^2_{g_b}|h|_{g_b}\,d\mu_{g_b}\\
			&\lesssim \left(\|\nabla^{g_b,2}h\|_{C^0_{2}}+\|h\|_{C^0_{0}}\right)\|\rho_{g_b}^{-1}h\|_{L^2}^2+\|h\|_{C^0_{0}}\|\nabla^{g_b}h\|_{L^2}^2\\
			&\lesssim \|h\|_{C^2_{\tau}}\|\nabla^{g_b}h\|^2_{L^2},
			\end{split}
			\end{equation}
			where we have used Hardy's inequality that holds on $(N^n,g_b)$ thanks to Theorem \ref{thm-min-har-inequ}. A similar estimate holds for mixed derivatives with respect to the parameter $t\in[0,1]$ that involve the terms $\Ric(g_t)$, the scalar product on symmetric $2$-tensors induced by $g_t$ and the Riemannian volume $d\mu_{g_t}$. 
			
			We use Proposition \ref{prop-pot-fct} that ensures that 
			\begin{equation}\label{control-covid-delta-f}
			\|\delta^k_{g_t}f(h,...,h)\|_{C^{2,\alpha}_{\tau}}\leq C\left(\|g_t-g_b\|_{C^{2,\alpha}_{\tau}}\right)\|h\|^k_{C^{2,\alpha}_{\tau}},\,\quad k\geq 0,
			\end{equation}
			to handle the derivatives falling on $e^{-f_{g_t}}$.
			
			For instance, let us handle the term involving $\left<\partial_t\Ric(g_t),h\right>_{g_t}\delta_{g_t}f(h)$ as follows:
			\begin{equation*}
			\begin{split}
			\left|\int_N\left<\partial_t\Ric(g_t),h\right>_{g_t}\delta_{g_t}f(h)\,e^{-f_{g_t}}d\mu_{g_t}\right|\lesssim\,&\left|\int_N\langle L_{g_b}h+\nabla^{g_b,2}\tr_{g_b}h,h\rangle_{g_b}\delta_{g_t}f(h)\,e^{-f_{g_t}}d\mu_{g_b}\right|\\
			&+\|h\|_{C^{2,\alpha}_{\tau}}\int_N |\Rm(g_b)|_{g_b}|h|^2_{g_b}+|\nabla^{g_b}h|^2_{g_b}\,d\mu_{g_b}\\
			&+\|h\|_{C^{2,\alpha}_{\tau}}\int_N|h|^2_{g_b}|\nabla^{g_b,2}h|_{g_b}\,d\mu_{g_b}\\
			\lesssim\,&\left|\int_N\langle \Delta_{g_b}h+\nabla^{g_b,2}\tr_{g_b}h,h\rangle_{g_b}\delta_{g_t}f(h)\,e^{-f_{g_t}}d\mu_{g_b}\right|\\
			&+\|h\|_{C^{2,\alpha}_{\tau}}\|\nabla^{g_b}h\|_{L^2}^2.
			\end{split}
			\end{equation*}
			Here we have used Hardy's inequality in the first line to get rid of the zeroth order term appearing in the Lichnerowicz operator, in the first term of the second line by using that the curvature $\Rm(g_b)$ decays at least quadratically and in the third line invoking the quadratic decay of $\nabla^{g_b,2}h$. Finally, the weighted Riemannian measure $e^{-f_{g_t}}d\mu_{g_t}$ on the righthand side of the first line has been turned into $d\mu_{g_b}$ since they are uniformly equivalent as measures by (\ref{control-covid-delta-f}) and the fact that $g_b+h$ and $g_b$ are uniformly equivalent as metrics. Notice that we have only made use of $h\in C^2_0$ here.

			It remains to estimate the terms involving the second covariant derivatives of $h$: by integration by parts, 
			\begin{equation*}
			\begin{split}
			\left|\int_N\langle \Delta_{g_b}h,h\rangle_{g_b}\delta_{g_t}f(h)e^{-f_{g_t}}d\mu_{g_b}\right|&\lesssim\int_N |\nabla^{g_b}h|^2_{g_b}|\delta_{g_t}f(h)|+|\nabla^{g_b}h|_{g_b}|h|_{g_b}|\nabla^{g_b}\delta_{g_t}f(h)|_{g_b}d\mu_{g_b}\\
			&\quad+\int_N|\nabla^{g_b}h|_{g_b}|h|_{g_b}|\delta_{g_t}f(h)||\nabla^{g_b}f_{g_t}|_{g_b}d\mu_{g_b}\\
			&\lesssim \|h\|_{C^{2,\alpha}_{\tau}}\|\nabla^{g_b}h\|_{L^2}^2+\|h\|_{C^{2,\alpha}_{\tau}}\|\nabla^{g_b}h\|_{L^2}\|\rho_{g_b}^{-1}h\|_{L^2}\\
			&\lesssim \|h\|_{C^{2,\alpha}_{\tau}}\|\nabla^{g_b}h\|_{L^2}^2.
			\end{split}
			\end{equation*}
			Here, we have made constant use of (\ref{control-covid-delta-f}). In particular, we have used the fact that $\nabla^{g_b}f_{g_t}$ and $\nabla^{g_b}\delta_{g_t}f(h)$ decay at least linearly together with Cauchy-Schwarz inequality in the second line and Hardy's inequality (Theorem \ref{thm-min-har-inequ}) in the third line. The integral involving the Hessian of $\tr_{g_b}h$ can be handled similarly.
			
			We are left with estimating integrals involving (the $t$-derivatives of) $\nabla^{g_t,2}f_{g_t}$. Let us notice first that any term which contains a $t$-derivative that falls either on the scalar product on symmetric $2$-tensors induced by $g_t$ or on the volume element $d\mu_{g_t}$ can be estimated as previously by using Hardy's inequality (Theorem \ref{thm-min-har-inequ}) in a quite straightforward way. Consequently, only the integrals that contain $t$-derivatives of $\nabla^{g_t,2}f_{g_t}$ and $e^{-f_{g_t}}$ will be estimated.
			
			We start with terms involving derivatives of $e^{-f_{g_t}}$ only. By integration by parts,
			\begin{equation*}
			\begin{split}
			\Bigg|\int_N\langle\nabla^{g_t,2}f_{g_t},h\rangle_{g_t}\delta_{g_t}^2f(h,h)&e^{-f_{g_t}}d\mu_{g_t}\Bigg|\\
			\lesssim\,&\int_N\left(|h|_{g_t}|\nabla^{g_t}f_{g_t}|^2_{g_t}+|\div_{g_t}h|_{g_t}|\nabla^{g_t}f_{g_t}|_{g_t}\right)|\delta_{g_t}^2f(h,h)|d\mu_{g_t}\\
			&+\int_N|h|_{g_t}|\nabla^{g_t}f_{g_t}|_{g_t}|\nabla^{g_t}\delta_{g_t}^2f(h,h)|_{g_t}\,d\mu_{g_t}\\
			\lesssim\,&\|h\|^2_{C^{2,\alpha}_{\tau}}\left(\|\rho_{g_b}^{-1}h\|_{L^2}\|\nabla^{g_t}f_{g_t}\|_{L^2}+\|\nabla^{g_t}h\|_{L^2}\|\nabla^{g_t}f_{g_t}\|_{L^2}\right)\\
			\lesssim\,&\|h\|_{C^{2,\alpha}_{\tau}}\|\nabla^{g_t}h\|_{L^2}\|\nabla^{g_t}f_{g_t}\|_{L^2},
			\end{split}
			\end{equation*}
			where we have used the linear decay of $\nabla^{g_t}f_{g_t}$ together with the fact that $\|h\|_{C^{2,\alpha}_{\tau}}\lesssim 1$. Now, from [(\ref{grad-est-int-ene}), Proposition \ref{prop-ene-pot-fct}],
			
			\begin{equation*}
			\begin{split}
			\left|\int_N\langle\nabla^{g_t,2}f_{g_t},h\rangle_{g_t}\delta_{g_t}^2f(h,h)e^{-f_{g_t}}d\mu_{g_t}\right|&\lesssim\|h\|_{C^{2,\alpha}_{\tau}}\|\nabla^{g_t}h\|_{L^2}^2,
			\end{split}
			\end{equation*}
			as desired.
			With the help of (\ref{first-var-lie-der-app}) from Lemma \ref{lemma-app-lie-der-lin}, we proceed to estimate terms involving $t$-derivatives of $\nabla^{g_t}f_{g_t}$: by integrating by parts,
			\begin{equation}
			\begin{split}\label{huge-est-covid}
			\left|\int_N\langle\frac{\partial}{\partial t}\nabla^{g_t,2}f_t,h\rangle_{g_t}\delta_{g_t}f(h)\,e^{-f_{g_t}}d\mu_{g_t}\right|\lesssim& \left|\int_N\langle\nabla^{g_t,2}\delta_{g_t}f(h),h\rangle_{g_t}\delta_{g_t}f(h)\,e^{-f_{g_t}}d\mu_{g_t}\right|\\
			& +\left|\int_N\langle\mathcal{L}_{\nabla^{g_t}f_{g_t}}(h),h\rangle_{g_t}\delta_{g_t}f(h)\,e^{-f_{g_t}}d\mu_{g_t}\right|\\
			& +\left|\int_N\langle\mathcal{L}_{h(\nabla^{g_t}f_{g_t})}(g_t),h\rangle_{g_t}\delta_{g_t}f(h)\,e^{-f_{g_t}}d\mu_{g_t}\right|.
			\end{split}
			\end{equation}
			
			Let us estimate the first integral on the righthand side of the previous inequalities (\ref{huge-est-covid}):			
			\begin{equation}
			\begin{split}\label{intermed-est-covid}
			\Bigg|\int_N\langle\nabla^{g_t,2}\delta_{g_t}f(h),h\rangle_{g_t}\delta_{g_t}f(h)\,&e^{-f_{g_t}}d\mu_{g_t}\Bigg|\\
			\lesssim\,& \int_N|\nabla^{g_t}\delta_{g_t}f(h)|_{g_t}|\div_{g_t}h|_{g_t}|\delta_{g_t}f(h)|d\mu_{g_t}\\
			&+\int_N|\nabla^{g_t}\delta_{g_t}f(h)|_{g_t}|h|_{g_t}\left(|\nabla^{g_t}f_{g_t}|_{g_t}+|\nabla^{g_t}\delta_{g_t}f(h)|_{g_t}\right)d\mu_{g_t}.
			\end{split}
			\end{equation}
			Now, since $\div_{g_b}h=0$, one has $|\div_{g_t}h|_{g_t}\lesssim |h|_{g_b}|\nabla^{g_b}h|_{g_b}.$ Using that $|\nabla^{g_t}\delta_{g_t}f(h)|_{g_t}$ decays at least linearly together with (\ref{control-covid-delta-f}) applied to $k=1$,
			\begin{equation*}
			\begin{split}
			\int_N|\nabla^{g_t}\delta_{g_t}f(h)|_{g_t}|\div_{g_t}h|_{g_t}|\delta_{g_t}f(h)|d\mu_{g_t}&\lesssim \|h\|_{C^{2,\alpha}_{\tau}}\|\nabla^{g_b}h\|_{L^2}^2+\|\rho_{g_b}^{-1}h\|_{L^2}\|\nabla^{g_b}h\|_{L^2}\\
			&\lesssim \|h\|_{C^{2,\alpha}_{\tau}}\|\nabla^{g_b}h\|_{L^2}^2,
			\end{split}
			\end{equation*}
			by Young's inequality and Hardy's inequality (Theorem \ref{thm-min-har-inequ}). Thanks to [(\ref{est-grad-first-der-pot-fct-ene}), Proposition \ref{prop-ene-pot-fct}],
			\begin{equation*}
			\int_N|\nabla^{g_t}\delta_{g_t}f(h)|^2_{g_t}|h|_{g_t}d\mu_{g_t}\lesssim\|h\|_{C^0}\|\nabla^{g_t}\delta_{g_t}f(h)\|_{L^2}^2\lesssim \|h\|_{C^0}\|\nabla^{g_b}h\|_{L^2}^2.
			\end{equation*}
			The other integrals involved on the righthand sides of (\ref{intermed-est-covid}) and (\ref{huge-est-covid})  can be treated in a similar way. The same is true for terms involving the second $t$-derivatives of $\nabla^{g_t,2}f_{g_t}$ by using (\ref{sec-var-lie-der-app}) from Lemma \ref{lemma-app-lie-der-lin}.
		\end{proof}
		This concludes the proof of Proposition \ref{local maximum stable integrable}.
	\end{proof}

	\subsection{A first way to prove a \L{}ojasiewicz inequality}\label{naive-loja-sec}~~\\

	In order to prove a \L{}ojasiewicz inequality in the stable integrable case, a natural strategy consists in proving it at an infinitesimal level and then it is sufficient to control the nonlinear terms, see \cite{Has-Sta} for example in the case of a closed manifold. Here, we succintly mention how one can implement this strategy in a non-compact situation. This will be proven in general in the next section.
	
	We start by proving an infinitesimal version of \L{}ojasiewicz inequality.
	
	\begin{prop}\label{prop-baby-loja-l2n/2+1} Let $(N^n,g_b)$, $n\geq 4$, be a linearly stable ALE Ricci-flat metric asymptotic to $\RR^n\slash\Gamma$, for some finite subgroup $\Gamma$ of $SO(n)$ acting freely on $\mathbb{S}^{n-1}$. Let  $\tau\in\left(\frac{n-2}{2},n-2\right)$. Then the following \L{}ojasiewicz inequality holds true for any sufficiently small symmetric $2$-tensor $h\in C_{\tau}^{2}(S^2T^*N)$:
		\begin{equation}
		\langle-L_{g_b} h,h\rangle_{L^2}\leq C\|L_{g_b}h\|^2_{L^2_{\frac{n}{2}+1}}, \quad \label{loja-ineq-stab-babyL2n/2+1}
		\end{equation}\
		for some positive constant $C=C(n,\tau,g_b)$.\\
	\end{prop}
	\begin{rk}
		This is the first order version of the inequality 
		$$|\lambda_{\operatorname{ALE}}(g)|\leq C \|\nabla \lambda_{\operatorname{ALE}}(g)\|^2_{L^2_{\frac{n}{2}+1}},$$
		which is an $L^2_{\frac{n}{2}+1}$-\L{}ojasiewicz inequality with optimal exponent $\theta=1$.
	\end{rk}
	\begin{proof}
		Let us prove inequality (\ref{loja-ineq-stab-baby}) for functions on $(N^n,g_b)$ first. Let $u\in C_{\tau}^{2}(N)$ which implies in particular that for any $\tau'<\tau$, $\Delta_{g_b} u\in L^2_{\tau'+2}$ and $u\in L^2_{\tau'}$. Let us now assume that $\tau>\frac{n}{2}-1$ which naturally implies that $u\in L^2_{\frac{n}{2}-1}$, $\nabla^{g_b} u\in L^2_{\frac{n}{2}}$ and $\Delta_{g_b} u \in L^2_{\frac{n}{2}+1}$. By Cauchy-Schwarz inequality, we have $\langle-\Delta_{g_b} u,u\rangle_{L^2}\leq \|\Delta_{g_b} u\|_{L^2_{\frac{n}{2}+1}}\|u\|_{L^2_{\frac{n}{2}-1}}$. Now, by Hardy's inequality from Theorem \ref{thm-min-har-inequ}, we get
		\begin{equation*}
		\|u\|_{L^2_{\frac{n}{2}-1}}^2 = \|\rho_{g_b}^{-1}u\|_{L^2_{\frac{n}{2}}}^2\leq C \|\nabla^{g_b}u\|^2_{L^2} = C\langle-\Delta_{g_b} u,u\rangle_{L^2}.
		\end{equation*}
		
		Therefore, 
		\begin{equation}
		\langle-\Delta_{g_b} u,u\rangle_{L^2}\leq C(n,\tau,g_b)\|\Delta_{g_b} u\|_{L^2_{\frac{n}{2}+1}}\langle-\Delta_{g_b} u,u\rangle_{L^2}^\frac{1}{2}.\label{int-baby-loj-fcts}
		\end{equation}
		
		Now, notice that the proof goes almost verbatim for symmetric $2$-tensors $h\in C_{\tau}^{2}(S^2T^*N)$ such that $\|h\|_{C_{\tau}^0}\leq 1$. Notice also that it is sufficient to restrict to tensors orthogonal to $\ker_{L^2}L_{g_b}$.
		
		Indeed, by Cauchy-Schwarz inequality, $\langle-L_{g_b} h,h\rangle_{L^2}\leq \|L_{g_b} h\|_{L^2_{\frac{n}{2}+1}}\|h\|_{L^2_{\frac{n}{2}-1}}$. On the one hand, by following the same reasoning as in the proof of (\ref{int-baby-loj-fcts}), one then gets:
		\begin{equation}
		\langle-L_{g_b} h,h\rangle_{L^2}\leq C(n,\tau,g_b)\|L_{g_b} h\|_{L^2_{\frac{n}{2}+1}}\langle-\Delta_{g_b} h,h\rangle_{L^2}^\frac{1}{2}.\label{int-baby-loj-tensors}
		\end{equation}
		On the other hand, by Theorem \ref{theo-der-kro}, if $h\in C^{2}_{\tau}(S^2T^*N)$ and if $h\perp\ker_{L^2}L_{g_b}$ (which is well-defined because $\tau>0$ and $\ker_{L^2}L_{g_b}\subset C^0_n$), one gets thanks to (\ref{int-baby-loj-tensors}), 
		\begin{equation}
		\langle-L_{g_b} h,h\rangle_{L^2}\leq C(n,\tau,g_b)\|L_{g_b} h\|_{L^2_{\frac{n}{2}+1}}\langle-L_{g_b} h,h\rangle_{L^2}^\frac{1}{2}.\label{int-baby-loj-tensors-bis}
		\end{equation}
	\end{proof}
	
	We now state and prove an interpolation inequality between weighted Sobolev spaces. By definition of our weighted norms, we have $\|.\|_{L^2}\leq \|.\|_{L^2_{\frac{n}{2}+1}}$. We will show that assuming that our tensors decay at infinity, we have a weaker reverse inequality.
	\begin{lemma}\label{lemma-interpol-delig}
		Let $T$ be in $L^2(S^2T^*N)\,\cap \,L^2_\beta(S^2T^*N)$ for $\beta>\frac{n}{2}+1$. Then, for $\delta\in(0,1)$ such that $\beta = \frac{n}{2}+\frac{1}{1-\delta}$ we have the following control :
		\begin{equation}
		\|T\|_{L^2_{\frac{n}{2}+1}}\leq \|T\|_{L^2}^\delta\|T\|_{L^2_\beta}^{1-\delta}.\label{interpolation L2}
		\end{equation}
	
		In particular, if $T$ is small enough in norm $C^0_{\tau+2}$ for $\tau>\frac{n-2}{2}$, we have $$\|T\|_{L^2_{\frac{n}{2}+1}}\leq \|T\|_{L^2}^\delta,$$ for any $\delta$ such that $ \tau+2> \frac{n}{2}+\frac{1}{1-\delta}$, that is $\delta<\frac{2\tau-(n-2)}{2\tau-(n-4)}$.
	\end{lemma}
	\begin{proof}
		For this, let us choose $0<\delta<1$ such that $\beta=\frac{n}{2}+\frac{1}{1-\delta}$ which always exists since $\beta>\frac{n}{2}+1$. We then have,
		\begin{align}
		\|T\|^2_{L^2_{\frac{n}{2}+1}} &= \int_{N}|T|_{g_b}^{2\delta}\cdot |T|_{g_b}^{2(1-\delta)}\rho_{g_b}^{2}\,d\mu_{g_b}\nonumber\\
		&\leq \left(\int_{N}|T|_{g_b}^2\,d\mu_{g_b}\right)^{\delta}\cdot\left(\int_{N}|T|_{g_b}^2\rho_{g_b}^{\frac{2}{1-\delta}}\,d\mu_{g_b}\right)^{1-\delta}\nonumber\\
		&=\left(\int_{N}|T|_{g_b}^2\,d\mu_{g_b}\right)^{\delta}\cdot\left(\int_{N}|T|_{g_b}^2\rho_{g_b}^{2\beta-n}\,d\mu_{g_b}\right)^{1-\delta}\nonumber\\
		&=\|T\|_{L^2}^{2\delta}\|T\|_{L^2_\beta}^{2(1-\delta)}\nonumber.
		\end{align}
		Here we have made use of the assumption $\frac{2}{1-\delta} = 2\beta-n$ together with the definition of the space $L^2_\beta(S^2T^*N)$.
	\end{proof}
	
We use this inequality in order to find a \L{}ojasiewicz inequality  with an $L^2$ norm of $\nabla\lambda_{\operatorname{ALE}}$ on the right-hand side.
	\begin{prop}\label{prop-baby-loja} Let $(N^n,g_b)$, $n\geq 5$, be a linearly stable ALE Ricci-flat metric asymptotic to $\RR^n\slash\Gamma$, for some finite subgroup $\Gamma$ of $SO(n)$ acting freely on $\mathbb{S}^{n-1}$. Let $\tau\in\left(\frac{n}{2},n-2\right)$. Then for any $0<\delta<\frac{2\tau-(n-2)}{2\tau-(n-4)}$, there exists $C>0$ such that for all  the following \L{}ojasiewicz inequality holds true for any sufficiently small symmetric $2$-tensor $h\in C_{\tau}^{2}(S^2T^*N)$:
		\begin{equation}
		\langle-L_{g_b} h,h\rangle_{L^2}^{2-\theta_{L^2}}\leq C\|L_{g_b}h\|^2_{L^2},\quad\theta_{L^2}:=2-\frac{1}{\delta}, \quad \label{loja-ineq-stab-baby}
		\end{equation}\
		for some positive constant $C=C(n,\theta_{L^2},\tau,g_b)$.\\
	\end{prop}
\begin{proof}
	Let us use the interpolation inequality [\eqref{interpolation L2}, Lemma \ref{lemma-interpol-delig}] on the $L^2_{\frac{n}{2}+1}$- norm on the right-hand side of the inequality \eqref{loja-ineq-stab-babyL2n/2+1}. This leads to the desired inequality (\ref{loja-ineq-stab-baby}) with $\theta>0$ such that $1-\frac{\theta}{2} = \frac{1}{2\delta}$ by considering $T:= L_{g_b} h$. Note however that this leads to $\theta>0$ only if $\delta>\frac{1}{2}$ which is only possible when $\tau>\frac{n}{2}$.
\end{proof}

	\begin{rk}
		The proof above works in all dimensions and yields a nontrivial infinitesimal \L{}ojasiewicz inequality (that is with $0<\theta<1$) assuming that $\tau>\frac{n}{2}$. However, in dimension $3$ and $4$, this imposes $\tau>n-2$ and induces several difficulties later on. We will also see in a forthcoming work that we cannot expect a $C^{2,\alpha}_\tau$-convergence for $\tau>n-2$ along the Ricci flow: this shows that we crucially need a \L{}ojasiewicz inequality adapted to the case $\tau<n-2$.
	\end{rk}

Both of the exponents obtained for the $L^2_{\frac{n}{2}+1}$ and the $L^2$-\L{}ojasiewicz inequalities of Propositions \ref{prop-baby-loja-l2n/2+1}  and \ref{prop-baby-loja} are moreover optimal as one can see in the following example on functions which obviously extends to conformal deformations.

\begin{exmp}
    Let $A\gg 1$ and $\tau>0$. Consider the cut off function $\chi_A$ of Example \ref{exemple masse infinie}. Let us define $u_{A,\tau} = \chi_A \rho_{g_b}^{-\tau}\in C^{2,\alpha}_{\tau}$.
    
    Then, denoting $f(u_{A,\tau})\sim g(A,\tau)$ if there exists $C>0$ independent of $A$ large enough such that $C^{-1}g(A,\tau)<f(u_{A,\tau})<Cg(A,\tau)$, we have the following controls:
    \begin{itemize}
        \item $-\int_{N}u_{A,\tau}\Delta_{g_b} u_{A,\tau}dv_{g_b}=\int_{N}|\nabla^{g_b} u_{A,\tau}|^2dv_{g_b} \sim A^{(n-2)-2\tau}$,
        \item $\|\Delta_{g_b} u_{A,\tau}\|_{L^2}^2=\int_{N}|\Delta_{g_b} u|^2dv_{g_b} \sim A^{(n-4)-2\tau}$, and
        \item $\|\Delta_{g_b} u_{A,\tau}\|_{L^2_{\frac{n}{2}+1}}^2=\int_{N}\rho_{g_b}^2|\Delta_{g_b} u|^2dv_{g_b} \sim A^{(n-2)-2\tau}$.
    \end{itemize}
    We therefore see that we exactly have 
    $$-\int_{N}u_{A,\tau}\Delta_{g_b} u_{A,\tau}dv_{g_b}\sim \|\Delta_{g_b} u_{A,\tau}\|_{L^2_{\frac{n}{2}+1}}^2, $$
     and 
    $$-\int_{N}u_{A,\tau}\Delta_{g_b} u_{A,\tau}dv_{g_b}\sim \|\Delta_{g_b} u_{A,\tau}\|_{L^2}^{2\delta}$$
    with $\delta = \frac{2\tau-(n-2)}{2\tau-(n-4)}$.
\end{exmp}

	\begin{rk}\label{rk-exp-loja}
		Notice that the exponent $2-\theta$ is smaller than $1$ and is asymptotically $1$ as the dimension $n$ increases. This is in contrast with the case where $(N^n,g_b)$ is a closed Riemannian manifold endowed with an integrable Ricci-flat metric $g_b$.
	\end{rk}
	
	By bounding from above the nonlinear terms of $\|\nabla\lambda_{\operatorname{ALE}}(g_b+h)\|_{L^2}$ by $\frac{1}{2}\|L_{g_b}h\|_{L^2}$ and by bounding the higher order terms of $|\lambda_{\operatorname{ALE}}(g_b+h)|$ by $\frac{1}{2}|\langle-L_{g_b}h,h\rangle_{L^2}|$ in an analogous way to the proof of Proposition \ref{local maximum stable integrable}, we get an $L^2$-\L{}ojasiewicz inequality. Note that controlling these nonlinear terms at this point is far from being an easy task, and in particular it does not seem to be possible to do so in dimension $3$ and $4$, see the discussion below Remark \ref{rk-exp-loja}. For these reasons, we do not state it and refer the reader to Theorem \ref{theo-loja-int-opt}.

	There are several difficulties to develop a similar argument in dimension $n=4$. 
	\begin{itemize}
		\item The exponent given in \eqref{loja-ineq-stab-baby} does not provide a \L{}ojasiewicz inequality with $\theta>0$.
		\item In the proof of the bounds for the nonlinear terms, a reason is that $2 = n-2 = \frac{n}{2}$ as well as $0 = \frac{n}{2}-2$ are exceptional values of the Laplacian in this dimension, which complicates the use of Fredholm theory. In particular, the space $\big(\ker_{L^2}L_{g_b}\big)^\perp$ is not closed for the norms $ H^2_0 $ (or $H^2_{-\epsilon}$).
		\item There are also asymptotically constant $2$-tensors in the kernel of the Lichnerowicz Laplacian which have to be dealt with to obtain informations on the $L^2 = L^2_2$-norm of the Hessian of functions. Indeed an inequality $$\|\nabla^{g_b,2}h\|_{L^2}\leq C \|L_{g_b}h\|_{L^2}$$
		for $h\perp \ker_{L^2}L_{g_b}$ is contradicted by the elements of $\ker_{C^0}L_{g_b}\cap (\ker_{L^2}L_{g_b})^\perp$. 
	\end{itemize}
	A priori, on this last point, the orthogonality is not well defined because the elements of $\ker_{L^2}L_{g_b}$ are $O(r^{-4})$ while the elements of $\ker_{C^0}L_{g_b}$ are $O(1)$ and this might lead to a nonconvergent integral. However, the elements of $\ker_{L^2}L_{g_b}$ are asymptotic to $H_2\cdot \rho_{g_b}^{-4}$ for $H_2$ a $2$-tensor whose coefficients are eigenfunctions of the spherical Laplacian for the second eigenvalue while the elements of $\ker_{C^0}L_{g_b}$ are asymptotic to some constant $2$-tensor $H_0$. Since $\int_{\mathbb{S}^3}\langle H_0,H_2\rangle_{\mathbb{S}^3}\, d\mu_{\mathbb{S}^3} = 0$, the $L^2$-product of an element of $\ker_{L^2}L_{g_b}$ and an element of $\ker_{C^0}L_{g_b}$ is well-defined.

	\section{A \L{}ojasiewicz inequality for $\lambda_{\operatorname{ALE}}$: the general case}\label{sec-loja-ineq-gal-case}
	
	In this section, we check that the classical proof of the \L{}ojasiewicz-Simon inequality, and in particular its version summarized in \cite{Col-Min-Ein-Tan-Con}, holds in the context of weighted spaces. We will then deduce a \L{}ojasiewicz-Simon inequality for $\lambda_{\operatorname{ALE}}$ in the neighborhood of a given Ricci-flat ALE metric.
	
	\subsection{A general $L^2_{\frac{n}{2}+1}$-\L{}ojasiewicz inequality for functionals on ALE metrics}\label{sec-gal-loja}~~\\
	
	The scheme of proof of \L{}ojasiewicz-Simon inequality by Lyapunov-Schmidt reduction summarized in \cite{Col-Min-Ein-Tan-Con} extends to the setting of weighted norms. Indeed, the fact that the function spaces are modeled on H\"older spaces $C^{k,\alpha}$ is not so essential in Colding-Minicozzi's proof, the crucial properties of these spaces being that they are Banach and that the linearization of the gradient is Fredholm between them.
	
Before we state the main result of this section, we discuss the notion of the gradient of  functionals in the setting of ALE metrics when defined on $C^{2,\alpha}_{\tau}$ or on $C^{0,\alpha}_{\tau+2}$, $\tau\in \left(\frac{n-2}{2},n-2\right)$, $\alpha\in(0,1)$. If $F:\mathcal{O}\rightarrow \RR$ is a $C^1$ functional defined on an open subset $\mathcal{O}$ of $C^{2,\alpha}_{\tau}$ then it admits a unique gradient denoted by $\nabla F$ defined on $\mathcal{O}$ with values into $L^{2}_{\frac{n}{2}+1}$ such that $D_{h}F(v)=\left<\nabla F(h),v\right>_{L^2}$ for all $h\in\mathcal{O}$ and $v\in L^2_{\frac{n}{2}-1}$. Similarly, we use the same notation to denote the gradient of any $C^1$ functional defined on $C^{0,\alpha}_{\tau+2}\subset L^2_{\frac{n}{2}+1}$, the only difference here being that the gradient belongs to $L^2_{\frac{n}{2}-1}$. 
	\begin{prop}\label{Lojasiewicz ineq weighted}
		Let $\tau\in(\frac{n-2}{2},n-2)$, $\alpha\in (0,1)$ and let $(N^n,g_b)$ be an ALE Ricci-flat metric. Let $E$ (respectively $F$) be a closed subspace of $L^2_{\frac{n}{2}-1}(S^2T^*N)$ (respectively of $L^2_{\frac{n}{2}+1}(S^2T^*N)$) such that $ C^{2,\alpha}_{\tau}(S^2T^*N)\,\cap\,E$ and $C^{0,\alpha}_{\tau+2}(S^2T^*N)\,\cap\,F$ are respectively a closed subset of $C^{2,\alpha}_{\tau}(S^2T^*N)$ and a closed subset of $C^{0,\alpha}_{\tau+2}(S^2T^*N)$. Let $G : \mathcal{O}\subset C^{2,\alpha}_{\tau}(S^2T^*N) \to \RR$ be an analytic functional in the sense of Definition \ref{def-analytic} defined on $\mathcal{O}$, a neighborhood of $0$ in $  C^{2,\alpha}_{\tau}(S^2T^*N)\cap E$.
		
		If it satisfies,
		\begin{enumerate}
			\item\label{item-0} the gradient of $G$, $\nabla G : \mathcal{O}\to C^{0,\alpha}_{\tau+2}$ has a Fr\'echet derivative at each point which varies continuously, with $\nabla G(0)=0$, and 
			\begin{equation}
			\big\|\nabla G(x)-\nabla G(y)\big\|_{L^2_{\frac{n}{2}+1}} \leq C\|x-y\|_{H^{2}_{\frac{n}{2}-1}},\label{lip-bd-nabla-G}
			\end{equation}
						\item \label{item-0-bis}the Fr\'echet derivative of $\nabla G:\mathcal{O}\rightarrow C^{0,\alpha}_{\tau}$ is continuous when interpreted as a map from $H^{2}_{\frac{n}{2}-1}$ to $L^2_{\frac{n}{2}+1}$,
			\item the linearization $L$ of (the extension to $H^2_{\frac{n}{2}-1}$ of) $\nabla G$ at $0$,
			\begin{enumerate}
				\item \label{item-1}is bounded from $H^{2}_{\frac{n}{2}-1}$ to $L^2_{\frac{n}{2}+1}$ and Fredholm $H^{2}_{\frac{n}{2}-1}\cap E$ to $L^2_{\frac{n}{2}+1}\cap F$,
				\item \label{item-2} its kernel $\mathbf{K}$ on $H^{2}_{\frac{n}{2}-1}\cap E$ equals its $L^2$-cokernel on $L^2_{\frac{n}{2}+1}\cap F$,
				\item \label{item-3} is bounded from $C^{2,\alpha}_{\tau}$ to $C^{0,\alpha}_{\tau+2}$,
				\item \label{item-4} is Fredholm from $C^{2,\alpha}_{\tau}\,\cap\, E$ to $C^{0,\alpha}_{\tau+2}\,\cap\, F$,
				\item \label{item-5} its kernel on $C^{2,\alpha}_\tau\,\cap \,E$ equals its $L^2$-cokernel on $C^{0,\alpha}_{\tau+2}\,\cap\,F$ and is equal to $\mathbf{K}$,
			\end{enumerate}
		\end{enumerate}
		then, there exist $\theta \in (0,1]$ and $C>0$ such that for all sufficiently small $x$ in $C^{2,\alpha}_{\tau}(S^2T^*N)\cap E$,
		\begin{equation*}
		|G(x)-G(0)|^{2-\theta}\leq C\|\nabla G(x)\|_{L^2_{\frac{n}{2}+1}}^{2}.
		\end{equation*}
		Moreover, the constant $\theta$ is independent of $\alpha\in(0,1)$ and is a monotone nondecreasing function of $\tau\in\left(\frac{n-2}{2},n-2\right)$.
	\end{prop}
	\begin{rk}
		We will use this quite general statement with
		\begin{itemize}
			\item the closed subspaces $E = \ker\div_{g_b}\subset C^{2,\alpha}_{\tau}$ and $F:=\div_{g_b}^{\ast}\left(C^{\infty}_c(TN)\right)^{\perp}.$
			\item the functional $G={\lambda}_{ALE}(g_b +.)$ defined on a $C^{2,\alpha}_\tau$-neighborhood of $0$ by Section \ref{extension tilde lambda},
			\item its $L^2(e^{-f_g}d\mu_g)$-gradient, $$h\mapsto\Ric(g_b+h) + \nabla^{g_b+h,2}f_{g_b+h},$$
			\item the linearization of $h\mapsto\Ric(g_b+h) + \nabla^{g_b+h,2}f_{g_b+h}$ at $0$ on $E$ is $\frac{1}{2}L_{g_b}$.
		\end{itemize}
	\end{rk}
	\begin{rk}
		Proposition \ref{Lojasiewicz ineq weighted} is stated in terms of function spaces modelled on symmetric $2$-tensors, the only reason for this restriction being its main application to the functional $\lambda_{\operatorname{ALE}}$ defined on metrics. The condition $\frac{n-2}{2}<\tau< n-2$ is assumed to ensure the Fredholmness of the linearization of the gradient in the particular case of $\lambda_{\operatorname{ALE}}$.
	\end{rk}

	\subsection{Proof of Proposition \ref{Lojasiewicz ineq weighted}}\label{sec-prop-loja-ineq}~~\\
	
	Let us consider $G$ satisfying the assumptions of Proposition \ref{Lojasiewicz ineq weighted} and let us prove that we can reduce our context to the finite-dimensional case by Lyapunov-Schmidt reduction by following the scheme of proof of \cite[Section 7]{Col-Min-Ein-Tan-Con}. We first have to find a replacement for \cite[Lemma 7.5]{Col-Min-Ein-Tan-Con} which generally does not hold in the setting of weighted Hölder spaces because of the following remark.
	\begin{rk}\label{cokernel weighted spaces}
		Even if $L$ is selfadjoint, its kernel and its $L^2$-cokernel might be different because we are dealing with weighted H\"older spaces. For example, the $L^2$-cokernel of the Lichnerowicz operator $L_{g_b} : C^{2,\alpha}_{\tau}(S^2T^*N)\to C^{0,\alpha}_{\tau+2}(S^2T^*N)$, that is the space $\mathbf{C}$ such that $L_{g_b}(C^{2,\alpha}_{\tau}) = \mathbf{C}^\perp \cap C^{0,\alpha}_{\tau+2}$ is the kernel of $L$ on $C^{k,\alpha}_{n-(\tau+2)}(S^2T^*N)$ where $-2<n-(\tau+2)<\frac{n}{2}-1$ (see Note \ref{note L2 cokernel}) which is larger than $\mathbf{K}$ if $\tau>n-2$. We therefore always have $\mathbf{K}\subset\mathbf{C}$ here and the inclusion can be strict in our applications because there are asymptotically constant $2$-tensors in the kernel of the Lichnerowicz operator on a Ricci-flat ALE metrics if we do not impose a decay at infinity.
	\end{rk}
	Denote by $\mathbf{K}$ the finite dimensional kernel (and $L^2$-cokernel for $0<\tau<n-2$) of $L : C^{2,\alpha}_{\tau}(S^2T^*N)\cap E\to C^{0,\alpha}_{\tau+2}(S^2T^*N)\cap F$, and define $\Pi_\mathbf{K}$ to be the associated $L^2$-projection onto $\mathbf{K}$. Define the mapping:
	\begin{equation*}
	\mathcal{N} =\nabla G + \Pi_\mathbf{K}.
	\end{equation*}
	\begin{lemma}\label{reduction LS}
		There is an open subset $$\mathcal{U}\subset  C^{0,\alpha}_{\tau+2}(S^2T^*N)\,\cap\,F,$$ about $0$ and a map $\Phi : \mathcal{U}\to C^{2,\alpha}_{\tau}(S^2T^*N)\cap E$ with $\Phi(0) = 0$, and $C>0$, so that for any $x,y\in \mathcal{U}$ and $z\in C^{2,\alpha}_{\tau}(S^2T^*N)\,\cap E$ sufficiently small,
		\begin{itemize}
			\item $\Phi\circ \mathcal{N}(z) = z$ and $\mathcal{N}\circ\Phi(x) = x$,
			\item $ \|\Phi(x)\|_{C^{2,\alpha}_{\tau}}\leq C \|x\|_{C^{0,\alpha}_{\tau+2}} $ and $ \|\Phi(x)-\Phi(y)\|_{H^2_{\frac{n}{2}-1}}\leq C \|x-y\|_{L^2_{\frac{n}{2}+1}} $,
			\item the function $ f := G\circ \Phi$ is analytic on $\mathcal{U}$. In particular, it is analytic on $\mathbf{K}$.
		\end{itemize}
	\end{lemma}
	\begin{proof}		
		By assumptions (\ref{item-0}) and (\ref{item-5}), the mapping $\mathcal{N} : C^{2,\alpha}_{\tau}\,\cap\, E\to   C^{0,\alpha}_{\tau+2}\,\cap\,F$ is $C^1$ and its Fr\'echet derivative at $0$ is 
		$$D_0\mathcal{N} = L + \Pi_\mathbf{K}.$$
		Note that since $L$ is Fredholm of index $0$ by assumptions (\ref{item-4}) and (\ref{item-5}), it is enough to prove that $D_0\mathcal{N}$ is injective and Fredholm in order to use the implicit function theorem to define $\Phi$. The kernel $\mathbf{K}$ being finite dimensional, the projection $\Pi_{\mathbf{K}}$ is a compact operator which implies that $D_0\mathcal{N}$ is Fredholm. Now, by assumption (\ref{item-5}), the $L^2$-cokernel of $L$ is $\mathbf{K}$. Therefore, if $L(x) + \Pi_\mathbf{K}(x) = 0$, then, by projecting on $\mathbf{K}^\perp$, we have $L(x) = 0$, hence $x\in \mathbf{K}$ and $\Pi_\mathbf{K}(x) =0$, thus, finally $x = 0$. We can then conclude like in \cite[Lemma 7.5]{Col-Min-Ein-Tan-Con} by using the implicit function theorem, Lemma \ref{th fcts implicites}. 
		
		The bound $ \|\Phi(x)\|_{C^{2,\alpha}_{\tau}}\leq C \|x\|_{C^{0,\alpha}_{\tau+2}} $ comes from the integral mean value theorem in Banach spaces and the fact that 
		\begin{equation}
		D_y\Phi = \big(D_{\Phi(y)}\mathcal{N}\big)^{-1}\label{inverse linearisation Phi}
		\end{equation}
		is continuous and bounded from $\mathcal{U} \subset C^{0,\alpha}_{\tau+2}\,\cap\,F$ to $C^{2,\alpha}_{\tau}$.
		
		The Lipschitz bound $ \|\Phi(x)-\Phi(y)\|_{H^2_{\frac{n}{2}-1}}\leq C \|x-y\|_{L^2_{\frac{n}{2}+1}} $ for sufficiently small $x,y\in \mathcal{U}$ is equivalent to $ \|x-y\|_{H^2_{\frac{n}{2}-1}}\leq C \|\mathcal{N}(x)-\mathcal{N}(y)\|_{L^2_{\frac{n}{2}+1}} $ for sufficiently small $x,y\in \mathcal{O}$. This in turn is implied if $ \|x-y\|_{H^2_{\frac{n}{2}-1}}\leq C \|D_0\mathcal{N}(x-y)\|_{L^2_{\frac{n}{2}+1}} $ for sufficiently small $x,y\in \mathcal{O}$ by assumption (\ref{item-0-bis}). Now, the same reasoning that led us to prove that $\mathcal{N}$ is a local diffeomorphism at $0$ between weighted H\"older spaces, $L+\Pi_{\mathbf{K}}:H^{2}_{\frac{n}{2}-1}\,\cap\,E\rightarrow L^2_{\frac{n}{2}+1}\,\cap\,F$ is an isomorphism of Banach spaces by assumptions (\ref{item-1}) and (\ref{item-2}): this implies the desired lower bound on $\|D_0\mathcal{N}(x-y)\|_{L^2_{\frac{n}{2}+1}}$ since $L+\Pi_{\mathbf{K}}=D_0\mathcal{N}$ on $C^{2,\alpha}_{\tau}\,\cap\,E$.

		The analyticity of $f$ on $\mathcal{U}$ (and therefore on $\mathbf{K}$ by assumption (\ref{item-5})) comes from the analyticity of $G$ and that of $\Phi$ ensured by the analytic implicit function theorem stated in Lemma \ref{th fcts implicites}.
	\end{proof}
	\begin{rk}
		In case the $L^2$-cokernel $\mathbf{C}$ and the kernel $\mathbf{K}$ are different, one can instead define $$\mathcal{N} := \Pi_{\mathbf{C}^\perp}\circ \nabla G +\Pi_\mathbf{K},$$
		and define $\Phi$ on its image. This however induces some technical difficulties in weighted spaces. Since we do not need it presently we only considered the case when $\mathbf{C} = \mathbf{K}$.
	\end{rk}
	
	Let us now adapt \cite[Lemma 7.10]{Col-Min-Ein-Tan-Con} to our slightly different case.
	\begin{lemma}\label{control nabla f Pi nabla G}
		There exists $C>0$ such that for any $x$ in a small enough neighborhood of $0$ in $C^{2,\alpha}_{\tau}(S^2T^*N)\,\cap\, E$, we have 
		\begin{equation}
		\|\nabla f(\Pi_\mathbf{K}(x))\|_{L^2_{\frac{n}{2}-1}}\leq C\|\nabla G(x)\|_{L^2_{\frac{n}{2}+1}}.\label{est-nabla-f-nabla-G}
		\end{equation}
		More generally, if $y_t:=\Pi_{\mathbf{K}}(x)+t\nabla G(x)$, $t\in[0,1]$, for $x$ in a small enough neighborhood of $0$ in $C^{2,\alpha}_{\tau}(S^2T^*N)\,\cap\, E$, then:
		\begin{equation}
		\|\nabla f(y_t)\|_{L^2_{\frac{n}{2}-1}}\leq C\|\nabla G(x)\|_{L^2_{\frac{n}{2}+1}},\label{est-nabla-f-nabla-G-path}
		\end{equation}
		for some positive constant $C$ independent of $x$ and $t$.
	\end{lemma}
	\begin{proof}
		For $y\in \mathcal{U}$ sufficiently small, since $f = G\circ \Phi$,  we have $D_yf(v)=D_{\Phi(y)}G\circ D_{y}\Phi(v)$ for $v\in L^{2}_{\frac{n}{2}+1}$. In particular, if $v\in L^{2}_{\frac{n}{2}+1}$
\begin{equation*}
\begin{split}
|D_yf(v)|\leq\,& \|\nabla G(\Phi(y))\|_{L^2_{\frac{n}{2}+1}}\|D_y\Phi(v)\|_{H^2_{\frac{n}{2}-1}}\\
\leq\,&C\|\nabla G(\Phi(y))\|_{L^2_{\frac{n}{2}+1}}\|v\|_{L^2_{\frac{n}{2}+1}},
\end{split}
\end{equation*}
where $C$ is a positive constant independent of $v\in L^{2}_{\frac{n}{2}+1}$. Here, we have used the Lipschitz bound on $\Phi$ established in Lemma \ref{reduction LS}. By definition of the gradient of $f$, one gets: 
		\begin{equation*}
		\|\nabla f(y)\|_{L^2_{\frac{n}{2}-1}}\leq C\|\nabla G\circ\Phi(y)\|_{L^2_{\frac{n}{2}+1}},
		\end{equation*}
	In particular, for any $x\in C^{2,\alpha}_{\tau}\,\cap\, E$ sufficiently small, we have
		\begin{equation}
		\|\nabla f(\Pi_\mathbf{K}(x))\|_{L^2_{\frac{n}{2}-1}}\leq C\|\nabla G\circ\Phi( \Pi_\mathbf{K}(x))\|_{L^2_{\frac{n}{2}+1}}.\label{lovely-inequ-nabla-f-G}
		\end{equation}
		Now, since $x = \Phi\big(\Pi_\mathbf{K}(x)+\nabla G(x)\big)$, the Lipschitz bound (\ref{lip-bd-nabla-G}) for $\nabla G$ and the one obtained in Lemma \ref{reduction LS} for $\Phi$ yield
		\begin{align*}
		\|\nabla G(\Phi\circ \Pi_\mathbf{K}(x)) - \nabla G(x)\|_{L^2_{\frac{n}{2}+1}} &= \|\nabla G(\Phi( \Pi_\mathbf{K}(x))) -\nabla G(\Phi(\Pi_\mathbf{K}(x)+\nabla G(x)))\|_{L^2_{\frac{n}{2}+1}}\\
		&\leq C \|\Phi( \Pi_\mathbf{K}(x))- \Phi(\Pi_\mathbf{K}(x)+\nabla G(x))\|_{H^2_{\frac{n}{2}-1}}\\
		&\leq C \|\nabla G(x)\|_{L^2_{\frac{n}{2}+1}}.
		\end{align*}
		This ends the proof of the desired estimate (\ref{est-nabla-f-nabla-G}) by the triangular inequality.
		
		In order to prove (\ref{est-nabla-f-nabla-G-path}), notice that since $f = G\circ \Phi$ and consequently $\Phi\circ \mathcal{N} = \mathrm{Id}$, we have $G = f\circ \mathcal{N}$ which by differentiation implies that we have
\begin{equation}
    D_{\Phi(y_t)} G = D_{y_t}f\circ D_{\Phi(y_t)} \mathcal{N}.\label{nablaG}
\end{equation}
Now, since $ \nabla_{\Phi(y_t)}\mathcal{N} $ is invertible with a bounded inverse by Lemma \ref{reduction LS}, we deduce from \eqref{nablaG} that $D_{y_t}f = D_{\Phi(y_t)} G\circ (D_{\Phi(y_t)} \mathcal{N})^{-1}$ and therefore that
\begin{equation}
    \|D_{y_t}f\|_1\leq C \|D_{\Phi(y_t)} G\|_2,\label{ineq diff f diff G}
\end{equation} 
where the norm $\|.\|_1$ is that of operators from $L^2_{n+1}$ to $\mathbb{R}$ and the norm $\|.\|_2$ is that of operators from $H^2_{\frac{n}{2}-1}$ to $\mathbb{R}$. Since for the $L^2(g_b)$ scalar product, the dual of $L^2_{\frac{n}{2}+1}$ is identified with $L^2_{\frac{n}{2}-1}$ and that of $L^2_{\frac{n}{2}-1}$ with
$L^2_{\frac{n}{2}+1}$, \eqref{ineq diff f diff G} rewrites in the following way thanks to the $L^2(g_b)$ gradients:
\begin{equation}
    \|\nabla f(y_t)\|_{L^2_{\frac{n}{2}-1}}\leq C\|\nabla G (\Phi(y_t))\|_{L^2_{\frac{n}{2}+1}}.\label{nabla f nablaG}
\end{equation}

 Now, we invoke \eqref{nabla f nablaG} to observe that it remains to control $\|\nabla G(\Phi(y_t))\|_{L^2_{\frac{n}{2}+1}}$ from above by $\|\nabla G(x)\|_{L^2_{\frac{n}{2}+1}}$. For this, we use \eqref{lip-bd-nabla-G} which yields 
\begin{equation*}
\begin{split}
    \|\nabla G(\Phi(y_t))-\nabla G(x)\|_{L^2_{\frac{n}{2}+1}}
    &\leq C\|\Phi(y_t)-x\|_{H^2_{\frac{n}{2}-1}}\\
    &= C \|\Phi(y_t)-\Phi(y_1)\|_{H^2_{\frac{n}{2}-1}}\\
    &\leq C\|y_t-y_1\|_{L^2_{\frac{n}{2}+1}}\\
    &= C (1-t)\|\nabla G(x)\|_{L^2_{\frac{n}{2}+1}}.
    \end{split}
\end{equation*}
This implies that $\|\nabla G(\Phi(y_t))\|_{L^2_{\frac{n}{2}+1}}\leq (1+C)\|\nabla G(x)\|_{L^2_{\frac{n}{2}+1}}$ by the triangular inequality and therefore by \eqref{nabla f nablaG},
\begin{equation}
    \|\nabla f(y_t)\|_{L^2_{\frac{n}{2}-1}}\leq C \|\nabla G(x)\|_{L^2_{\frac{n}{2}+1}}.\label{nabla f nablaG(x)}
\end{equation}
This ends the proof of the desired estimate.
	\end{proof}
	
	Let us now show that an adaptation of \cite[Lemma 7.15]{Col-Min-Ein-Tan-Con} to our situation yields a control in $L^2_{\frac{n}{2}+1}$ only. 
	
	\begin{lemma}\label{loja orth noyau 1}
		There exists $C>0$ such that for sufficiently small $x\in C^{2,\alpha}_{\tau}(S^2T^*N)\cap E$, we have
		\begin{equation}
		|G(x)-f(\Pi_\mathbf{K}(x))|\leq C \|\nabla G(x)\|_{L^2_{\frac{n}{2}+1}}^2.\label{est G - f Pi}
		\end{equation}
	\end{lemma}
	\begin{proof}
		Define for all $t\in [0,1]$, 
		$$y_t:= \Pi_\mathbf{K}(x) + t\nabla G(x),$$
		for which we have $\Phi(y_1) = x$, $y_0 = \Pi_\mathbf{K}(x)$ and $\frac{d}{dt}y_t =\nabla G(x)$.

		By integration, we have
		\begin{align}
		G(x)-f(\Pi_\mathbf{K}(x)) &= f(y_1)-f(y_0) \nonumber\\
		&= \int_0^1 \langle \nabla f(y_t) , \nabla G(x) \rangle_{L^2} \,dt,\label{difference-G-fPi}
		\end{align}
		where $\nabla f(y_t)$ (respectively $\nabla G(x)$) is interpreted as an element of $L^2_{\frac{n}{2}-1}$ (respectively $L^2_{\frac{n}{2}+1}$). Thanks to [\eqref{est-nabla-f-nabla-G-path}, Lemma \ref{control nabla f Pi nabla G}], there exists $C>0$ such that for all $t\in[0,1]$, we have 
		\begin{equation*}
		\|\nabla f (y_t) \|_{L^2_{\frac{n}{2}-1}}\leq C \|\nabla G(x) \|_{L^2_{\frac{n}{2}+1}}.
		\end{equation*}

			The estimate \eqref{est G - f Pi} then comes from Cauchy-Schwarz inequality.
	\end{proof}
	
	We can then conclude exactly like in the end of the proof of \cite[Theorem 7.3]{Col-Min-Ein-Tan-Con}, by using the finite dimensional \L{}ojasiewicz inequality on $\mathbf{K}$. Note that this is the only step for which the analyticity of the functional is used.
	
	Denote $f_\mathbf{K}$ the restriction of $f$ to the finite-dimensional space $\mathbf{K}$ which is an analytic function. Let $x\in C^{2,\alpha}_\tau\cap E$ be small enough. Thanks to the estimate \eqref{est-nabla-f-nabla-G}, and thanks to the finite-dimensional \L{}ojasiewicz inequality \cite{loj}, we have for some $0<\theta\leq 1$,
	\begin{align}
	C^2\|\nabla G(x)\|_{L^2_{\frac{n}{2}+1}}^2&\geq C\|\nabla f (\Pi_\mathbf{K}(x))\|_{L^2_{\frac{n}{2}-1}}^2\\
	&\geq \|\nabla f_{|\mathbf{K}}(\Pi_\mathbf{K}(x))\|_{L^2_{\frac{n}{2}-1}}^2\\
	&\geq |f_\mathbf{K}(\Pi_\mathbf{K}(x))-f_\mathbf{K}(0)|^{2-\theta}\\
	&=|f(\Pi_\mathbf{K}(x))-G(0)|^{2-\theta},
	\end{align}
	where  $C$ denotes the positive constant of \eqref{est-nabla-f-nabla-G}. Here, we see $\mathbf{K}$ equipped with the $L^2_{\frac{n}{2}-1}$-norm (but any other norm would just change the constants as the dimension is finite). We can then finally use the estimate \eqref{est G - f Pi} together with the triangular inequality to obtain a general \L{}ojasiewicz inequality : 
	\begin{equation}
	|G(x)-G(0)|^{2-\theta}\leq C\|\nabla G(x)\|^2_{L^2_{\frac{n}{2}+1}},\label{loja L2-n/2+1}
	\end{equation}
	with $0<\theta\leq 1$. We therefore obtain a \L{}ojasiewicz inequality with a $L^2_{\frac{n}{2}+1}$-norm for the gradient.  The above constant $\theta$ moreover does not depend on $\alpha$ or $\tau$.
	
	\subsection{Proof of a general \L{}ojasiewicz inequality for $\lambda_{\operatorname{ALE}}$ on ALE metrics}\label{sec-proof-gal-loja}~~\\

	Let us now ensure that our functional $G=\lambda_{\operatorname{ALE}}$ satisfies the assumptions of Proposition \ref{Lojasiewicz ineq weighted}. We start by defining the set $E =\ker \div_{g_b}$. Let us spend some time understanding the image $L_{g_b}(C^{2,\alpha}_{\tau}\cap \ker\div_{g_b})$ which intuitively is $C^{0,\alpha}_{\tau+2}\cap \ker\div_{g_b}$. However, strictly speaking, we a priori cannot see $C^{0,\alpha}_{\tau+2}\cap \ker\div_{g_b}$ as a closed subset of $C^{0,\alpha}_{\tau+2}$ since the equation $\div_{g_b}h=0$ is not well-defined as it takes one derivative. It has to be understood in the weak sense. Let us be more precise about this in order to define the set $F:=\div^*_{g_b}(C^\infty_c(TN))^\perp$. 
	\begin{lemma}\label{definition E et F pour loja}
		Let us define $E := \ker \div_{g_b}$, $\mathbf{K} := \ker_{L^2} L_{g_b}$ and define $\Pi_\mathbf{K}$ the $L^2(g_b)$-projection on $\mathbf{K}$. Then the image $(L_{g_b} + \Pi_{\mathbf{K}})(C^{2,\alpha}_{\tau}\cap E)$ is $C^{0,\alpha}_{\tau+2}\cap \div^*_{g_b}(C^\infty_c(TN))^\perp$.
	\end{lemma}	
\begin{proof}
Consider $h\in C^{2,\alpha}_{\tau}\cap \ker\div_{g_b}$ and a smooth compactly supported vector field $X$ and denote the $L^2$-adjoint of $\div_{g_b}$ by $\div^{\ast}_{g_b}$. Notice that for such a vector field $X$, we have $\div^*_{g_b}X=-\frac{1}{2}\Li_X(g_b)$. Now, similarly to the proof of Lemma \ref{lemma-equiv-def-stable}, observe that,
\begin{align}
    \langle L_{g_b}h,\div^*_{g_b}X \rangle_{g_b} &= \langle h,L_{g_b}(\div^*_{g_b}X) \rangle_{g_b}\\
    &=\langle h,\div^*_{g_b}B_{g_b}\div^*_{g_b}X \rangle_{g_b}\\
    &=\langle \div_{g_b}h,B_{g_b}\div^*_{g_b}X \rangle_{g_b} = 0
\end{align}
and we see that since the elements of $\mathbf{K}$ are divergence-free according to Proposition \ref{prop-lic-fred}, the image of $L_{g_b} + \Pi_{\mathbf{K}}$ is included in $C^{0,\alpha}_{\tau+2}\cap \div^*_{g_b}(C^\infty_c(TN))^\perp$.

Conversely, let $v\in C^{0,\alpha}_{\tau+2}\cap \div^*_{g_b}(C^\infty_c(TN))^\perp$ and let us show that there exists $h_0\in C^{2,\alpha}_{\tau}\cap \ker\div_{g_b}$ such that $v=L_{g_b}h_0+\Pi_{\mathbf{K}}h_0$. 

Thanks to the Fredholm properties of $L_{g_b}$, see Proposition \ref{prop-lic-fred}, one has $(L_{g_b} + \Pi_{\mathbf{K}})(C^{2,\alpha}_{\tau}) = C^{0,\alpha}_{\tau+2}$. 
In particular, there exists $h\in C^{2,\alpha}_{\tau}$ such that $v=L_{g_b}h + \Pi_{\mathbf{K}}h$.

Now, Proposition \ref{prop-decomp-2-tensor} ensures that we have a unique decomposition $h = h_0 + \div^*_{g_b}V$ for a vector field $V\in C^{3,\alpha}_{\tau-1}$ and a symmetric $2$-tensor $h_0\in C^{2,\alpha}_{\tau}\cap \ker\div_{g_b}$. Let us prove that $\div^*_{g_b}V=0$. This will prove the expected result.

Since $v\in  \div^*_{g_b}(C^\infty_c(TN))^\perp$ and by density,
$$\langle L_{g_b}h+\Pi_{\mathbf{K}}h,\div^{*}_{g_b}V\rangle_{L^2}=0.$$

Notice that $\langle \Pi_{\mathbf{K}}h,\div^{*}_{g_b}V\rangle_{L^2}=0$ independently since $\Pi_{\mathbf{K}}h\in \mathbf{K}$ is divergence-free.
 
On the other hand, by symmetry of $L_{g_b}$, 
\begin{equation}
\begin{split}
   0&=\langle L_{g_b}h,\div^*_{g_b}V\rangle_{L^2}\\
   &=\langle h,L_{g_b}\div^*_{g_b}V\rangle_{L^2}\\
   &=\langle h,\div^*_{g_b}B_{g_b}(\div^*_{g_b}V)\rangle_{L^2}\\
   &=\langle h_0,\div^*_{g_b}B_{g_b}(\div^*_{g_b}V)\rangle_{L^2}+\langle \div^*_{g_b}V,\div^*_{g_b}B_{g_b}(\div^*_{g_b}V)\rangle_{L^2}\\
   &=\langle \div^*_{g_b}V,\div^*_{g_b}B_{g_b}(\div^*_{g_b}V)\rangle_{L^2}.
   \end{split}
\end{equation}
Here, we have used integration by parts and the fact that $h_0$ is divergence-free in the last line. A similar computation to [\eqref{IBP-Bianchi-lemma-loc-stab}, Lemma \ref{lemma-equiv-def-stable}] tells us that $B_{g_b}(\div^*_{g_b}V)=0=\nabla^{g_b}\div_{g_b}V$. This in turn implies by Bochner formula for vector fields that $\Delta_{g_b}V=0$ since $g_b$ is Ricci-flat. The use of the maximum principle then shows that $V=0$. This fact ends the proof of the lemma.

\end{proof}

	We now check conditions [(\ref{item-0}), \eqref{lip-bd-nabla-G}, (\ref{item-0-bis})] from Proposition \ref{Lojasiewicz ineq weighted}.
	\begin{prop}[Energy estimates]\label{prop-energy-est}
		Let $(N^n,g_b)$ be an ALE Ricci-flat metric, asymptotic to $\RR^n\slash\Gamma$, for some finite subgroup $\Gamma$ of $SO(n)$ acting freely on $\mathbb{S}^{n-1}$ and let $\tau\in (\frac{n-2}{2},n-2)$. Then there exists a neighborhood $B_{C^{2,\alpha}_{\tau}}(g_b,\varepsilon)$ of $g_b$ such that the following energy estimate holds true:
		\begin{equation}
		\|\nabla\lambda_{\operatorname{ALE}}(g_2)-\nabla\lambda_{\operatorname{ALE}}(g_1)\|_{L^2_{\frac{n}{2}+1}}\leq C(n,g_b,\varepsilon)\|g_2-g_1\|_{H^2_{\frac{n}{2}-1}}, \label{energy-est}
		\end{equation}
		for any metric $g_2$, $g_1$ in $B_{C^{2,\alpha}_{\tau}}(g_b,\varepsilon)$.
		Moreover, the differential of the map $g\in B_{C^{2,\alpha}_{\tau}}(g_b,\varepsilon)\rightarrow \nabla\lambda_{\operatorname{ALE}}(g)\in B_{C^{0,\alpha}_{\tau}}(0,\varepsilon)$ satisfies:
		\begin{equation}
		\|D_{g_2}\nabla\lambda_{\operatorname{ALE}}-D_{g_1}\nabla\lambda_{\operatorname{ALE}}(h)\|_{L^2_{\frac{n}{2}+1}}\leq C(n,g_b,\varepsilon)\|g_2-g_1\|_{C^{2,\alpha}_{\tau}}\|h\|_{H^2_{\frac{n}{2}-1}}, \label{energy-est-bis}
		\end{equation}
		for any metric $g_2$, $g_1$ in $B_{C^{2,\alpha}_{\tau}}(g_b,\varepsilon)$ and $h\in H^2_{\frac{n}{2}-1}$.
	\end{prop}

	\begin{proof}
		Let us fix a neighborhood $B_{C^{2,\alpha}_{\tau}}(g_b,\varepsilon)$ such that the properties on the potential function $f_g$ established in Proposition \ref{prop-pot-fct} hold true.
		
		Strictly speaking, the symbol $\nabla \lambda_{\operatorname{ALE}}(g)$ means the gradient of $\lambda_{\operatorname{ALE}}$ at $g$ with respect to the Hilbert space $L^2(d\mu_{g_b})$. Thanks to Proposition \ref{lambdaALE analytic}, $\nabla\lambda_{\operatorname{ALE}}(g)$ can be computed in coordinates from the $L^2(e^{-f_g}d\mu_g)$-gradient of $\lambda_{\operatorname{ALE}}$ as follows. Indeed, consider an orthonormal frame with respect to $g_b$ at a point such that $g_{ij}=(1+\lambda_i)\delta_{ij}$ and observe that:
		\begin{equation*}
\nabla\lambda_{\operatorname{ALE}}(g)_{ij}=-e^{-f_g}\frac{d\mu_g}{d\mu_{g_b}}(1+\lambda_i)^{-1}(1+\lambda_j)^{-1}\left(\Ric(g)+\nabla^{g,2}f_g\right)_{ij},
\end{equation*}
which can be schematically written as 
\begin{equation}
\nabla\lambda_{\operatorname{ALE}}(g)=-e^{-f_g}\frac{d\mu_g}{d\mu_{g_b}}g^{-1}\ast g^{-1}\ast g_b\ast g_b\ast (\Ric(g)+\nabla^{g,2}f_g).\label{l2-grad-lam-gb}
\end{equation}
The estimate (\ref{l2-grad-lam-gb}) then implies that for two metrics $g_i$, $i=1,2$ in a $C^{2,\alpha}_{\tau}$-neighborhood of $g_b$,
\begin{equation}
\begin{split}
|\nabla\lambda_{\operatorname{ALE}}(g_2)-\nabla\lambda_{\operatorname{ALE}}(g_1)|_{g_b}\leq& C |\Ric(g_2)+\nabla^{g_2,2}f_{g_2}-\Ric(g_1)-\nabla^{g_1,2}f_{g_1}|_{g_b}\\
&+C\left(|g_2-g_1|_{g_b}+|f_{g_2}-f_{g_1}|\right)\rho_{g_b}^{-2},
\label{red-case-l2-wei-grad}
\end{split}
\end{equation}
where $C=C(n,g_b,\varepsilon,\tau)$ are positive constants. Here we have used the weakened assumption on the decay of the Ricci curvatures $\Ric(g_i)=O(\rho_{g_b}^{-2-\tau})=O(\rho_{g_b}^{-2})$ together with the decay of the Hessians $\nabla^{g_i,2}f_{g_i}=O(\rho_{g_b}^{-2})$ from Proposition \ref{prop-pot-fct}. In particular, Hardy's inequality from Theorem \ref{thm-min-har-inequ} and Proposition \ref{prop-ene-pot-fct} show that:
\begin{equation}
\begin{split}
\|\nabla\lambda_{\operatorname{ALE}}(g_2)-\nabla\lambda_{\operatorname{ALE}}(g_1)\|_{L^2_{\frac{n}{2}+1}}&\leq C \|\Ric(g_2)+\nabla^{g_2,2}f_{g_2}-\Ric(g_1)-\nabla^{g_1,2}f_{g_1}\|_{L^2_{\frac{n}{2}+1}}\\
&\quad+C\left(\|g_2-g_1\|_{H^2_{\frac{n}{2}-1}}+\|\nabla^{g_b}(f_{g_2}-f_{g_1})\|_{L^2_{\frac{n}{2}}}\right)\\
&\leq C \|\Ric(g_2)+\nabla^{g_2,2}f_{g_2}-\Ric(g_1)-\nabla^{g_1,2}f_{g_1}\|_{L^2_{\frac{n}{2}+1}}\\
&\quad+C\|g_2-g_1\|_{H^2_{\frac{n}{2}-1}}.\label{prel-est-ener-lam}
\end{split}
\end{equation}

		We first do the proof of the energy estimates (\ref{energy-est}) in case one of the two metrics is $g_b$, which implies in particular that $\nabla\lambda_{\operatorname{ALE}}(g_b)=0$.
		Let us write $h:=g-g_b$ where $g\in B_{C^{2,\alpha}_{\tau}}(g_b,\varepsilon)$. 
				By Lemmata \ref{Ric-lin-lemma-app} and \ref{lem-lin-equ-Ric-first-var}, linearizing the Ricci curvature of $g$ at $g_b$ gives schematically:
		\begin{equation}
		\begin{split}
		-2\Ric(g)&=-2\Ric_{g_b}+L_{g_b}h-\Li_{B_{g_b}(h)}(g_b)+Q(h,\nabla^{g_b}h,\nabla^{g_b,2}h)\\
		&=L_{g_b}h-\Li_{B_{g_b}(h)}(g_b)+Q(h,\nabla^{g_b}h,\nabla^{g_b,2}h),\label{ric-lin-bianchi}
		\end{split}
		\end{equation}
		where $Q(h,\nabla^{g_b}h,\nabla^{g_b,2}h)$ satisfies pointwise 
		\begin{equation*}
		\left|Q(h,\nabla^{g_b}h,\nabla^{g_b,2}h)\right|_{g_b}\leq C(n,g_b,\varepsilon)\left(|\Rm(g_b)|_{g_b}|h|^2_{g_b}+|\nabla^{g_b}h|_{g_b}^2+|h|_{g_b}|\nabla^{g_b,2}h|_{g_b}\right).
		\end{equation*}
		Therefore, by using the fact that $g=g_b+h\in B_{C^{2,\alpha}_{\tau}}(g_b,\varepsilon)$,
		\begin{eqnarray}
		\|\Ric(g)\|_{L^2_{\frac{n}{2}+1}}\leq C(n,g_b,\varepsilon)\|h\|_{H^2_{\frac{n}{2}-1}},\label{L^2-a-priori-bound-Ric}
		\end{eqnarray}
		where we suppressed the dependence of the norm on $S^2T^*N$.\\

		Now, let us estimate $\|\nabla^{g,2}f_g\|_{L^2}$. According to the Bochner formula applied to the smooth metric measure space $\left(N,g,\nabla^gf_g\right)$:
		\begin{equation}
		\begin{split}
		\left(\Delta_g-\langle\nabla^gf_g,\nabla^g\cdot\rangle_g\right)|\nabla^g f_g|_g^2&=2|\nabla^{g,2}f_g|^2_g+2(\Ric(g)+\nabla^{g,2}f_g)(\nabla^gf_g,\nabla^gf_g)\\
		&+2\left\langle\nabla^g\left(\Delta_gf_g-\langle\nabla^gf_g,\nabla^gf_g\rangle_g\right),\nabla^gf_g\right\rangle_g.\label{bochner-for-hess-f}
		\end{split}
		\end{equation}
		Once (\ref{bochner-for-hess-f}) is multiplied by $\rho_{g_b}^2$, an integration by parts, legitimated by the decay of $f_g$ established in Proposition \ref{prop-pot-fct}, gives:
		\begin{equation}
		\begin{split}\label{IPP-Monster}
		\|\nabla^{g,2}f_g\|&_{L^2(\rho_{g_b}^2e^{-f_g}d\mu_g)}^2=\|\Delta_{g,f_g}f_g\|^2_{L^2(\rho_{g_b}^2e^{-f_g}d\mu_g)}-\int_N(\Ric(g)+\nabla^{g,2}f_g)(\nabla^gf_g,\nabla^gf_g)\,\rho_{g_b}^2e^{-f_g}d\mu_g\\
		&+2\int_N\Delta_{g,f_g}f_g\,\rho_{g_b}\,\langle\nabla^{g}\rho_{g_b},\nabla^gf_g\rangle_g\,e^{-f_g}d\mu_g-\int_N\nabla^{g,2}f_g(\nabla^gf_g,\nabla^g(\rho_{g_b}^2e^{-f_g}))\,d\mu_g\\
		&\leq c\||\nabla^gf_g|^2_g+\R_g\|_{L^2(\rho_{g_b}^2e^{-f_g}d\mu_g)}^2+\frac{1}{2}\|\nabla^{g,2}f_g\|_{L^2(\rho_{g_b}^2e^{-f_g}d\mu_g)}^2+c\|\nabla^gf_g\|_{L^2(e^{-f_g}d\mu_g)}^2\\
		&+c\left(\||\nabla^gf_g|^2_g\|^2_{L^2(\rho_{g_b}^2e^{-f_g}d\mu_g)}+\|\Ric(g)\|_{L^2(\rho_{g_b}^2e^{-f_g}d\mu_g)}^2\right)\\
		&\leq\frac{1}{2}\|\nabla^{g,2}f_g\|_{L^2(\rho_{g_b}^2e^{-f_g}d\mu_g)}^2\\
		&+c\left(\||\nabla^gf_g|^2_g\|^2_{L^2(\rho_{g_b}^2e^{-f_g}d\mu_g)}+\|\nabla^gf_g\|^2_{L^2(e^{-f_g}d\mu_g)}+\|\Ric(g)\|_{L^2(\rho_{g_b}^2e^{-f_g}d\mu_g)}^2\right),
		\end{split}
		\end{equation}
		where we use (\ref{equ-criti-lambda-pot}) together with Young's inequality in the third line and where $c=c(n)$ denotes a positive constant that may vary from line to line. By Proposition \ref{prop-pot-fct}, $\rho_{g_b}|\nabla^gf_g|_g\leq C(n,\varepsilon,g_b)$ pointwise. Consequently, we get:
		\begin{equation*}
		\|\nabla^{g,2}f_g\|_{L^2(\rho_{g_b}^2e^{-f_g}d\mu_g)}^2\leq C(n,\varepsilon,g_b)\left(\|\nabla^gf_g\|^2_{L^2(e^{-f_g}d\mu_g)}+\|\Ric(g)\|_{L^2(\rho_{g_b}^2e^{-f_g}d\mu_g)}^2\right),
		\end{equation*}
		which in turn implies:
		\begin{equation}
		\|\nabla^{g,2}f_g\|_{L^2(\rho_{g_b}^2e^{-f_g}d\mu_g)}^2\leq C(n,\varepsilon,g_b)\left(\|\nabla^gf_g\|^2_{L^2(e^{-f_g}d\mu_g)}+\|h\|_{H^2_{\frac{n}{2}-1}}^2\right),\label{intermediate-est-hessian}
		\end{equation}
		thanks to (\ref{L^2-a-priori-bound-Ric}). Indeed, by Proposition \ref{existence propriete-wg}, 
		\begin{eqnarray}
		|w_g-1|=|w_g-w_{g_b}|\leq \frac{1}{2},\quad g\in B_{C^{2,\alpha}_{\tau}}(g_b,\varepsilon),\label{c^0-est-pot-fct}
		\end{eqnarray}
		if $\varepsilon$ is chosen small enough, so that the metrics $w_g\cdot g$ and $g$ are uniformly bi-Lipschitz on $N$.

		Finally, it remains to estimate $\|\nabla^gf_g\|_{L^2(e^{-f_g}d\mu_g)}$ or equivalently $\|\nabla^gf_g\|_{L^2}$  from above. Concatenating (\ref{intermediate-est-hessian}) and [(\ref{grad-est-int-ene}), Proposition \ref{prop-ene-pot-fct}] ends the proof of (\ref{energy-est}) in case $g_1=g_b$ once (\ref{prel-est-ener-lam}) is invoked.\\
		
		Let us treat the general case, i.e. let $g_1$ and $g_2$ be two metrics in $B_{C^{2,\alpha}_{\tau}}(g_b,\varepsilon)$, let $h:=g_2-g_1$ and $g_t:=g_1+(t-1)h$ for $t\in[1,2]$, and let us estimate the difference of $\nabla\lambda_{\operatorname{ALE}}(g_2)$ and $\nabla\lambda_{\operatorname{ALE}}(g_1)$ with the help of Lemmata \ref{Ric-lin-lemma-app} and \ref{lem-lin-equ-Ric-first-var} as follows:
		\begin{equation}
		\begin{split}\label{first-diff-lambda}
		&-2\Ric(g_2)-\Li_{\nabla^{g_2}f_{g_2}}(g_2)+2\Ric(g_1)+\Li_{\nabla^{g_1}f_{g_1}}(g_1)\\
		=&L_{g_1}h-\Li_{B_{g_1}(h)}(g_2)+\int_1^2\frac{\partial}{\partial t}\Li_{\nabla^{g_t}f_{g_t}}(g_t)\,dt+Q(h,\nabla^{g_1}h,\nabla^{g_1,2}h).
		\end{split}
		\end{equation}
		Now, observe that:
		\begin{equation}
		\begin{split}\label{int-first-der-lie-f}
		\int_1^2\frac{\partial}{\partial t}\Li_{\nabla^{g_t}f_{g_t}}(g_t)\,dt&=\int_1^2-\Li_{h(\nabla^{g_t}f_{g_t})}(g_t)+\Li_{\nabla^{g_t}f_{g_t}}(h)+\Li_{\nabla^{g_t}(\delta_{g_t}f(h))}(g_t)\,dt.
		\end{split}
		\end{equation}
		On the one hand, one has
		\begin{equation}
		\begin{split}\label{one-first-term-lie}
		\int_1^2\left|\Li_{\nabla^{g_t}f_{g_t}}(h)\right|_{g_1}\,dt&\lesssim\int_1^2|\nabla^{g_t}h|_{g_1}|\nabla^{g_t}f_{g_t}|_{g_1}+|h|_{g_1}|\nabla^{g_t,2}f_{g_t}|_{g_1} \,dt\\
		&\lesssim \rho_{g_b}^{-1}|\nabla^{g_1}h|_{g_1}\|h\|_{C^{2,\alpha}_{\tau}}+\rho_{g_b}^{-2}|h|_{g_1}\|h\|_{C^{2,\alpha}_{\tau}}.
		\end{split}
		\end{equation}
		Here, the sign $\lesssim$ means less than or equal up to a positive constant uniform in $t\in[1,2]$ which might depend on $n$, $g_b$, $\varepsilon$.
		On the other hand, one has similarly,
		\begin{equation}
		\begin{split}\label{sec-first-term-lie}
		\left|\int_1^2\Li_{h(\nabla^{g_t}f_{g_t})}(g_t)\,dt\right|_{g_1}&\lesssim\int_1^2|\nabla^{g_t}(h(\nabla^{g_t}f_{g_t}))|_{g_1} \,dt\\
		&\leq \rho_{g_b}^{-1}|\nabla^{g_1}h|_{g_1}\|h\|_{C^{2,\alpha}_{\tau}}+\rho_{g_b}^{-2}|h|_{g_1}\|h\|_{C^{2,\alpha}_{\tau}}.
		\end{split}
		\end{equation}
		Then by (\ref{first-diff-lambda}) together with (\ref{int-first-der-lie-f}), (\ref{one-first-term-lie}) and (\ref{sec-first-term-lie}), one has pointwise:
		\begin{equation*}
		\begin{split}
		|\Ric(g_2)+\nabla^{g_2,2}f_{g_2}-&\Ric(g_1)-\nabla^{g_1,2}f_{g_1}|_{g_1}
		\\
		&\lesssim|L_{g_1}h|_{g_1}+\left|\Li_{B_{g_1}(h)}(g_1)\right|_{g_1}+\int_1^2|\Li_{\nabla^{g_t}(\delta_{g_t}f(h))}(g_t)|_{g_1}\,dt\\
		&+|\Rm(g_1)|_{g_1}|h|^2_{g_1}+|\nabla^{g_1}h|_{g_1}^2+|h|_{g_1}|\nabla^{g_1,2}h|_{g_1}\\
		&+\rho_{g_b}^{-1}|\nabla^{g_1}h|_{g_1}\|h\|_{C^{2,\alpha}_{\tau}}+\rho_{g_b}^{-2}|h|_{g_1}\|h\|_{C^{2,\alpha}_{\tau}}\\
		&\lesssim |L_{g_1}h|_{g_1}+\left|\Li_{B_{g_1}(h)}(g_1)\right|_{g_1}+\int_1^2|\Li_{\nabla^{g_t}(\delta_{g_t}f(h))}(g_t)|_{g_1}\,dt\\
		&+\rho_{g_b}^{-1}|\nabla^{g_1}h|_{g_1}+\rho_{g_b}^{-2}|h|_{g_1},
		\end{split}
		\end{equation*}
		where we have used the fact that $g_i\in B_{C^{2,\alpha}_{\tau}}(g_b,\varepsilon)$, $i=1,2$ in the last line.

		Let us notice that by definition of the space $H^2_{\frac{n}{2}-1}$ and that of the linear Bianchi gauge given in \eqref{defn-bianchi-op},
		\begin{equation*}
		\begin{split}
		\|L_{g_1}h\|_{L^2_{\frac{n}{2}+1}}&\lesssim\|\Rm(g_1)\ast h\|_{L^2_{\frac{n}{2}+1}}+\|\nabla^{g_1}h\|_{L^2}+\|\nabla^{g_1,2}h\|_{L^2_{\frac{n}{2}+1}}\\
		&\lesssim \|h\|_{H^2_{\frac{n}{2}-1}},\\
		\|\Li_{B_{g_1}(h)}(g_1)\|_{L^2_{\frac{n}{2}+1}}&\lesssim\left\|\nabla^{g_1}\left(\div_{g_1}h-\frac{\nabla^{g_1}\tr_{g_1}h}{2}\right)\right\|_{L^2_{\frac{n}{2}+1}}\\
		&\lesssim \|\nabla^{g_1,2}h\|_{L^2_{\frac{n}{2}+1}} \lesssim \|h\|_{H^2_{\frac{n}{2}-1}},
		\end{split}
		\end{equation*}
		which implies that:
		\begin{equation}
		\begin{split}\label{last-but-not-least}
		\|\Ric(g_2)+\nabla^{g_2,2}f_{g_2}-&\Ric(g_1)-\nabla^{g_1,2}f_{g_1}\|_{L^2_{\frac{n}{2}+1}}\\
		&\lesssim \|h\|_{H^2_{\frac{n}{2}-1}}+\left\|\int_1^2\Li_{\nabla^{g_t}(\delta_{g_t}f(h))}(g_t)\,dt\right\|_{L^2_{\frac{n}{2}+1}}\\
		&\lesssim \|h\|_{H^2_{\frac{n}{2}-1}}+\int_1^2\|\Li_{\nabla^{g_t}(\delta_{g_t}f(h))}(g_t)\|_{L^2(\rho_{g_b}^2d\mu_{g_t})}\,dt.
		\end{split}
		\end{equation}
		Here, we have used Cauchy-Schwarz inequality together with the fact that the metrics $g_t$, $t\in[1,2]$ are uniformly equivalent in the last line.
		
		According to (\ref{prel-est-ener-lam}) and (\ref{last-but-not-least}), it remains to control $\|\Li_{\nabla^{g_t}(\delta_{g_t}f(h))}(g_t)\|_{L^2(e^{-f_{g_t}}d\mu_{g_t})}$ from above uniformly in $t\in[1,2]$. By using the Bochner formula for the smooth metric measure space $\left(N^n,g_t,\nabla^{g_t}f_{g_t}\right)$ endowed with the measure $\mu_t:=e^{-f_{g_t}}d\mu_{g_t}$: 
		\begin{equation*}
		\begin{split}
		\Delta_{g_t,f_{g_t}}|\nabla^{g_t} \left(\delta_{g_t}f(h)\right)|_{g_t}^2&:=\left(\Delta_{g_t}-\langle\nabla^{g_t}f_{g_t},\nabla^{g_t}\cdot\rangle_{g_t}\right)|\nabla^{g_t} \left(\delta_{g_t}f(h)\right)|_{g_t}^2\\
		&=2|\nabla^{g_t,2}\delta_{g_t}f(h)|^2_{g_t}+2\left\langle\nabla^{g_t}\Delta_{g_t,f_{g_t}}\left(\delta_{g_t}f(h)\right),\nabla^{g_t}\left(\delta_{g_t}f(h)\right)\right\rangle_{g_t}\\
		&+2(\Ric(g_t)+\nabla^{g_t,2}f_{g_t})(\nabla^{g_t}\left(\delta_{g_t}f(h)\right),\nabla^{g_t}\left(\delta_{g_t}f(h)\right)).
		\end{split}
		\end{equation*}
		Therefore we proceed in the same way as we did in (\ref{IPP-Monster}) to get:
		\begin{equation}
		\begin{split}\label{est-hess-first-var-pot-fct-ene}
		\|\nabla^{g_t,2}\left(\delta_{g_t}f(h)\right)\|_{L^2(\rho_{g_b}^2d\mu_{t})}&\lesssim\|\Delta_{g_t,f_{g_t}}\left(\delta_{g_t}f(h)\right)\|_{L^2(\rho_{g_b}^2d\mu_{t})}+\|\nabla^{g_t}\left(\delta_{g_t}f(h)\right)\|_{L^2(d\mu_{t})}.
		\end{split}
		\end{equation}
		In order to estimate the righthand side of the previous inequality in terms of $\|h\|_{H^2_{\frac{n}{2}-1}}$ only, we invoke [(\ref{est-grad-first-der-pot-fct-ene}), (\ref{first-var-pot-fct-ell-equ-gal-case-prop}), Proposition \ref{prop-ene-pot-fct}] together with the triangular inequality.

		Coming back to (\ref{est-hess-first-var-pot-fct-ene}), this implies that:
		\begin{equation}
		\begin{split}
		\|\nabla^{g_t,2}\left(\delta_{g_t}f(h)\right)\|_{L^2(\rho_{g_b}^2d\mu_{g_t})}&\lesssim\|\Delta_{g_t,f_{g_t}}\left(\delta_{g_t}f(h)\right)\|_{L^2(\rho_{g_b}^2d\mu_{g_t})}+\|\nabla^{g_t}h\|_{L^2}\\
		&\lesssim \|\nabla^{g_1}h\|_{L^2}+\|\nabla^{g_1,2}h\|_{L^2_{\frac{n}{2}+1}}\lesssim \|h\|_{H^2_{\frac{n}{2}-1}}.
		\end{split}
		\end{equation}
		Here we used the fact that $|\nabla^{g_t}h|_{g_t}\lesssim |\nabla^{g_1}h|_{g_1}+|\nabla^{g_1}(g_t-g_1)|_{g_1}|h|_{g_1}$  to get the second inequality since again, $g_i\in B_{C^{2,\alpha}_{\tau}}(g_b,\varepsilon)$, $i=1,2$. This ends the proof of estimate \eqref{energy-est}.\\
			
			The proof of \eqref{energy-est-bis} is along the same lines as those of the proof of \eqref{energy-est}. Therefore, we only sketch its main steps. According to 	Proposition \ref{snd-var-gal-lambda}, the differential of the $L^{2}(e^{-f_g}d\mu_g)$-gradient of $\lambda_{\operatorname{ALE}}$ denoted by $\nabla^{L^2(e^{-f_g}d\mu_g)}\lambda_{\operatorname{ALE}}$ at a metric $g\in B_{C^{2,\alpha}_{\tau}}(g_b,\varepsilon)$ along a variation $h\in H^2_{\frac{n}{2}-1}$ is:
			\begin{equation}
			\begin{split}\label{for-diff-nabla-lam-wei}
			2D_g\left(\nabla^{L^2(e^{-f_g}d\mu_g)}\lambda_{\operatorname{ALE}}\right)(h)&=\Delta_{f_g}h+2\Rm(g)(h)-\Li_{B_{f_g}(h)}(g)\\
			&\quad+h\circ\Ric_{f_g}(g)+\Ric_{f_g}(g)\circ h-2\left(\frac{\tr_gh}{2}-\delta_gf(h)\right)\Ric_{f_g}(g).
			\end{split}
			\end{equation}
			As explained in (\ref{l2-grad-lam-gb}) at the beginning of the proof of \eqref{energy-est}, dealing either with the $L^{2}(e^{-f_g}d\mu_g)$-gradient or the $L^{2}(d\mu_{g_b})$-gradient of $\lambda_{\operatorname{ALE}}$ lead to the same expected estimate. Linearizing (\ref{for-diff-nabla-lam-wei}) applied to $g_2\in B_{C^{2,\alpha}_{\tau}}(g_b,\varepsilon)$ at a metric $g_1\in B_{C^{2,\alpha}_{\tau}}(g_b,\varepsilon)$ leads to \eqref{energy-est-bis}: indeed, the only difficulty consists in estimating the terms involving $\delta_gf(h)$ or equivalently, the volume variation $\frac{\tr_gh}{2}-\delta_gf(h).$ This is done with the help of Proposition \ref{var-vol-var-ell-eqn-prop} by linearizing \eqref{var-vol-var-ell-eqn-for} applied to $g_2$ at $g_1$ together with the use of Proposition \ref{prop-ene-pot-fct-bis}.
	\end{proof}

	We are in a good position to prove the main result of this section:
	\begin{theo}[A \L{}ojasiewicz inequality for ALE metrics]\label{theo-loja-ALE}
		Let $(N^n,g_b)$, $n\geqslant 4$, be a Ricci-flat ALE metric. Let $\tau\in \left(\frac{n-2}{2},n-2\right)$ and $\alpha\in(0,1)$. Then the functional $\lambda_{\operatorname{ALE}}$ satisfies the following $L^2_{\frac{n}{2}+1}$-\L{}ojasiewicz inequality: there exists $\varepsilon>0$, a constant $C>0$ and $\theta\in(0,1)$ such that for all $g\in B_{C^{2,\alpha}_{\tau}}(g_b,\varepsilon)$:
		\begin{equation}
		   |\lambda_{\operatorname{ALE}}(g)|^{2-\theta}\leq C\|\Ric(g) + \nabla^{g,2}{f_g}\|^2_{L^2_{\frac{n}{2}+1}}. \label{Lojasiewicz lambda ALE l2n/2+1}
		\end{equation} 
		
	\end{theo}
	\begin{proof}
		It suffices to prove that the functional $G:=\lambda_{\operatorname{ALE}}$ and its derivatives satisfy the assumptions of Proposition \ref{Lojasiewicz ineq weighted}.
		
		Let us first note that $\lambda_{\operatorname{ALE}}$ is analytic by Proposition \ref{lambdaALE analytic}.
		
		Now, the energy estimates (\ref{lip-bd-nabla-G}) hold true thanks to [\eqref{energy-est}, Proposition \ref{prop-energy-est}]. 
		Moreover, according to [(\ref{first-var-prop}), Proposition \ref{first-var-lambda}] and Proposition \ref{lambdaALE analytic}, the gradient of $\lambda_{\operatorname{ALE}}$ in $L^2(e^{-f_g}d\mu_g)$ is $-(\Ric(g) + \nabla^{g,2}f_g)$ for $g\in B_{C^{2,\alpha}_{\tau}}(g_b,\varepsilon)$. Condition [(\ref{item-0-bis}), Proposition \ref{Lojasiewicz ineq weighted}] is ensured by [\eqref{energy-est-bis}, Proposition \ref{prop-energy-est}].
		
		Next, observe that by Proposition \ref{scaling diffeo tildelambda}, the inequality \eqref{Lojasiewicz lambda ALE l2n/2+1} is invariant by diffeomorphisms of $N$ in the connected component of the identity such that their generating vector field lies in a neighborhood of $0_{TN}$ in $C^{3,\alpha}_{\tau-1}(TN)$. For this reason, we invoke Proposition \ref{prop-gauge-div-free} to restrict our space of metrics to the space $B_{C^{2,\alpha}_{\tau}}(g_b,\varepsilon)\,\cap\,E$ where $E:= \ker_{C^{2,\alpha}_{\tau}}(\div_{g_b})$ and we restrict the image to $C^{0,\alpha}_{\tau+2}\,\cap \,F$ with $F:=\div^*_{g_b}(C^\infty_c(TN))^\perp$ as in Lemma \ref{definition E et F pour loja}.  This space of metrics is the space of divergence-free metrics with respect to $g_b$. It is crucial to gauge the diffeomorphism invariance of $\lambda_{\operatorname{ALE}}$ away to expect the Fredholmness of its Hessian. 
		
		We will denote $\lambda_{\operatorname{ALE}}^E$ the restriction of $\lambda_{\operatorname{ALE}}$ to $E$ which is also analytic. 
		
		Since $\|\nabla \lambda_{\operatorname{ALE}}^E\|_{L^2_{\frac{n}{2}+1}}\leq \|\nabla \lambda_{\operatorname{ALE}}\|_{L^2_{\frac{n}{2}+1}}$, the inequalities go in the right direction, and it is enough to prove the corresponding \L ojasiewicz inequality for $\lambda_{\operatorname{ALE}}^E$.
		
		Notice that the linearization of $\nabla \lambda_{\operatorname{ALE}}^E$ at $0\in S^2T^*N$ is half the Lichnerowicz operator $\frac{1}{2}L_{g_b}$ by Proposition \ref{second-var-prop}. Proposition \ref{prop-lic-fred} ensures that $L_{g_b}$ is symmetric and is a bounded operator from $C^{2,\alpha}_{\tau}(S^2T^*N)$ to $C^{0,\alpha}_{\tau+2}(S^2T^*N)$. It is moreover bounded from $H^2_{\frac{n}{2}-1}(S^2T^*N)$ to $L^2_{\frac{n}{2}+1}(S^2T^*N)$. The Fredholmness of $L_{g_b}$ follows from Proposition \ref{prop-lic-fred}. Conditions (\ref{item-2}) and (\ref{item-5}) are met thanks to Proposition \ref{prop-lic-fred}. This ends the proof of Theorem \ref{theo-loja-ALE}.

	\end{proof}

	Let us finally prove that we obtain the optimal exponent in the integrable situation.
	
	Let $g_b$ be an integrable Ricci-flat ALE metric. In this situation, the idea is to replace the kernel $\ker_{L^2}L_{g_b}$ by the actual zero-set of $\lambda_{\operatorname{ALE}}$ among $C^{2,\alpha}_\tau$ divergence-free perturbations of $g_b$ which is an analytic manifold whose tangent space at $g_b$ is $\ker_{L^2}L_{g_b}$.
	
	\begin{theo}\label{theo-loja-int-opt}
		Let $(N^n,g_b)$, $n\geq 4$, be a Ricci-flat ALE metric whose infinitesimal Ricci-flat deformations are integrable. Let $\tau\in(\frac{n-2}{2},n-2)$ and $\alpha\in(0,1)$.
		
		Then there exists $\epsilon>0$ and a constant $C>0$ such that for all $g\in B_{C^{2,\alpha}_{\tau}}(g_b,\varepsilon)$, the following $L^2_{\frac{n}{2}+1}$-\L{}ojasiewicz inequality holds: 
		\begin{equation*}
		|\lambda_{\operatorname{ALE}}(g)|\leq C \|\Ric(g)+\nabla^{g,2}f_g\|_{L^2_{\frac{n}{2}+1}}^{2}. 
		\end{equation*}
		
		In particular, if $n\geq 5$ and $\tau\in(\frac{n}{2},n-2)$ then for any $0<\delta<\frac{2\tau-(n-2)}{2\tau-(n-4)}$, there exists $C>0$ such that for all $g\in B_{C^{2,\alpha}_\tau}(g_b,\epsilon)$, we have the following $L^2$-\L{}ojasiewicz inequality:
		$$ |\lambda_{\operatorname{ALE}}(g)|^{2-\theta}\leq C \|\nabla \lambda_{\operatorname{ALE}}(g)\|_{L^2}^{2}, \quad\theta:=2-\frac{1}{\delta}.$$
	\end{theo}
	
	\begin{proof}
		Denote $W_{g_b}:= \{\bar{g}_v$, \text{ for small } $v\in\ker_{L^2}L_{g_b}\}$, where the $ \bar{g}_v $ are the metrics of Definition \ref{definition integrable}. The functional $\lambda_{\operatorname{ALE}}$ and its gradient vanish on $W_{g_b}$ since the  metrics $\bar{g}_v$ are Ricci-flat, and the \L{}ojasiewicz inequality trivially follows on this space.
		
		Let $g$ be $C^{2,\alpha}_\tau$-close enough to $g_b$. Then, by Proposition \ref{gauge fixing ALE integrable} there exists a unique $\bar{g}_v\in W_{g_b}$ and a diffeomorphism $\phi:N\to N$ in the connected component of the identity such that its infinitesimal generator belongs to $C^{3,\alpha}_{\tau-1}(TN)$ and such that 
		\begin{itemize}
			\item $ \phi^*g - \bar{g}_v \perp_{L^2(\bar{g}_v)}\ker_{L^2(\bar{g}_v)}L_{\bar{g}_v} $, and
			\item $\div_{\bar{g}_v}\phi^*g = 0$.
		\end{itemize}
		Therefore, up to changing the reference Ricci-flat metric by a metric in $W_{g_b}$, since $\lambda_{\operatorname{ALE}}$ and the $L^2$-norm of its gradient are invariant by the pull-back by $\phi$ thanks to Proposition \ref{scaling diffeo tildelambda}, the situation is reduced to proving a \L{}ojasiewicz inequality on the orthogonal of the kernel, which is exactly the statement of Lemma \ref{loja orth noyau 1}.
	\end{proof}

	\section{Deformation of Ricci-flat ALE metrics, scalar curvature and mass}\label{sec-covid-mass}
	
	Let us now mention some applications of the functional ${\lambda}_{ALE}$ that we introduced and its properties. 
	
	\subsection{Local positive mass theorems}~~\\
	
	Let us mention some direct applications of the functional $\lambda_{\operatorname{ALE}}$ and its properties for Ricci-flat ALE metrics.
	
	\begin{coro}\label{local positive mass}
		Let $(N^n,g_b)$ be a Ricci-flat ALE metric which is either \emph{integrable} and \emph{stable} or a local maximizer of ${\lambda}_{ALE}$ with respect to the $C^{2,\alpha}_\tau$-topology. Then any deformation of $g_b$ small enough in $C^{2,\alpha}_\tau$ which has nonnegative and integrable scalar curvature in $C^0_{\tau'}$, for some $\tau'>n$, satisfies
		$$m_{\operatorname{ADM}}(g)\geq 0,$$
		with equality on Ricci-flat metrics only.
	\end{coro}
	\begin{proof}
		Consider $(N^n,g_b)$ a stable Ricci-flat ALE. By Proposition \ref{local maximum stable integrable}, it is a local maximum for ${\lambda}_{\operatorname{ALE}}$ in the $C^{2,\alpha}_\tau$-topology and therefore, for any metric $g$ sufficiently $C^{2,\alpha}_\tau$-close to $g_b$, we have $\lambda_{\operatorname{ALE}}^0(g)\leq 0$ with equality only if the metric is Ricci-flat. Since $\R_g\geq 0$, if $\R_g\in L^1$ we have $\lambda_{\operatorname{ALE}}^0(g)\geq 0$. This implies that the mass is nonnegative : $$0\geq\lambda_{\operatorname{ALE}}(g) = \lambda_{\operatorname{ALE}}^0(g)-m_{\operatorname{ADM}}(g)\geq -m_{\operatorname{ADM}}(g).$$
		Moreover, if we have equality, we necessarily have $\lambda_{\operatorname{ALE}}(g)= 0$, $\lambda_{\operatorname{ALE}}^0(g)= 0$ and $m_{\operatorname{ADM}}(g)=0$ and since the only maximizers of $\lambda_{\operatorname{ALE}}$ are Ricci-flat, $g$ has to be Ricci-flat.
	\end{proof}
	
	\begin{rk}
		Corollary \ref{local positive mass} is not a consequence of a previously known positive mass theorem. Actually, the positive mass theorem is known to be false in the ALE context \cite{LeB-Counter-Mass}.
	\end{rk}
	
	Finally, notice as in \cite{Hal-Has-Sie} that there are counter-examples to the rigidity part of the positive mass theorem among self-similar solutions to the Ricci flow. Indeed, Feldman, Ilmanen and Knopf \cite{Fel-Ilm-Kno} have constructed complete expanding gradient K\"ahler-Ricci solitons on the total space of the tautological line bundles $L^{-k}$, $k>n$ over $\mathbb{CP}^{n-1}$. These solutions on $L^{-k}$ are $U(n)$-invariant and asymptotic to the cone $C(\mathbb{S}^{2n-1}/\mathbb{Z}_k)$ endowed with the Euclidean metric $\frac{1}{2}i\partial\overline{\partial}\, |\cdot|^2$, where $\mathbb{Z}_k$ acts on $\mathbb{C}^n$ diagonally. The curvature tensor of these solitons decay exponentially fast to $0$ at infinity, in particular these metrics are ALE and their mass vanish. On the other hand, the scalar curvature of these metrics is positive everywhere. 
	
	\begin{rk}
		We can define and control $\lambda_{\operatorname{ALE}}^0$ on the example of Feldman-Ilmanen-Knopf, denoted by $g_{\operatorname{FIK}}$. Since for any $s>0$, we have $\lambda_{\operatorname{ALE}}^0(sg) = s^{\frac{n}{2}- 1}\lambda_{\operatorname{ALE}}^0(g)$ by Lemma \ref{scaling lambdaALE}, and therefore since their example is an expanding soliton, we can prove that the Ricci flow starting at $g_{\operatorname{FIK}}$ satisfies $\lambda_{\operatorname{ALE}}^0(g_{\operatorname{FIK}}(t)) = (1+ct)^{\frac{n}{2}- 1}\lambda_{\operatorname{ALE}}^0(g_{\operatorname{FIK}})>0$ for some $c>0$. This is in contrast with the compact situation where a Ricci-flow starting at a metric with positive $\lambda$-functional necessarily develops a finite-time singularity.
	\end{rk}
	
	\subsection{Global properties on spin manifolds}~~\\
	
	On spin $4$-manifolds, the stability of Ricci-flat ALE metrics as maximizers of ${\lambda}_{ALE}$ is ensured globally.
	\begin{prop}\label{prop-spin-def-local}
		Let $(N^4,g)$ be an ALE metric of order $\tau>1 = \frac{4-2}{2}$ on a spin manifold asymptotic to $\mathbb{R}^4\slash\Gamma$ for $\Gamma\subset SU(2)$.
		Assume the scalar curvature $\R_g$ is integrable and non-negative. Then, we have $$\lambda_{\operatorname{ALE}}(g)\leq 0,$$
		with equality if and only if $(N^4,g)$ is a hyperk\"ahler (Ricci-flat) ALE metric.
		
	\end{prop}
	\begin{proof}
		First of all, Lemma \ref{scaling lambdaALE} ensures the finiteness of $\lambda_{\operatorname{ALE}}(g)$ under such assumptions on the scalar curvature. Witten's formula \cite{wit} for the mass on spin asymptotically Euclidean manifolds which was extended by Nakajima \cite{nak} to spin ALE metrics with group in $SU(2)$, states that there exists $\psi$, a Dirac spinor asymptotic to a constant spinor with unit-norm for which 
		\begin{equation}
		m_{\operatorname{ADM}}(g)=\int_N(4|\nabla^g \psi|_g^2+\R_g|\psi|_g^2)\,d\mu_g\label{witten formula}.
		\end{equation}
		Using $w=|\psi|_g$ as a test function and Kato's inequality $|\nabla^g \psi|_g\geq |\nabla^g |\psi|_g|_g$, we find a lower bound $$m_{\operatorname{ADM}}(g)\geq \lambda_{\operatorname{ALE}}^0(g)$$ 
		similarly to \cite{Has-Per-Fct}. 
		
		Now, according to \cite[(3.9)]{Cal-Gau-Her}, Dirac spinors satisfy a pointwise \emph{improved} Kato inequality, and we have $\sqrt{1-\frac{1}{4}}|\nabla^g \psi|_g\geqslant|\nabla^g |\psi|_g|_g$. The equality $m_{\operatorname{ADM}}(g)= \lambda_{\operatorname{ALE}}^0(g)$ therefore implies that the spinor is parallel. Then one conclude that the ALE metric $g$ under consideration is hyperk\"ahler: see \cite[Proof of Theorem 3.3]{nak} for a proof of this fact.
		
	\end{proof}

	\begin{rk}
		We therefore recover that Ricci-flat ALE metrics on spin manifolds are \emph{stable} (they are actually hyperkähler by \cite{nak}). Recall that it is a folklore conjecture (see \cite[Section $1$, $3)$]{Ban-Kas-Nak} for instance) that all $4$-dimensional simply connected ALE Ricci-flat metrics are hyperk\"ahler. In light of the results of this paper, it is tempting to study the stability of $4$-dimensional Ricci-flat ALE metrics as a first step towards the previous conjecture.	
		\end{rk}
	
	\newpage
	\appendix

	\section{Real analytic maps between Banach spaces}\label{app-A}

	Let us recall some basic definitions and theorems about real analytic maps between Banach spaces and real analytic submanifolds.
	
	Let $ a_k : V^k\to W$, $k\in \mathbb{N}$ be a symmetric $k$-linear form on $V$ taking values in $W$. The \emph{power series} in $x$ denoted $x\in V \mapsto\sum_k a_k x^k\in W$ from the Banach space $V$ with values in the Banach space $W$ is defined as the sum of the $a_k(x,...,x)$. We will say that it is converging if the real sum $\sum_k \|a_k(x,...,x)\|_{W}$ converges.
	
	\begin{defn}[Real analytic map]\label{def-analytic}
		Let $V$ and $W$ be Banach spaces, and $U$ an open subset of $V$. A map $f: V\to W$ is \emph{real analytic} if for each point $x\in U$, $f$ is equal to a power series converging in a neighborhood of $x$.
	\end{defn}
	
	This definition is adapted to the techniques of analysis, and in particular, thanks to \cite[Implicit Function Theorem, p. 1081]{Whi-Ana-Fct}, the implicit function theorem holds for this regularity. 
	
	\begin{lemma}[{\cite[Implicit Function Theorem, p. 1081]{Whi-Ana-Fct}}]\label{th fcts implicites}
		Let $X$, $Y$ and $Z$ be Banach spaces, and $A$ an open subset of $X\times Y$. Assume that a real analytic map $f : A\to Z$ satisfies for $(x_0,y_0)\in A$,
		\begin{enumerate}
			\item $f(x_0,y_0)=0,$
			\item $d_{(x_0,y_0)}f(0,.) : Y\mapsto Z$ is a topological isomorphism.
		\end{enumerate}
		Then,
		\begin{enumerate}
			\item there exists $N(x_0)$, an open neighborhood of $x_0$ in $X$ and a unique continuous map $g : N(x_0)\to Y$ such that $g(x_0)= y_0$, $(x,g(x))\in A$ for $x\in N(x_0)$, and
			$$f(x,g(x))=0,$$
			\item $g$ is real analytic,
			\item $d_xg = -\big(d_{(x,g(x))}f(0,.)\big)^{-1}\circ \big(d_{(x,g(x))} f (.,0)\big)$.
		\end{enumerate}
	\end{lemma}
	
	It will be crucial for us to note that usual operations on tensors are real analytic between H\"older spaces whose regularity is higher than the number of derivatives involved. This can be proved by adapting the proof of \cite[Lemma 13.7]{Koi-Ein-Met} which deals with $H^s$-regularity for $s>0$ large enough to the H\"older context thanks to the theory developped in \cite{Pal-Found}. This also holds between weighted H\"older spaces as long as the weight makes the multiplication of two functions continuous.
	
	\begin{lemma}\label{opérations analytiques}
		The operations of derivation, multiplication, contraction and integration are analytic between Hölder or between weighted Hölder spaces in which the multiplication is continuous. That is, denoting $C_\tau^{k,\alpha}$ one of the spaces of the present article, as long as there exists $C>0$ such that for any two functions $f,g\in C^{k,\alpha}$, we have
		$$\|fg\|_{C_\tau^{k,\alpha}}\leq C\|f\|_{C_\tau^{k,\alpha}}\|g\|_{C_\tau^{k,\alpha}}.$$
		One can check that with the spaces of Definition \ref{def-weighted-sobolev-norms}, this is satisfied as long as $\tau>0$. By construction, the derivation is also continuous.
	\end{lemma}
	
	\section{Divergence-free gauging of asymptotically conical spaces}\label{app-B}
	
	As usual in geometrically covariant problems, one needs to fix a gauge transverse to the action of the diffeomorphism group to use elliptic theory in order to study the equation $\Ric(g)=0$. In this particular situation, a natural gauge is the divergence-free gauge because of the simplification of the expression of the second variation of the $\lambda_{\operatorname{ALE}}^0$ functional, see Remark \ref{remark jauge div}. We will prove with the help of the inverse function theorem that in a suitable neighborhood of an asymptotically conical Ricci-flat metric, it is possible to construct divergence-free gauges. That is, given $(N^n,g_b)$ a fixed Ricci-flat asymptotically conical (AC) metric, for any metric $g$ close enough to $g_b$, there exists a diffeomorphism $\phi : N\to N$ such that 
	\begin{equation}
	\div_{g_b}\phi^*g =0.\label{eq jauge divergence}
	\end{equation}
	
	We will actually work in the larger class of asymptotically conical metrics in this appendix as most of the analysis is exactly the same.
	
	\begin{defn}[Asymptotically conical manifolds]
		We will call a Riemannian manifold $(N^n,g)$ \emph{asymptotically conical} (AC) of order $\tau>0$ if the following holds : there exists a compact set $K\subset N$, a radius $R>0$, $(S,g_S)$ a closed Riemannian manifold and a smooth diffeomorphism $\Phi : S\times [R,+\infty)\mapsto N\backslash K$ such that, denoting $g_{C(S)}:=dr^2+r^2g_S$, we have, for all $k\in \mathbb{N}$,
		$$ \rho^k\big|\nabla_{g_{C(S)}}^{k}(\Phi^*g-g_{C(S)})\big|_e = O(\rho^{-\tau}),$$
		on $S\times [R,+\infty)$, where $\rho =\max\{1, d_{C(S)}(.,0)\}$.
	\end{defn}
	
	The main result of this section is the following.
	\begin{prop}\label{prop-gauge-div-free}
		Let $(N^n,g_b)$ be a non-flat Ricci-flat manifold asymptotic to a smooth Ricci-flat cone $(C(S^{n-1}),g_S)$.
		
		Then, for all $k\in \mathbb{N}^*$, $\alpha\in(0,1)$ and for all $1<\beta\leq n-1$, there exists $\epsilon>0$ such that for any metric $g\in B_{C^{k,\alpha}_{\beta}}(g_b,\varepsilon)$ such that
		\begin{equation*}
		\|g-g_b\|_{C^{k,\alpha}_{\beta}}\leq \epsilon,
		\end{equation*}
		there exists a vector field $X\in C^{k+1,\alpha}_{\beta-1}(TN)$ satisfying, 
		\begin{equation*}
		\div_{g_b}((\exp^{g_b}_X)^*g)=0,
		\end{equation*}
		where $\exp_X$ is the diffeomorphism $\exp^{g_b}_X : x\in N \mapsto \exp_x^{g_b}(X(x))$. Moreover, there exists $C = C(g_b,k,\alpha,\beta)>0$ such that we have
		$$\|(\exp^{g_b}_X)^*g-g\|_{C^{k,\alpha}_{\beta}}\leq C\|X\|_{C^{k+1,\alpha}_{\beta-1}}\leq C^2 \|\div_{g_b}g\|_{C^{k-1,\alpha}_{\beta+1}},$$
		and we can choose $X$ depending analytically on $g$. 
		
		If $(N^n,g_b)$ is ALE, that is if $(C(S^{n-1}),g_S) = (\mathbb{R}^n\slash\Gamma,g_e)$ for $\Gamma\subset SO(n)$ acting freely on $\mathbb{S}^{n-1}$, then, the result holds for $1<\beta< n$. Moreover, the above vector field $X$ is unique and depends analytically on $g$.
	\end{prop}
	\begin{rk}
		The result holds on Euclidean space $\RR^n$ for $1<\beta< n-1$.
	\end{rk}
	\begin{rk}
		A completely analogous proof would give the same result for the Bianchi gauge. However, we cannot do the same to cancel out $\partial_i h_{ij} - \partial_j h_{ii} = \div_{g_b}(g) - \nabla^{g_b} \tr_{g_b}(g)$ thanks to the action of a diffeomorphism since the linearized equation is $d^*d$ which is not elliptic. This is consistent with the fact that the mass and the integrability of the curvature are invariant by diffeomorphisms.
	\end{rk}
	\begin{rk}	
		The assumption $\beta>1$ is not linked to any Fredholm property of the operator we will look at. It is however important to ensure for instance that the product of two functions $$(u,v)\in C^{2,\alpha}_{\beta-1}\times C^{2,\alpha}_{\beta-1} \mapsto u.v\in C^{2,\alpha}_{\beta-1},$$ is continuous. This fact is necessary to prove that our operators are analytic.
	\end{rk}
	We state the following linear version of Proposition \ref{prop-gauge-div-free} whose proof is along the same lines of that of Proposition \ref{prop-gauge-div-free}:
	\begin{prop}\label{prop-decomp-2-tensor}
		Let $(N^n,g_b)$ be a non-flat Ricci-flat manifold asymptotic to a smooth Ricci-flat cone $(C(S^{n-1}),g_S)$. Let $\beta\in (1,n-1)$ and $\alpha\in(0,1)$. Then  the following decomposition holds true:
\begin{equation}
C^{2,\alpha}_{\beta}(S^2T^*N)=\ker_{C^{2,\alpha}_{\beta}}\div_{g_b}\oplus\ima \div^{\ast}_{g_b}|_{C^{3,\alpha}_{\beta-1}}.
\end{equation}
Equivalently, for any symmetric $2$-tensor $h$ in $C^{2,\alpha}_{\beta}$, there exist a unique divergence-free symmetric $2$-tensor $h'$ in $C^{2,\alpha}_{\beta}$ and a unique vector field $X$ on $N$ in $C^{3,\alpha}_{\beta-1}$ such that $h=h'+\Li_X(g)$. Moreover this decomposition is $L^2$-orthogonal if $\beta>\frac{n}{2}$.
\end{prop}
	
	\subsection{Invertibility of the linearization}\label{inv-lin-section}
	
	Our approach based on the implicit function theorem starts with the study of 
	\begin{equation}
	\div_{g_b}\Li_{X}(g_b) = - \div_{g_b} h,\label{eq linearisation divergence}
	\end{equation}
	where $X$ is a vector field, and $h$ is a symmetric $2$-tensor. Notice that equation (\ref{eq linearisation divergence}) is the linearization of equation \eqref{eq jauge divergence}. The operator $\Box$, defined as
	\begin{equation}
	\Box_{g_b}(X):=\div_{g_b}\Li_{X}(g_b)=\nabla^{g_b}(\div_{g_b}X)+\Delta_{g_b}X, 
	\end{equation}
	is elliptic and self-adjoint.\\
	
	Let us discuss about the kernel of $\Box$ on a Ricci-flat cone.\\
	
	Let $n\geq 4$, and $(C(S),g_{C(S)}):=(\mathbb{R}_{+}\times S,dr^2+r^2g_S)$ be an $n$-dimensional Ricci-flat cone with smooth link $(S,g_S)$, satisfying $\Ric(g_S)=(n-2)g_S$. Let us follow \cite[Section 2]{Che-Tian-Ric-Fla}, and \cite[Section 4.1]{Ache-Via} and consider the operator as acting on $1$-forms rather than on vector fields on $C(S)$. By harmonic decomposition, any such vector field can be decomposed as an infinite sum of terms of the following types :
	\begin{enumerate}
		\item $p(r)\psi$, where the $1$-form $\psi$ satisfies $d^*_S\psi = 0$, and $d^*_Sd_S\psi = \mu \psi$,
		\item $r^{-1}l(r)\phi dr + u(r)r d_S \phi$, where the function $\phi$ satisfies $d^*_S\phi = 0$, et $d^*_Sd_S\phi = \lambda \phi$.
	\end{enumerate}
	Moreover, the operator $\Box$ preserves these two types of $1$-forms. 
	
	For $\lambda$ and $\mu$ eigenvalues of $d^*_Sd_S + d_Sd_S^*$ on functions and $1$-forms respectively, we set
	\begin{equation}
	a_{\mu}^\pm:= \frac{4-n}{2}\pm\sqrt{\frac{(4-n)^2}{4}+\mu},\label{alpha mu}
	\end{equation}
	and
	\begin{equation}
	b_{\lambda}^\pm:=\frac{2-n}{2} \pm\sqrt{\frac{(2-n)^2}{4}+\lambda}.\label{beta lambda}
	\end{equation}
	\begin{rk}
		In dimension $n=4$, we get 
		$$ a_{\mu}^\pm:= \pm\sqrt{\mu}, $$
		and
		$$ b_{\lambda}^\pm:=-1 \pm\sqrt{1+\lambda}.$$
	\end{rk}
	
	Every $1$-form in the kernel of $\Box$ on $C(S)$ is a sum of the following types of homogeneous $1$-forms :
	\begin{enumerate}
		\item $r^{a^\pm_\mu}\psi$,
		\item \begin{itemize}
			\item $r^{b^{\pm}_\lambda}d_S\phi \;+\; b^\pm_\lambda r^{b^\pm_\lambda-1}\phi dr$,
			\item $2r^{b^{\pm}_\lambda+2}d_S\phi \;+\; b^\mp_\lambda r^{b^\pm_\lambda+1}\phi dr$.
		\end{itemize}
	\end{enumerate}
	We call \emph{exceptional values} for $\Box$ the possible homogeneity rates, that is $a^\pm_\mu-1$, $b^\pm_\lambda-1$ or $b^\pm_\lambda+1$ corresponding to eigenvalues of the Hodge Laplacian on functions or $1$-forms. We will need to estimate these exceptional values to use analysis in weighted Hölder spaces.
	
	\begin{rk}\label{rem poids crit sphere}
		For $(S,g_S) = (\mathbb{S}^{n-1},g_{\mathbb{S}^{n-1}})$, we have the following values for the first exceptional values associated to the $j$-th eigenvalues of the Hodge Laplacian on functions and $1$-forms respectively $\lambda_j(g_{\mathbb{S}^{n-1}})$ and $\mu_j(g_{\mathbb{S}^{n-1}})$.
		\begin{enumerate}
			\item $a_{\mu_j(g_{\mathbb{S}^{n-1}})}^+-1 = j$, $j\in \mathbb{N}^*$,
			\item $a_{\mu_j(g_{\mathbb{S}^{n-1}})}^--1 = -(n-2)-j$, $j\in \mathbb{N}^*$,
			\item $b_{\lambda_j(g_{\mathbb{S}^{n-1}})}^++1 =1+j$, $j\in \mathbb{N}$,
			\item $b_{\lambda_j(g_{\mathbb{S}^{n-1}})}^-+1 =-(n-3)-j$, $j\in \mathbb{N}$,
			\item $b_{\lambda_j(g_{\mathbb{S}^{n-1}})}^+-1 =-1+j$, $j\in \mathbb{N}$,
			\item $b_{\lambda_j(g_{\mathbb{S}^{n-1}})}^--1 =-(n-1)-j$, $j\in \mathbb{N}$.
		\end{enumerate}

	\end{rk}
	
	More generally, let us discuss the eigenvalues of Einstein manifolds with positive scalar curvature.\\
	
	For Einstein manifolds with positive scalar curvature, the celebrated Lichnerowicz-Obata theorem provides a lower bound for the first eigenvalue of the Hodge Laplacian on functions and $1$-forms as well as a rigidity statement. See \cite{Gal-Mey-Spec} for an exposition as well as a similar result for $p$-forms. The result is a consequence of Bochner formulas (and Weitzenbock formulas for $p$-forms).
	
	\begin{lemma}[Lower bounds on the first eigenvalues]\label{borne inf vp}
		If a Riemannian manifold $(M^n,g)$ satisfies the lower bound $\Ric(g)\geq (n-1)g$, then, $\lambda_1(g)$ and $\mu_1(g)$, which are respectively the first eigenvalue of the Hodge laplacian on functions, and on $1$-forms satisfy
		\begin{equation}
		\lambda_1(g)\geq \lambda_1(g_{\mathbb{S}^n}),
		\end{equation}
		with equality if and only if $(M,g) = (\mathbb{S}^n,g_{\mathbb{S}^n})$, and 
		\begin{equation}
		\mu_1(g)\geq \mu_1(g_{\mathbb{S}^n}),
		\end{equation}
		with equality if and only if $(M,g) = (\mathbb{S}^n,g_{\mathbb{S}^n})$.
	\end{lemma}
	
	We therefore get an a priori bound on the eigenvalues of Einstein manifolds with positive scalar curvature which translates directly to a control for the exceptional values of our operator $\Box$.

	\begin{prop}\label{bornes exception quotients}
		Let $(S^{n-1},g_S)$ satisfy $\Ric(g_S)= (n-2)g_S$ but $(S^{n-1},g_S)\neq (\mathbb{S}^{n-1},g_{\mathbb{S}^{n-1}})$, then there is no exceptional value in $[-(n-2),0]$. 
		
		Moreover, if the following inequality holds
		$$\lambda_1(g_S)\geq \lambda_2(g_{\mathbb{S}^{n-1}}),$$
		where $\lambda_i$ is the $i$-th eigenvalue (without multiplicity) of the Hodge Laplacian on functions. Then, on the Ricci-flat cone $(C(S),g_{C(S)})$ with link $(S,g_S)$, there is no exceptional value for the operator $\Box$ in $(-(n-1),1)$. This is in particular the case for cones over nontrivial quotients of the sphere in all dimensions.\\
	\end{prop}
	\begin{rk}
		For $(\mathbb{S}^{n-1},g_{\mathbb{S}^{n-1}})$, there is no exceptional values in $(-(n-2),0)$.
	\end{rk}
	\begin{proof}
		By the description of the kernel of $\Box$ on $C(S)$ at the beginning of Section \ref{inv-lin-section}, we know that the exceptional values are of one of the forms $a^\pm_\mu-1$, $b^\pm_\lambda-1$, $b^\pm_\lambda+1$.
		For the sphere, the exceptional values corresponding to the lowest eigenvalues are detailed in Remark \ref{rem poids crit sphere}. Assuming that $\lambda_1(g_S)\geq \lambda_2(g_{\mathbb{S}^{n-1}}),$ we see from the expressions \eqref{alpha mu} and \eqref{beta lambda} that the only possible exceptional values in $\big(-(n-1),1\big)$ are $b^+_{\lambda_0(g_S)}-1= -1$ and $b^{-}_{\lambda_0(g_S)}+1 = -(n-3)$, but as noted in \cite[Lemma 3.2]{ozu2}, coming back to the expressions of the associated $1$-forms, we find that they vanish.
		\\
		
		If we do not have the bound $\lambda_1(g_S)\geq \lambda_2(g_{\mathbb{S}^{n-1}})$, we still have the estimates of Lemma \ref{borne inf vp} which plugged in the expressions \eqref{alpha mu} and \eqref{beta lambda} imply that there are exceptional values in $(0,1]$ and $[-(n-1),-(n-2))$, but none in $[-(n-2),0]$
		\\
		
		In dimension $3$, all Einstein manifolds with positive scalar curvature are of the form $\mathbb{S}^3\slash\Gamma$, where $\Gamma$ is a finite subgroup of $SO(4)$ acting freely on the sphere. If $\Gamma\neq \{e\}$, then the eigenfunctions and $1$-forms of $\mathbb{S}^3\slash\Gamma$ are the projection of the eigenfunctions and $1$-forms of $\mathbb{S}^3$ which are $\Gamma$-invariant. It turns out that the first eigenvalue on functions and $1$-forms of $\mathbb{S}^3$ are respectively associated to restrictions of non-zero linear functions and constant vector fields on $\mathbb{R}^4$, but since none of these are $\Gamma$-invariant if $\Gamma\neq \{e\}$. This implies the bound $\lambda_1(g_S)\geq \lambda_2(g_{\mathbb{S}^{n-1}})$.
		
		The same argument works for $\mathbb{S}^{n-1}\slash\Gamma$ for $\Gamma\subset SO(n)$, a finite subgroup acting freely on the sphere. In this case, the first exceptional value below $1$ is $-(n-1)$.
	\end{proof}
	
	
	We can now conclude that our operator is invertible in suitable H\" older spaces.
	\begin{prop}\label{inverse deltadelta}
		Let $k\in \mathbb{N}^*$, $\alpha\in(0,1)$, and $(N^n,g_b)$ be a Ricci-flat manifold asymptotic to a non-flat Ricci-flat cone $(C(S),g_{C(S)})$. Then, for all $1<\beta\leq n-1$, the operator $\Box : C^{k+2,\alpha}_{\beta-1}(TN)\to C^{k,\alpha}_{\beta+1}(TN)$ is surjective. 
		
		It is moreover bijective for $1<\beta<n$ if we assume that the link satisfies
		$$\lambda_1(g_S)\geq \lambda_2(g_{\mathbb{S}^{n-1}}).$$
	\end{prop}
	\begin{rk}
		For $\mathbb{R}^n$, the results holds for $1<\beta<n-1$.
	\end{rk}
	\begin{proof}
		Consider the operator $\Box : C^{k+2,\alpha}_{\beta-1}(TN)\to C^{k,\alpha}_{\beta+1}(TN)$. By the theory of weighted Hölder spaces,  for $\beta$ nonexceptional, the cokernel of the self-adjoint operator $\Box: C^{k+2,\alpha}_{\beta-1}(TN)\to C^{k,\alpha}_{\beta+1}(TN)$ is the kernel of $\Box$ on $C^{k,\alpha}_{n-1-\beta}(TN)$ (see Note \ref{note L2 cokernel}). Let $Y\in C^{k,\alpha}_{n-1-\beta}(TN)$, and assume that $\Box Y=0$. We can consider the following integration by parts
		\begin{equation*}
		\begin{split}
		0=&\int_N \langle\Box Y,Y\rangle_{g_b}\,d\mu_{g_b} \\ =& \int_N |\Li_Yg_b|_{g_b}^2\,d\mu_{g_b}  + \lim_{R\to \infty} \int_{\{\rho_{g_b}= R\}} (\Li_Yg_b)(\nu,Y)\,d\sigma_{g_b} ,
		\end{split}
		\end{equation*}
		where $\nu$ is the outward  unit normal to the hypersurface $\{\rho_{g_b}= R\}$ and where $d\sigma_{g_b}$ is the $(n-1)$-dimensional Hausdorff measure on $\{\rho_{g_b}= R\}$ induced by $d\mu_{g_b}$. Since we assumed that $Y\in C^{k,\alpha}_{n-1-\beta}(TN)$, we have $|\Li_Yg_b|_{g_b} = O(\rho^{-n+\beta})$ and $|Y|_{g_b} =  O(\rho^{-n+1+\beta})$ therefore, the boundary term vanishes for $-n + 2\beta <0 $, that is $ \beta<\frac{n}{2}$. This implies that for $ \beta<\frac{n}{2}$, we actually have $\Li_Yg_b=0$ and $Y$ is a Killing vector field. By the formula $\Box = \frac{1}{2}(\nabla^*\nabla - \Ric(g_b))$, we have $\nabla^*\nabla Y = 0,$ since $\Ric(g_b)=0$ and by integrating by parts again, we have 
		\begin{equation*}
		0=\int_N \langle\nabla^*\nabla Y ,Y\rangle_{g_b}\,d\mu_{g_b} = \int_N |\nabla^{g_b} Y|^2_{g_b}\,d\mu_{g_b} ,
		\end{equation*}
		since the boundary terms vanish because of our decay assumption.
		
		Consequently, $\nabla^{g_b} Y = 0$, hence $Y = 0$ because of the decay at infinity. The operator $\Box : C^{k+2,\alpha}_{\beta-1}\to C^{k,\alpha}_{\beta+1}$ is therefore surjective for $0<\beta<\frac{n}{2}$, and since the kernel and cokernel are constant between exceptional values, that is for $\beta-1\in \big(0,n-2\big]$ (and even $\beta-1\in \big(-1,n-1\big)$ if $\lambda_1(g_S)\geq \lambda_2(g_{\mathbb{S}^{n-1}})$), the operator is surjective for $0<\beta\leq n-1$ (and $0<\beta<n$ if $\lambda_1(g_S)\geq \lambda_2(g_{\mathbb{S}^{n-1}})$).
		
		Let us now prove that it is moreover injective under the assumption $\lambda_1(g_S)\geq \lambda_2(g_{\mathbb{S}^{n-1}})$. Consider $X\in C^{k+2,\alpha}_{\beta-1}(TN)$ satisfying $ \Box X=0$. Since there is no exceptional value between $-(n-1)$ and $1$ by Proposition \ref{bornes exception quotients}, $X$ actually lies in $C^{k+2,\alpha}_{n-1}(TN)$ and,
		\begin{align*}
		0= \int_N \langle\Box X,X\rangle_{g_b}\,d\mu_{g_b} =& \int_N |\Li_Xg_b|_{g_b}^2\,d\mu_{g_b} + \lim_{R\to \infty} \int_{\{\rho_{g_b}= R\}} (\Li_Xg_b)(\nu,X)\,d\sigma_{g_b},
		\end{align*}
		where $\nu$ is the outward unit normal to the hypersurface $\{\rho_{g_b}= R\}$. Since we assumed that $X\in C^{k,\alpha}_{n-1}(TN)$, we have $|\Li_Xg_b|_{g_b} = O(\rho_{g_b}^{-n})$ and $|X|_{g_b} =  O(\rho_{g_b}^{-(n-1)})$. Therefore, the boundary term vanishes, and by the same argument as above, $X= 0$.
	\end{proof}
	
	\subsection{A divergence-free gauge}
	
	We can now apply the implicit function theorem to solve our gauging problem.
	
	\begin{proof}[Proof of Proposition \ref{prop-gauge-div-free}]
		Let $(N^n,g_b)$ be a Ricci-flat manifold asymptotic to a smooth Ricci-flat cone $(C(S^{n-1}),g_S)$, and denote by $\mathcal{S}_b$ a complement of the kernel of $\Box$ in $C^{k+2,\alpha}_{\beta-1}(TN)$ ($\mathcal{S}_b= C^{k+2,\alpha}_{\beta-1}(TN)$ if we have the bound $\lambda_1(g_S)\geq \lambda_2(g_{\mathbb{S}^{n-1}})$).
		
		Let us define $\Phi :  \mathcal{U}^{k+1,\alpha}_{\beta}(g_b,\varepsilon)\times C^{k+2,\alpha}_{\beta-1}(TN)\cap\mathcal{S}_b\to C^{k,\alpha}_{\beta+1}(TN)$ for a metric $g$ and vector field $X$,
		$$\Phi(g,X):= \div_{(\exp_X^{g_b})_*g_b}g.$$
		The conclusion of our proposition is that for any $g$, there exists $X(g)$ such that $\Phi(g,X(g))= 0$ and $g\mapsto X(g)$ is analytic. Indeed, for any diffeomorphism $\phi:N\to N$, $\phi_*(\div_{g}\phi^* g) = \div_{\phi_*g}g$.
		
		We want to apply the implicit function theorem, Lemma \ref{th fcts implicites} to $\Phi$. Let us first note that by Lemma \ref{opérations analytiques}, which is a local result, $\Phi$ is analytic for $k$ larger than or equal to $1$. We have $\Phi(g_b,0)= 0$, and $d_{(g_b,0)}\Phi(0,\cdot)= \Box$ is an isomorphism by Proposition \ref{inverse deltadelta} (once restricted to the orthogonal of its kernel). The Proposition follows by the implicit function Theorem, Lemma \ref{th fcts implicites}.
	\end{proof}
	
	\section{First and second variations of geometric quantities}\label{app-C}
	In this appendix, we collect first and second derivatives of various geometric quantities.
	We start with a qualitative formula for the linearization of the Ricci curvature with respect to a background metric:
	\begin{lemma}\label{Ric-lin-lemma-app}
		Let $(N^n,g_1)$ be a Riemannian manifold. Let $g_2$ be a second Riemannian metric on $N$ uniformly equivalent to $g_1$, i.e. $C^{-1}\leq g_2\leq Cg_1$ for some positive constant $C$. Then, if $h:=g_2-g_1$ and if $g_t:=g_1+(t-1)h$, $t\in[1,2]$,
		\begin{equation}
		-2\Ric(g_t)=-2\Ric(g_1)+g_t^{-1}\ast\nabla^{g_1,2}g_t+g_t^{-1}\ast h\ast\Rm(g_1)+g_t^{-1}\ast g_t^{-1}\ast\nabla^{g_1}g_t\ast\nabla^{g_1}g_t-\Li_{B}(g_t),\label{ric-shi-estimate}
		\end{equation}
		where $B$ is a vector field on $N$ defined as follows:
		\begin{equation}
		\begin{split}
		B=B(g_t,\nabla^{g_1},g_t)&:=\frac{1}{2}g_t^{kl}\left(\nabla^{g_1}_k(g_t)_{il}+\nabla^{g_1}_l(g_t)_{ik}-\nabla^{g_1}(g_t)_{kl}\right)\\
		&=\div_{g_1}(g_t-g_1)-\frac{1}{2}\nabla^{g_1}\tr_{g_1}(g_t-g_1)+g_t^{-1}\ast(g_t-g_1)\ast\nabla^{g_1}g_t.\label{lin-Bian-app}
		\end{split}
		\end{equation}
		Moreover,
		\begin{eqnarray}
		\left|\frac{\partial}{\partial t}\Ric(g_t)\right|_{g_1}&\lesssim&\left(1+|h|_{g_1}\right)|\nabla^{g_1,2}h|_{g_1}+|\Rm(g_1)|_{g_1}|h|_{g_1}+\left(1+|h|_{g_1}\right)|\nabla^{g_1}h|_{g_1}^2,\label{first-der-Ric-rough}\\
		\left|\frac{\partial^2}{\partial t^2}\Ric(g_t)\right|_{g_1}&\lesssim& |h|_{g_1}|\nabla^{g_1,2}h|_{g_1}+|\Rm(g_1)|_{g_1}|h|^2_{g_1}+|h|_{g_1}|\nabla^{g_1}h|_{g_1}^2,\label{sec-der-Ric-rough}
		\end{eqnarray}
		where the sign $\lesssim$ means less than or equal up to a positive constant that is uniform in $t\in[1,2]$.
	\end{lemma}
	\begin{proof}
		For a proof of (\ref{ric-shi-estimate}) and (\ref{lin-Bian-app}), see for instance [Lemma $2.1$, Chapter $2$, \cite{Shi-Def}]. To get (\ref{first-der-Ric-rough}) and (\ref{sec-der-Ric-rough}) respectively, it suffices to differentiate (\ref{ric-shi-estimate}) and (\ref{lin-Bian-app}) once and twice respectively.
	\end{proof}
	Now, we recall the variations of the Ricci and scalar curvatures:
	\begin{lemma}\label{lem-lin-equ-Ric-first-var}
		Let $(N^n,g)$ be a Riemannian manifold. Let $h$ be a smooth symmetric $2$-tensor on $N$. Then,
		\begin{equation}
		\begin{split}
		\delta_{g}(-2\Ric)(h)&=\Delta_{g}h+2\Rm(g)(h)-\Ric(g)\circ h-h\circ \Ric(g)-\Li_{B_{g}(h)}\\
		&=L_gh-\Li_{B_{g}(h)},
		\end{split}
		\end{equation}
		where $L_g$ denotes the Lichnerowicz operator as defined in introduced in Definition \ref{defn-Lic-Op} and where $B_{g}(h)$ denotes the linearized Bianchi gauge defined by:
		\begin{equation}
		B_{g}(h):=\div_{g}h-\frac{1}{2}\nabla^{g}\tr_{g}h.\label{defn-bianchi-op}
		\end{equation}
		In particular, the first variation of the scalar curvature along a variation $h\in S^2T^*N$ is:
		\begin{equation}
		\begin{split}
		\delta_{g}\R(h)&=\div_{g}\div_{g}h-\Delta_{g}\tr_gh-\left<h,\Ric(g)\right>_g.\label{lem-lin-equ-scal-first-var}
		\end{split}
		\end{equation}
		
	\end{lemma}
	A proof of this lemma can be found for instance in \cite[Chapter $2$]{Cho-Boo}.
	
	We recall the following qualitative estimates whose proofs are essentially based on the definition of the Hessian of a function on a Riemannian manifold:
	\begin{lemma}\label{lemma-app-lie-der-lin}
		Let $(N^n,g_1)$ be a Riemannian manifold. Let $g_2$ be a second Riemannian metric on $N$ uniformly equivalent to $g_1$, i.e. $C^{-1}\leq g_2\leq Cg_1$ for some positive constant $C$. Denote $g_t:=g_1+(t-1)g_2$, $t\in[1,2]$ and assume there exists a one-parameter family $(f_{g_t})_{t\in[1,2]}$ of smooth functions on $N$ that varies smoothly in $t$. Then,
		\begin{equation}
		\frac{\partial}{\partial t}\nabla^{g_t,2}f_{g_t}=\nabla^{g_t,2}(\delta_{g_t}f(h))+\frac{1}{2}\Li_{\nabla^{g_t}f_{g_t}}(h)-\frac{1}{2}\Li_{h(\nabla^{g_t}f_{g_t})}(g_t),\label{first-var-lie-der-app}
		\end{equation}
		and,
		\begin{equation}
		\begin{split}
		\frac{\partial^2}{\partial t^2}\nabla^{g_t,2}f_{g_t}&=\nabla^{g_t,2}(\delta^2_{g_t}f(h,h))+g_t^{-1}\ast \nabla^{g_t}h\ast\nabla^{g_t}(\delta_{g_t}f(h))\\
		&+g_t^{-1}\ast g_t^{-1}\ast\nabla^{g_t}h\ast\nabla^{g_t}h\ast \nabla^{g_t}f_{g_t}+g_t^{-1}\ast g_t^{-1}\ast h\ast \nabla^{g_t}h\ast\nabla^{g_t}f_{g_t}.\label{sec-var-lie-der-app}
		\end{split}
		\end{equation}
	\end{lemma}
	

	\newpage

	\bibliographystyle{alpha.bst}
	\bibliography{bib-4d-RF}
	
\end{document}